 \documentclass[11pt,letterpaper]{amsart}

 \usepackage{txfonts} 

 \usepackage{color,xcolor}

 \usepackage[colorlinks,linkcolor=blue,citecolor=cyan]{hyperref}

 \theoremstyle{plain}
 
 \newcommand\dd{\,\mathrm{d}}

\newtheorem{theorem}{Theorem}[section]
\newtheorem{lemma}[theorem]{Lemma}

\theoremstyle{definition}
\newtheorem{definition}[theorem]{Definition}

\newtheorem{corollary}[theorem]{Corollary}
\newtheorem{proposition}[theorem]{Proposition}

\newtheorem{problem}[theorem]{Problem}

 \usepackage[nameinlink]{cleveref}

 \usepackage{amsmath,amsthm}
 
\theoremstyle{remark}
\newtheorem{remark}[theorem]{Remark}

\numberwithin{equation}{section}

\begin{document}
	
	\title[]{The structure of fully nonlinear  equations and its applications to prescribed problems  on  complete conformal metrics}

\author{Rirong Yuan}
\address{School of Mathematics, South China University of Technology, Guangzhou 510641, China}
\email{yuanrr@scut.edu.cn}


\date{}

\dedicatory{}

\begin{abstract}
	
	This paper investigates the structure of fully nonlinear equations and their geometric applications. We solve some fully nonlinear version of the Loewner-Nirenberg and Yamabe problems. Notably, we introduce Morse theory technique to construct admissible metrics 
	 under minimal assumptions on the underlying metric, 
	 which can be further relaxed in a broad setting. Additionally, we provide topological obstructions to show the optimality of our structural conditions.

\end{abstract}

\maketitle

 \tableofcontents

  \section{Introduction}  
  

Let $f$ be a smooth symmetric function defined in an open symmetric convex cone  $\Gamma\subsetneq\mathbb{R}^n$,  with vertex at origin and  nonempty 
boundary 
$\partial \Gamma\neq\emptyset$,  containing the  positive cone 
$$\Gamma_n:= \{(\lambda_1,\cdots,\lambda_n)\in \mathbb{R}^n: \mbox{ each } \lambda_i>0\}\subseteq\Gamma.$$
Moreover, $f$ satisfies the fundamental hypotheses as in the following:
\begin{equation}\label{concave}\begin{aligned}
		\mbox{$f$ is a concave function in $\Gamma$,}
\end{aligned}\end{equation} 
\begin{equation}
	\label{elliptic-weak-3}
	\begin{aligned}
		\lim_{t\to+\infty}f(t\lambda)>-\infty, \mbox{ } \forall\lambda\in\Gamma_n,
	\end{aligned}
\end{equation} 
or equivalently (according to Lemma \ref{coro-4} below)  
\begin{equation}
	\label{elliptic-weak}
	\begin{aligned}
		f_{i}(\lambda):=\frac{\partial f}{\partial \lambda_{i}}(\lambda)
		\geq 0 \mbox{ in } \Gamma, \,\, \forall 1\leq i\leq n.
	\end{aligned}
\end{equation}   
In some case, we may assume
\begin{equation}
	\label{sumfi>0}
	\begin{aligned}
		\sum_{i=1}^n f_i(\lambda)>0 \mbox{ in } \Gamma.
	\end{aligned}
\end{equation}  
Such operators   include   the $k$-Hessian operator  as a special case,  where the corresponding function is   \begin{equation}	f=\sigma_k^{1/k}, \,\, \Gamma=\Gamma_k. \nonumber\end{equation} Here  $\sigma_k$ is the $k$-th elementary symmetric function (denote $\sigma_0\equiv 1$), and $\Gamma_k$ is the  $k$-th G{\aa}rding cone.  

Assume that $(M,g)$ is a smooth connected Riemannian manifold of dimension $n\geq 3$
with Levi-Civita connection $\nabla$. 
Let ${Ric}_g$ and 
${R}_g$ be the Ricci and scalar curvature of $g$, 
respectively. 
Denote the modified Schouten tensor by  (see \cite{Gursky2003Viaclovsky})
\begin{equation}
	\begin{aligned}
		A_{{g}}^{\tau,\alpha}=\frac{\alpha}{n-2} \left({Ric}_{g}-\frac{\tau}{2(n-1)}   {R}_{g}\cdot {g}\right), \,\, \alpha=\pm1,  \mbox{  }\tau \in \mathbb{R}.\nonumber
	\end{aligned}
\end{equation}
When $\tau=\alpha=1$, it is the Schouten tensor 
$$A_g=\frac{1}{n-2} \left({Ric}_g-\frac{1}{2(n-1)}{R_g}\cdot g \right).$$ 
For $\tau=n-1$, $\alpha=1$, it corresponds to the Einstein tensor  
$$G_g={Ric}_g-\frac{1}{2}{R}_g\cdot g,$$ 
which is of special importance  in Einstein's theory of gravitation due to the divergence free. 

Given a symmetric $(0,2)$ tensor $A$, denote $\lambda(g^{-1}A)$ the eigenvalues of $A$ with respect to $g$. 
%
Denote $g_u=e^{2u}g$ for some function $u$. 

This paper is devoted to  looking for 
conformal   metrics with further conditions on  (modified) 
Schouten tensor
\begin{equation}
\label{main-equ0-modifiedSchouten}
	\begin{aligned} 
		f(\lambda(g_u^{-1}A_{g_u}^{\tau,\alpha})) =\psi.
	\end{aligned} 
\end{equation}
In particular, when $\tau=1$ and $\alpha=-1$ it reads as follows
\begin{equation}
	\label{main-equ0-Schouten}
	\begin{aligned} 
		f(\lambda(-g_u^{-1}A_{g_u})) =\psi.
	\end{aligned} 
\end{equation}  

The special case of deforming to constant
 scalar curvature on closed manifolds is 
referred  to as  the
 Yamabe problem, which is  a higher dimensional analogue of uniformization theorem for Riemann surface. The
 Yamabe problem was  proved combining the work of Aubin \cite{Aubin1976}, Schoen \cite{Schoen1984} and Trudinger \cite{Trudinger1968}. 
The $\sigma_k$-Yamabe problem  was initiated by Viaclovsky \cite{Viaclovsky2000} in an attempt to 
 find a conformal metric satisfying
\begin{equation}
	\begin{aligned}
		\sigma_k(\lambda({g_u}^{-1}A_{g_u}))=\mbox{const.} \nonumber
	\end{aligned}
\end{equation}
The $\sigma_2$-Yamabe problem  was   solved by Chang-Gursky-Yang \cite{ChangGurskyYang2002,ChangGurskyYang2002-JAM} in dimension 4 with  significant 
topological applications, and further 
by Ge-Wang \cite{Ge2006Wang} on nonlocally conformally flat manifold of dimension   $n>8$.
When $k> {n}/{2}$, the $\sigma_k$-Yamabe problem was solved by Gursky-Viaclovsky \cite{Gursky2007Viaclovsky}.
The remaining case 
$k\leq {n}/{2}$ 
was settled by Sheng-Trudinger-Wang \cite{ShengTrudingerWang2007} under
an additional
 hypothesis that the problem is variational.
When $(M,g)$ is locally conformally flat, the 
problem is well-understood due to the  works by Li-Li \cite{Li2003YYLi} via elliptic theory,  independently by Guan-Wang \cite{Guan2003Wang-CrelleJ} $(k\neq {n}/{2})$  and Brendle-Viaclovsky \cite{Brendle04Viaclovsky} $(k= {n}/{2})$ using flow method. 

The modified Schouten tensor has drawn significant interest in the mathematical literature, in part because \eqref{main-equ0-modifiedSchouten} can approximate fully nonlinear equations for the Schouten tensor, enabling deeper analysis and applications.
Under the  assumption
\begin{equation}
	\label{tau-alpha-2-sigmak1}
	\begin{cases}
		\tau<1 \,&\mbox{ if } \alpha=-1, \\ 
		\tau>n-1 \,&\mbox{ if } \alpha=1,
	\end{cases}
\end{equation} 
it was considered  in  existing literature since the work of 
\cite{Gursky2003Viaclovsky},   among which 
\cite{LiJY2005Sheng,Sheng2006Zhang,Guan2008IMRN,Gursky-Streets-Warren2011,LiGang-2022,Li2011Sheng,FuShengYuan,Chen-Guo-He2022}, to name just a few, focusing on  $\sigma_k$-curvature equation 
 \begin{equation}
	\label{equation1-sigmak-taualpha}	
	\begin{aligned} 
		\sigma_k^{1/k} (\lambda(g_u^{-1}A_{g_u}^{\tau,\alpha}))= \psi 	\end{aligned} \end{equation} 
and the  Krylov type equation  
 \begin{equation}
	\label{main-flow1-sigmak}  
	\begin{aligned}
		\sigma_{k}(\lambda(g_u^{-1}A_{g_u}^{\tau,\alpha}))=\sum_{j=0}^{k-2} \beta_j(x)   \sigma_j (\lambda(g_u^{-1}A_{g_u}^{\tau,\alpha})) +  \psi(x) \sigma_{k-1} (\lambda(g_u^{-1}A_{g_u}^{\tau,\alpha})).
	\end{aligned} 
\end{equation} 
Their approach relies critically on the assumption \eqref{tau-alpha-2-sigmak1}, which guarantees uniform ellipticity or parabolicity. Unfortunately, this assumption excludes the case 
$\tau=n-1$, corresponding to the Einstein tensor, thus limiting the generality of their results.
 In earlier drafts \cite{yuan2020conformal,yuan-PUE-conformal,yuan-PUE2-note} of this paper, 
 by studying \eqref{main-equ0-Schouten},
  we investigated \eqref{main-equ0-modifiedSchouten} and found the optimal range for $(\tau,\alpha)$, but under the assumption that $f$ is positive, homogeneous  and concave. In this paper, we aim to remove this restriction and further develop structural theorems that build on but go beyond our previous work, introducing new perspectives and techniques.
\subsection{The structure of nonlinear  operators} 

The  operator $f$  has drawn significant interest due to its connection with fully nonlinear equations of the similar form 
\begin{equation} 
	\label{equation-cns} 
	\begin{aligned}
		f(\lambda(D^2 u))=\psi \mbox{  } \mbox{ in } \Omega\subset\mathbb{R}^n.  
	\end{aligned} 
\end{equation}
The study of such equations was initiated by \cite{CNS3}, and significant progress on  subsolutions and a priori estimates has recently been made in  \cite{Guan12a,Gabor,Guo-Phong2024,Guo-Phong2024-2}.
Notice that the eigenvalues of coefficient matrix of the linearlized operator 
 of \eqref{equation-cns} 
at $u$ are   
given by \[f_1(\lambda), \cdots, f_n(\lambda), \mbox{ for } \lambda= \lambda(D^2 u).\] 

In the first part of this paper, we briefly study  
the distribution of $f_1(\lambda), \cdots, f_n(\lambda)$  for  fully nonlinear  equations. To this end, we introduce the notion of a test cone.
\begin{definition}
	\label{def-testcone}
	
	Let $\mathcal{S}$ be a 
	symmetric  subset  of  
	$\{\lambda\in\Gamma: f(\lambda)<\sup_\Gamma f \}.$ 
  Denote  
	\begin{equation} 
		\label{def1-MTC}
		\begin{aligned}\mathcal{C}^{\mathfrak{m}}_{\mathcal{S},f}=\left\{\mu\in\mathbb{R}^n: \sum_{i=1}^n f_i(\lambda)\mu_i\geq 0, \,\, \forall\lambda\in \mathcal{S}\right\}.
		\end{aligned}
	\end{equation}  
 We call ${\mathcal{C}}^{\mathfrak{m}}_{\mathcal{S},f}$ the \textit{maximal test cone} of $(f,\mathcal{S})$.
Denote $\mathring{\mathcal{C}}^{\mathfrak{m}}_{\mathcal{S},f}$  
the interior of $\mathcal{C}^{\mathfrak{m}}_{\mathcal{S},f}$.
	
	In addition, for simplicity
	we say that $\mathcal{C}_{\mathcal{S},f}$ is a	\textit{test cone} for $(f,\mathcal{S})$,  if   $\bar{\Gamma}_n\subseteq\mathcal{C}_{\mathcal{S},f}\subseteq {\mathcal{C}}^{\mathfrak{m}}_{\mathcal{S},f}$, and 
	 $\mathcal{C}_{\mathcal{S},f}$ is a symmetric convex cone with vertex at origin.

\end{definition}

We can verify that the maximal test cone is a closed, symmetric, convex cone with vertex at the origin and containing 
 $\bar{\Gamma}_n$.
  This notion is intrinsically connected to the concept of partial uniform ellipticity, as defined in Definition \ref{def-PUE} below.

 One of  the primary aims of this paper 
 is to determine the maximal test cone in the context of genericity.  
To accomplish this, we first introduce and review relevant notions for the pair 
$(f,\Gamma)$. 
Define  the  {\em asymptotical cone} as follows:
\begin{equation}
	\label{component1}
	\begin{aligned}
		\Gamma_{\mathcal{G}}^{f} =
		\left\{\lambda\in\Gamma: \lim_{t\to +\infty}f(t\lambda)>-\infty\right\}.
	\end{aligned}
\end{equation}
In addition, 
$\mathring{\Gamma}_{\mathcal{G}}^{f}$
denotes the interior of $\Gamma_{\mathcal{G}}^{f}$. This
is an  open symmetric convex cone with vertex at origin. 
 We stress that such a cone is fundamental to the study of the maximal test cone and, consequently, to exploring the fully nonlinear equations. 
   
   For   $\sigma\in (\sup_{\partial\Gamma}f, \sup_\Gamma f)$,  
denote the superlevel set and level set respectively by
\begin{equation}
	\begin{aligned} 
		\Gamma^\sigma=  \{\lambda\in\Gamma: f(\lambda)>\sigma\}, \nonumber
		\,\, 
		\partial\Gamma^\sigma= \{\lambda\in\Gamma: f(\lambda)=\sigma\}. \nonumber 
	\end{aligned} 
\end{equation} 
Moreover,  $\bar{\Gamma}^\sigma\equiv\Gamma^\sigma \cup \partial\Gamma^\sigma$ denotes the closure of $\Gamma^\sigma.$  Henceforth   
$$\sup_{\partial\Gamma}f=\sup_{\lambda^0\in\partial\Gamma}\limsup_{\lambda\to\lambda^0}f(\lambda).$$

 Our first result  may be stated as follows.
\begin{theorem} 
	\label{lemma-Y6} 
	Let $n\geq 2$, and
let $\mathcal{S}$ be  as 
	in Definition \ref{def-testcone}.
	Suppose, in addition to \eqref{concave} and \eqref{elliptic-weak-3}, that $\partial\Gamma^\sigma\subseteq\mathcal{S}$ for some 
	$\sigma\in(\sup_{\partial\Gamma}f, \sup_\Gamma f).$
	Then 
	$\mathcal{C}^{\mathfrak{m}}_{\mathcal{S},f}
	=\bar{\Gamma}_{\mathcal{G}}^{f}.$
	
\end{theorem}

 \begin{remark}
 	Note that 
 	$\mathcal{C}^{\mathfrak{m}}_{\mathcal{S},f}	$ may not contained in $\bar{\Gamma}$. A concrete example is when 
 	$\mathcal{S}=\{t\vec{\bf1}:t>0\}$, in which case 
 	$\mathcal{C}^{\mathfrak{m}}_{\mathcal{S},f}	 =\bar{\Gamma}_1.$  Here $\vec{\bf1}=(1,\cdots,1)$ is as in \eqref{def1-vec} below.
 \end{remark}

 The theorem determines the maximal test cone in  generic cases, showing that the closure of $\Gamma_{\mathcal{G}}^{f}$ is exactly 
 the 
  \textit{maximal test cone} whenever $\mathcal{S}$ contains some level set. This includes, among others,  the most  important cases:
  \begin{itemize}
  	\item 
  	$\mathcal{S}=\Gamma$, the global version case;  
  	\item 	$\mathcal{S}=\partial\Gamma^\sigma$, the level set version case. 
  \end{itemize}   
This is new even if $f=\sigma_k^{1/k}$.  
  Note that in Theorem \ref{lemma-Y6}, we do not require the common condition 
\begin{equation}	
	\label{addistruc}	
	\begin{aligned}
		\lim_{t\rightarrow+\infty} f(t\lambda)>f(\mu), \, \,  \forall \lambda,\, \mu\in\Gamma. 
	\end{aligned}
\end{equation}   
 Our results apply to general symmetric functions that satisfy only 
 \eqref{concave} 
 and 
  \eqref{elliptic-weak}.
Hence, in addition to operators satisfying \eqref{addistruc}, they extend to functions such as
\begin{equation}	\label{Guan-Zhang-3}	
	\begin{aligned}
	f= \frac{\sigma_{k}}{\sigma_{k-1}}		-\sum_{j=1}^{k-1} \frac{\alpha_j}{\sigma_j} 	-\sum_{j=0}^{k-2} \frac{\beta_j \sigma_j}{\sigma_k},  \,\, \forall	\alpha_j, \, \beta_j\geq0,\,  	\sum_{j=1}^{k-1}\alpha_j+ \sum_{j=0}^{k-2}  \beta_j >0,
\end{aligned}
\end{equation} 
corresponding to  $\mathring{\Gamma}_{\mathcal{G}}^{f}=\Gamma_{k}$. 
Such a function is concave in $\Gamma_{k-1}$ according to  some result of    \cite{Guan2021Zhang,Guan-Dirichlet}. 
When 
$\alpha_j\equiv0$ for all 
$i$, this is closely linked to the equation \eqref{main-flow1-sigmak}.
 The inclusion $\Gamma_{\mathcal{G}}^{f}\subseteq \mathcal{C}^{\mathfrak{m}}_{\Gamma,f}$ was obtained by \cite{yuan1-closed} in the study of  
 blow-up phenomena for fully nonlinear  equations. This result was achieved without requiring the extra assumption \eqref{addistruc}, thereby extending Sz\'ekelyhidi's \cite{Gabor} $C^{2,\alpha}$-estimate to the optimal setting.
%
%
 As another application, 
Theorem \ref{lemma-Y6} can be employed to explore additional structural properties of 
nonlinear operators, as detailed in Theorems \ref{Y-k+1-4} and \ref{thm2-inequalities2}. This provides an effective framework for 
studying the structure of fully nonlinear elliptic and parabolic equations, including, but not limited to, the prescribed curvature equations \eqref{main-equ0-modifiedSchouten} and \eqref{main-equ0-Schouten} associated with conformal deformations of (modified) Schouten tensors. As a special case of Theorem \ref{Y-k+1-4}, we derive
\begin{theorem} 
	\label{Y-k+1-4-coro} 
	In the presence of 
	\eqref{concave}, \eqref{elliptic-weak-3} and \eqref{sumfi>0},
 the following statements are mutually equivalent.
	\begin{enumerate}
		\item    $(0,\cdots,0,1)\in \mathring{\Gamma}_{\mathcal{G}}^{f}$, 
		i.e.  $\mathring{\Gamma}_{\mathcal{G}}^{f}$ is 
		of type 2 in the sense of \cite{CNS3}. 
		
		\item     $f$ is of fully uniform ellipticity in $\Gamma$   (in the sense of Definition \ref{def-PUE} below).

	 \item   $f$ is of fully uniform ellipticity in $\partial\Gamma^\sigma$ for some $\sigma\in(\sup_{\partial\Gamma}f, \sup_\Gamma f)$.
	 
	 	
	\end{enumerate}
	
\end{theorem}

As a consequence of Theorem \ref{Y-k+1-4-coro}, we may obtain the following result.  
Denote  
\begin{equation}	\label{varrho_1}	\begin{aligned} 	
		(1,\cdots,1,1-\varrho_{\Gamma_{\mathcal{G}}^{f}})\in\partial\Gamma_{\mathcal{G}}^{f}.  \end{aligned}\end{equation} 
 In particular, $\varrho_{\Gamma_k}=\frac{n}{k}.$
 Also write 	\begin{equation}
 	\label{def1-vec}
 	\vec{\bf 0}=(0,\cdots,0)\in\mathbb{R}^n, \,\, 	
 	\vec{\bf 1}=(1,\cdots,1)\in\mathbb{R}^n. 
 \end{equation} 

\begin{corollary}
	\label{coro1-laplace1}
	The function
 $\tilde{f}(\lambda)\equiv f(\sum_j \lambda_j\vec{\bf 1}-\varrho\lambda)$ is of fully uniform ellipticity in $\tilde{\Gamma}\equiv\{\lambda\in\mathbb{R}^n: \sum_j \lambda_j\vec{\bf 1}-\varrho\lambda\in\Gamma\}$, provided that \begin{equation}
 	\label{assumption-4}
 	\varrho<\varrho_{\Gamma_{\mathcal{G}}^{f}}, \,\, \varrho\neq0.
 \end{equation}
 Moreover, 
\eqref{assumption-4} is sharp,
in the sense that if $\varrho = \varrho_{\Gamma_{\mathcal{G}}^{f}}$  then $\tilde{f}$ is not uniformly elliptic  in $\tilde{\Gamma}.$
\end{corollary}

\subsection{Prescribed problems on the conformal classes of complete metrics}

%
%
According to Theorem \ref{Y-k+1-4-coro},  $f$ is of fully uniform ellipticity in $\Gamma$ under the assumption 
\begin{equation}
	\label{assumption2-type2}
	\begin{aligned}
		(0,\cdots,0,1)\in \mathring{\Gamma}_{\mathcal{G}}^f.
	\end{aligned}
\end{equation}
Meanwhile, we may connect \eqref{main-equ0-modifiedSchouten} with \eqref{main-equ0-Schouten}. 
To do this, we take $\varrho=\frac{n-2}{\tau-1}$ and then check that $	\mathrm{tr}\left(g^{-1}(-A_g)\right)g
-\varrho \left(-A_g\right) 
=  \frac{n-2}{\alpha(\tau-1)}A_{g}^{\tau,\alpha},$
and so
\begin{equation}
	\label{check-2}
	\begin{aligned}
	\tilde{f}(\lambda[-g^{-1}A_g])
	=f\left(\frac{n-2}{\alpha(\tau-1)}\lambda[g^{-1}A_{g}^{\tau,\alpha}]\right). 
	\end{aligned}
\end{equation} 
Therefore, 
  under the assumption 
\begin{equation}
	\label{tau-alpha-sharp}
	\begin{cases}
		\tau<1 \,&\text{ if } \alpha=-1, \\ 
		\tau>1+ (n-2)\varrho_{\Gamma_{\mathcal{G}}^{f}}^{-1} \,&\mbox{ if } \alpha=1,
	\end{cases}
\end{equation}
 where $\varrho_{\Gamma_{\mathcal{G}}^{f}}$ 
 is as in \eqref{varrho_1}, 
  Corollary \ref{coro1-laplace1}  indicates that \eqref{main-equ0-modifiedSchouten} falls into the framework of \eqref{main-equ0-Schouten} with $(0,\cdots,0,1)\in \mathring{\Gamma}_{\mathcal{G}}^{f}$  and   vice versa.    
In particular, for 
 \eqref{equation1-sigmak-taualpha} and \eqref{main-flow1-sigmak}, 
\eqref{tau-alpha-sharp} read as 
\begin{equation}
	\label{tau-alpha-sharp-sigmak1}
	\begin{cases}
		\tau<1 \,&\mbox{ if } \alpha=-1, \\ 
		\tau>1+\frac{k(n-2)}{n} \,&\mbox{ if } \alpha=1. \nonumber
	\end{cases}
\end{equation} 

We remark that   \eqref{tau-alpha-sharp}    allows the equation for Einstein tensor when  $\mathring{\Gamma}_{\mathcal{G}}^f\neq \Gamma_n$ 
 \begin{equation}
 	\label{equ1-Einstein}
 	\begin{aligned}
 		f(\lambda[g_u^{-1}G_{g_u}])=\psi. \nonumber
 	\end{aligned}
 \end{equation}
Beyond its crucial role in gravitation theory, the Einstein tensor holds special importance in dimension three owing to its strong link to sectional curvature (as noted in \cite[Section 2]{Gursky-Streets-Warren2010}, see also  Lemma \ref{prop1-einstein-sectional}).
 
In this paper, we will demonstrate that \eqref{tau-alpha-sharp} and \eqref{assumption2-type2} enable the construction of smooth complete conformal metrics satisfying \eqref{main-equ0-modifiedSchouten} and \eqref{main-equ0-Schouten}, respectively. 
These assumptions are critical, as  the topological obstructions in Section \ref{sec1-topo} reveal their necessity.   
Prior to presenting our results, we 
review some additional 
notions.

\begin{definition}
	\label{def2-admissible}
	For the equation \eqref{main-equ0-modifiedSchouten} (respectively,  \eqref{main-equ0-Schouten}),  following \cite{CNS3}
	we say that $g$ is $\Gamma$-{\em admissible}  if 
	$	\lambda(g^{-1}A_{{g}}^{\tau,\alpha})\in\Gamma$
	(respectively, $\lambda(-g^{-1}A_{g})\in\Gamma$).
 Meanwhile,  such a Riemannian manifold $(M,g)$ is also called {\em admissible}.	
 
	For simplicity, we also call such $g$ and $(M,g)$ an admissible metric and admissible manifold, respectively. 
\end{definition}

 \begin{definition}
	\label{def1-quasi-pseudo-admissible}
	
	For the equation \eqref{main-equ0-modifiedSchouten}   
	we call $g$ a $\mathring{\Gamma}_{\mathcal{G}}^{f}$ \textit{pseudo-admissible} metric if 	$	\lambda(g^{-1}A_{{g}}^{\tau,\alpha})\in
	\bar{\Gamma}_{\mathcal{G}}^{f}$. 
	Meanwhile, 
	we say a metric $g$ is $\mathring{\Gamma}_{\mathcal{G}}^{f}$ \textit{quasi-admissible} if  
	\begin{equation}\label{assump1-metric}\begin{aligned}	\lambda(g^{-1}A_{g}^{\tau,\alpha}) \in \bar{\Gamma}_{\mathcal{G}}^{f} \mbox{ in }  M, 	 
		 	\mbox{ and }
					\lambda(g^{-1}A_{g}^{\tau,\alpha}) \in 
					\mathring{\Gamma}_{\mathcal{G}}^{f} \mbox{ at some } p_0\in  M. \nonumber
			\end{aligned}\end{equation}
			Moreover, we say ${g}_u$ is a \textit{maximal} admissible metric to the equation \eqref{main-equ0-modifiedSchouten} 
			if
			$u\geq w$ in $M$ for any admissible conformal metric $g_w$ satisfying the same equation.

			Accordingly,  we have notions of 
			$\mathring{\Gamma}_{\mathcal{G}}^{f}$ \textit{pseudo-admissible, quasi-admissible} and $\mathring{\Gamma}_{\mathcal{G}}^{f}$ \textit{maximal} solutions for 
			equation  
		 \eqref{main-equ0-Schouten}.
		\end{definition} 
	

In the majority of scenarios considered in this paper, the right-hand side of the equation  is a smooth    function  satisfying the following nondegeneracy condition 
\begin{equation} 
	\label{non-degenerate1}
	\begin{aligned}
		\sup_{\partial \Gamma} f< \psi<\sup_{\Gamma}f.   
	\end{aligned}
\end{equation}
However, in other instances  we substitute this 
 with a more 
 enhanced
 requirement 
\begin{equation} 
	\label{non-degenerate2}
	\begin{aligned}
		\sup_{\partial \Gamma_{\mathcal{G}}^f} f< \psi<\sup_{\Gamma}f, \textrm{ where } \sup_{\partial \Gamma_{\mathcal{G}}^f}  f=\sup_{\lambda^0\in\partial\Gamma_{\mathcal{G}}^f}\,\limsup_{\lambda\to\lambda^0, \, \lambda\in \mathring{\Gamma}_{\mathcal{G}}^f} f(\lambda).  
	\end{aligned}
\end{equation}
Moreover, in order to derive local gradient estimate
it requires to impose a technique assumption:
There holds for some  
$\sigma\in (\sup_{\partial\Gamma}f,\, \inf_M\psi]$
 and  $K_0 \in \mathbb{R}$ that 
\begin{equation}
	\label{condition-key100-1}
	\begin{aligned} 
		\sum_{i=1}^n f_i(\lambda)\lambda_i\geq -K_0\sum_{i=1}^n f_i(\lambda) \mbox{ in } \partial\Gamma^\sigma.
	\end{aligned}
\end{equation}   
By Theorem \ref{lemma-Y6}, this condition holds with  $\sigma=\inf_M\psi$ and $K_0=0$ provided $\psi$ obeys \eqref{non-degenerate2}. 
Furthermore,  if $f$ satisfies   
\eqref{addistruc}  then $\Gamma_{\mathcal{G}}^{f}=\Gamma$,  and   Lemma \ref{lemma3.4} ensures that \eqref{elliptic-weak} and \eqref{sumfi>0} hold simultaneously. 



 \subsubsection{Conformal bending of   closed manifolds}  
\label{PrePro0} 

Let $(M,g)$ be a smooth, connected and closed (compact without boundary) Riemannian manifold. 
Moreover, we assume that $g$ is a $\mathring{\Gamma}_{\mathcal{G}}^{f}$ {\em quasi-admissible} metric in the sense of Definition \ref{def1-quasi-pseudo-admissible}. 

By the well-known theorem of Gao-Yau \cite{Gao1986Yau}  and Lohkamp \cite{Lohkamp-1}, any  smooth closed manifold carries a Riemannian metric of negative Ricci curvature.
It would be 
interesting to determine which  conformal class admits
such a metric. 
 %
Below, we prove that every Riemannian metric with quasi-negative Ricci curvature on a closed manifold is conformal to a metric of negative Ricci curvature.
\begin{theorem}
	\label{thm1-Ricci}
	Let $(M,g)$ be a closed connected Riemannian manifold of dimension $n\geq  3$ with quasi-negative Ricci curvature. Then there is a unique smooth conformal metric  $g_u=e^{2u}g$  of negative Ricci curvature with
	$\det(-{ {g}_u^{-1}}Ric_{{g}_u})=1.$
\end{theorem}

In fact 
we   solve 
the equations  on closed 
manifolds with 	$\mathring{\Gamma}_{\mathcal{G}}^{f}$  quasi-admissible metric. 
	On these quasi-admissible manifolds,   as 
	outlined in
	previous drafts  \cite{yuan-PUE-conformal,yuan-PUE2-note} of this paper,
		we  introduce some technique from Morse theory to construct 	$\mathring{\Gamma}_{\mathcal{G}}^{f}$-admissible conformal metrics.

\begin{theorem}
	\label{thm2-existence-closedmanifold}
In addition to 
 	\eqref{concave}, \eqref{elliptic-weak-3}, \eqref{sumfi>0},   \eqref{non-degenerate1},
 we	assume
 \eqref{tau-alpha-sharp} and  \eqref{condition-key100-1} hold.
	 Then there is a smooth admissible  metric $g_u=e^{2u}g$ satisfying  \eqref{main-equ0-modifiedSchouten}. 	
	 In addition, the  resulting metric is unique provided that  $f$ further satisfies \eqref{addistruc}.
\end{theorem}

	It is important to highlight that in this theorem, the restriction on 
$g$ depends only on the assumption that it is quasi-admissible, rather than admissible. This stands in contrast to the assumptions made   in published literature.  

By the identity \eqref{check-2}, this theorem is equivalent to the following:

\begin{theorem} 
	\label{thm1-existence-closedmanifold}
Suppose \eqref{concave}, \eqref{elliptic-weak-3},  \eqref{sumfi>0}, 
	\eqref{non-degenerate1},    \eqref{condition-key100-1} and 
	\eqref{assumption2-type2} hold.
	Then there is a smooth admissible conformal  metric $g_u=e^{2u}g$ satisfying 	 \eqref{main-equ0-Schouten}. 
	Moreover, the   conformal metric is unique if  $f$ further satisfies \eqref{addistruc}.
\end{theorem}


\subsubsection{A fully nonlinear version of Loewner-Nirenberg problem} 
\label{result-manifoldwithboundary}
	Let $(\bar{M},g)$ be a compact Riemannian manifold with smooth boundary $\partial M$, 
	let $M$ denote  the interior of $\bar{M}$. 
%
Based on Theorem \ref{lemma-Y6} and Morse theory, 
we prove 
\begin{theorem}
	\label{thm3-completemetric} 
	In addition to  \eqref{concave}, \eqref{elliptic-weak-3},  \eqref{sumfi>0},   \eqref{non-degenerate1}, \eqref{condition-key100-1} and
  \eqref{tau-alpha-sharp},
	we assume  that
	 one of the following conditions holds.
	\begin{itemize}
		\item[$(i)$]  $\lambda(g^{-1}A_{g}^{\tau,\alpha}) \in \bar{\Gamma}_{\mathcal{G}}^{f}$ in $\bar{M}$. 
		
		\item[$(ii)$]  
		$(\frac{\tau-2}{\tau-1}, \cdots, \frac{\tau-2}{\tau-1}, \frac{\tau}{\tau-1})\in\bar{\Gamma}_{\mathcal{G}}^{f}$.
	\end{itemize}
	Then the interior $M$ admits a smooth complete admissible   metric $g_u=e^{2u}g$ subject to  	 \eqref{main-equ0-modifiedSchouten}. Moreover, $g_u$ is minimal in the sense that $u\leq w$ for any admissible $C^2$   metric $g_w=e^{2w}g$ obeying \eqref{main-equ0-modifiedSchouten}.
\end{theorem}

Both $(i)$ and $(ii)$  are key assumptions for constructing 
$\mathring{\Gamma}_{\mathcal{G}}^{f}$-admissible conformal metrics, using the specific Morse theory technique.
 Notably,   these assumptions can be removed in specific contexts; see 
 Appendix \ref{sec1-cones} 
 for further results and details.

\begin{corollary} 
	\label{existence2-compact-construction}
	In Theorem \ref{thm3-completemetric} both $(i)$ and $(ii)$ 
	can be dropped
	if  $\mathring{\Gamma}_{\mathcal{G}}^{f}=\Gamma_k$ with
	$2\leq  k\leq  \frac{n}{2}$.
\end{corollary}

From   the identity \eqref{check-2}, Theorem \ref{thm3-completemetric} has the following equivalent form.

 \begin{theorem}
\label{thm2-completemetric} 

Suppose, in addition to \eqref{concave}, \eqref{elliptic-weak-3}, \eqref{sumfi>0},  
  \eqref{non-degenerate1} and   \eqref{assumption2-type2},   that \eqref{condition-key100-1} holds for some $\sup_{\partial\Gamma}f<\sigma\leq \inf_M\psi.$ 
Assume that one of the following  holds.
\begin{itemize}
	\item[$(i)'$]  $\lambda(-g^{-1}A_g)\in \bar{\Gamma}_{\mathcal{G}}^{f}$ in $\bar{M}$. 
	
	\item[$(ii)'$]  
	 $(1,\cdots,1,-1)\in\bar{\Gamma}_{\mathcal{G}}^{f}$. 
\end{itemize}
Then the interior $M$ admits a smooth complete admissible   metric $g_u=e^{2u}g$ satisfying 	 \eqref{main-equ0-Schouten}. Moreover, $g_u$ is minimal in the sense that $u\leq w$ for any admissible $C^2$   metric $g_w=e^{2w}g$ obeying	 \eqref{main-equ0-Schouten}.
 \end{theorem}

 Below, we   
 present a topological 
 obstruction to clarify that the assumptions labeled as \eqref{tau-alpha-sharp} 
  and \eqref{assumption2-type2} 
 in Theorem 
 \ref{thm3-completemetric}  
  and \ref{thm2-completemetric}
  cannot be further dropped.
\begin{remark}
	 [Topological obstruction] 
 Let $\Omega\subset \mathbb{R}^3$ be a  smooth, bounded and simply connected domain with  nontrivial higher order homotopy group $\pi_i(\Omega)\neq 0$ for some $i\geq 2$,  
 let $g_E$ be the standard Euclidean metric.  
 Let $g=e^{2u}g_E$ be the solution  to the equation \eqref{main-equ0-Schouten} on      $\Omega$   
 with right-hand side $\psi$ 
satisfying  \eqref{non-degenerate2}. We know $\lambda(-g^{-1}A_g)\in \mathring{\Gamma}_{\mathcal{G}}^{f}$.
Suppose in addition that 
$(0,\cdots,0,1)\in\partial{\Gamma}_{\mathcal{G}}^{f}$.
 Then by 
Proposition \ref{corollary-2}
the Einstein tensor of $g$ must be positive. According to the relation between sectional curvature  and 
  the Einstein tensor in dimension three
  (Lemma \ref{prop1-einstein-sectional}),
  the sectional curvature of $g$ is negative. 
  This  
 shows that $\Omega$ is diffeomorphism to $\mathbb{R}^3$ by the Cartan-Hadamard theorem, which contradicts to
  $\pi_i(\Omega)\neq 0$.
\end{remark}

In some case, we may obtain the asymptotic behavior and uniqueness of the solution. 
Denote the distance from $x$ to $\partial M$ with respect to $g$ by   \begin{equation}	\label{distance-function}	\mathrm{\sigma}(x)=\mathrm{dist}_g(x,\partial M).\end{equation}


 \begin{theorem}
	\label{thm2-unique-0}
Suppose  	$(\frac{\tau-2}{\tau-1}, \cdots, \frac{\tau-2}{\tau-1}, \frac{\tau}{\tau-1})\in\bar{\Gamma}_{\mathcal{G}}^{f}$ and that \eqref{concave}, \eqref{elliptic-weak-3},  \eqref{sumfi>0}, \eqref{tau-alpha-sharp}, \eqref{non-degenerate1}  and \eqref{condition-key100-1}
hold.
Assume  $\psi\big|_{\partial M}\equiv 	 f(D_o\vec{\bf1})$. 
Then the interior
$M$ admits a  smooth complete admissible metric  $ {g}_u=e^{2 {u}}g$ satisfying  \eqref{main-equ0-modifiedSchouten}. Moreover,
	\begin{equation}
		\label{asymptotic-rate3}
		\begin{aligned}
			\lim_{x\rightarrow\partial M} ({u}(x)+\log \mathrm{\sigma}(x))=\frac{1}{2}\log\frac{\alpha (n\tau+2-2n)}{2(n-2)D_o}. \nonumber
		\end{aligned}
	\end{equation} 
	%
	In addition, the   conformal metric is unique whenever  $f$ further satisfies \eqref{addistruc}.
\end{theorem}
As a result, we may conclude that 
\begin{corollary}
	\label{thm1-infinitevolume}
	Suppose that $(M,g)$ is a compact connected Riemannian manifold of dimension $n\geq  3$ with smooth boundary. Then there  is   a unique complete conformal metric  $g_u$ 
	with  $Ric_{ {g}_u}<-1$ and
	$\frac{\sigma_n}{\sigma_{n-1}}(\lambda(-{g}_u^{-1}Ric_{ {g}_u}))=1$. 
\end{corollary}

In the above corollary, a complete metric can be obtained in each conformal class, contrasting with \cite[Theorems A and C]{Lohkamp-1}. Our existence result may also provide a PDE proof of some results in \cite{Lohkamp-2}.

Again,
Theorem  \ref{thm2-unique-0} is equivalent  to the following theorem.
\begin{theorem}
	\label{thm1-unique-0}
	In addition to  \eqref{concave}, \eqref{elliptic-weak-3}, \eqref{sumfi>0} and \eqref{condition-key100-1},
	we assume $(0,\cdots,0,1)\in\mathring{\Gamma}_{\mathcal{G}}^{f}$ and $(1,\cdots,1,-1)\in\bar{\Gamma}_{\mathcal{G}}^{f}$. Suppose that
	$\psi$ satisfies \eqref{non-degenerate1} and
	$\psi\big|_{\partial M}\equiv   f(D_o\vec{\bf1})$. Then  the interior $M$ admits a smooth admissible complete metric $g_u=e^{2u}g$ satisfying \eqref{main-equ0-Schouten}. Moreover, the solution obeys the following asymptotic behavior
	\begin{equation}
		\begin{aligned}
			\lim_{x\rightarrow\partial M} ({u}(x)+\log \mathrm{\sigma}(x)) =  \frac{1}{2}\log\frac{1}{2D_o}. \nonumber
		\end{aligned}
	\end{equation}
	%
In addition, when $f$  satisfies \eqref{addistruc}, the resulting conformal metric is unique.
\end{theorem}




\subsubsection{A complete noncompact version of fully nonlinear Yamabe problem} 
\label{result-complete-noncompact}

Let $(M,g)$ be a complete noncompact connected smooth Riemannian manifold.
Definitely not too surprisingly, things become  more subtle in the case when the background space is complete and  noncompact.   
Unlike the Yamabe problem on closed manifolds, its complete noncompact version is not always solvable, as shown by Jin \cite{Jin1988}. Thus, solving the prescribed curvature problem in the conformal class of complete noncompact metrics requires additional assumptions.

Prior to Jin's work, Ni \cite{Ni-82Invent,Ni-82Indiana} established existence and nonexistence results for the prescribed scalar curvature equation on Euclidean spaces under suitable conditions. Under strong restrictions on the asymptotic behavior of the prescribed functions and the curvature of the background manifold, Aviles-McOwen \cite{Aviles1985McOwen,Aviles1988McOwen2} and Jin \cite{Jin1993} studied the prescribed scalar curvature equation on negatively curved complete noncompact manifolds. Recently, Fu-Sheng-Yuan \cite{FuShengYuan} extended these results to the prescribed 
$\sigma_k$-curvature equation \eqref{equation1-sigmak-taualpha} with $\tau<1$ and 
$\alpha=-1$. However, the imposed restrictions are quite strong and not optimal, even for the conformal prescribed scalar curvature equation. Determining sufficient conditions for the existence of solutions in the class of complete metrics remains an open problem.

Below, 
building on  Theorem \ref{thm3-completemetric} and Morse theory, we solve a fully nonlinear version of the Yamabe problem on complete noncompact Riemannian manifolds under certain asymptotic assumptions at infinity.

\begin{theorem}
	\label{thm2-noncompact} 
	Suppose, in addition to	\eqref{concave}, 
	\eqref{addistruc}, \eqref{non-degenerate1} and  \eqref{tau-alpha-sharp},   that 
	$(M,g)$ carries a  $C^2$ complete pseudo-admissible 
	metric  $g_{\underline{u}}=e^{2\underline{u}}g$  satisfying
	\begin{equation}
		\label{key-assum2}
		\begin{aligned}
			{f(\lambda(g_{\underline{u}}^{-1}A_{g_{\underline{u}}}^{\tau,\alpha}))}   \geq  {\psi} \mbox{ holds uniformly in $M\setminus K_0$.}
		\end{aligned}
	\end{equation}
	Here   $K_0$ is a compact subset of $M$.
	Then there  
	exists 
	a unique smooth complete maximal conformal  admissible metric 
	\eqref{main-equ0-modifiedSchouten}.  
\end{theorem}

  This is equivalent to the following theorem.

  \begin{theorem}
  	\label{thm1-noncompact} 
  	In addition to	\eqref{concave}, 
  	\eqref{addistruc}, \eqref{non-degenerate1} and  \eqref{assumption2-type2}, we assume that 
  	$(M,g)$ carries a  $C^2$ complete pseudo-admissible 
  	metric  $g_{\underline{u}}=e^{2\underline{u}}g$  subject to
  	\begin{equation}
  		\label{key-assum1}
  		\begin{aligned}
  			{f(\lambda(-g_{\underline{u}}^{-1}A_{g_{\underline{u}}}))}   \geq  {\psi} \mbox{ holds uniformly in $M\setminus K_0$.}
  		\end{aligned}
  	\end{equation}
  	Here   $K_0$ is a compact subset of $M$.
  	Then there  
  	exists 
  	a unique smooth complete maximal conformal  admissible metric 
  	\eqref{main-equ0-Schouten}.  
  \end{theorem}

  In the case 
   $(f,\Gamma_1)=(\sigma_1,\Gamma_1)$, we obtain an analogue for the conformal prescribed scalar curvature equation. This extends \cite[Theorem A]{Aviles1988McOwen2} by Aviles-McOwen. Recall that, by a result of \cite{Bland89Kalka}, any noncompact manifold of dimension 
 $n\geq 3$ admits a complete metric with constant negative scalar curvature. Thus, on a noncompact manifold, every negative smooth function bounded from below can be realized as the scalar curvature of some complete metric.

\begin{remark}
	Theorems \ref{thm2-noncompact} and \ref{thm1-noncompact} establish that all geometric and analytic obstructions to the solvability of \eqref{main-equ0-modifiedSchouten} and \eqref{main-equ0-Schouten} on complete noncompact Riemannian manifolds are encoded in the asymptotic conditions \eqref{key-assum2} and \eqref{key-assum1} near infinity. These conditions are both sufficient and necessary:
	
	\begin{itemize}
		\item The equations \eqref{main-equ0-modifiedSchouten} and \eqref{main-equ0-Schouten} are not always solvable in the class of complete noncompact metrics.

		\item Any solution necessarily satisfies the asymptotic condition at infinity, making it a necessary condition for existence.
	\end{itemize} 
	Notably, aside from the asymptotic assumption at infinity, no additional restrictions on the prescribed functions or the curvature of the background manifold are required. This represents a significant improvement over previous related works.
\end{remark}

The paper is organized as follows. In Section \ref{sec1-topo}, we present topological obstructions to show that the uniform ellipticity assumption cannot be removed. In Section \ref{sec-construction-Morse}, we use Morse theory techniques to construct admissible metrics. In Section \ref{sec-description-MC}, we determine the maximal test cone via the asymptotic cone 
$\Gamma_{\mathcal{G}}^f$ in the generic case. As applications, we investigate the structure of  nonlinear concave operators in Section \ref{sec-structure}. Specifically, in Subsection \ref{sec2-PUE}, we introduce the notion of a test cone to measure partial uniform ellipticity. Further applications are provided in Subsection \ref{subsec63-application}. 
In Sections \ref{sec-C0-estimate} and \ref{sec4-estimate-local}, we derive global 
$C^0$-estimates and interior   estimates for first and second derivatives, respectively. In Section \ref{sec-existence-closed}, we establish existence results on closed manifolds. In Section \ref{sec-DirichletProblem}, we derive boundary estimates for the Dirichlet problem. In Section \ref{sec-existence-Dirichlet}, we solve the Dirichlet problem with either finite or infinite boundary data. In Section \ref{sec-Asymp-Uniquess}, we establish asymptotic behavior and uniqueness results for the fully nonlinear Loewner-Nirenberg problem. In Section \ref{sec-complete-noncompact-existence}, we prove existence results on complete noncompact manifolds under certain asymptotic conditions near infinity. Finally, 
in Appendix \ref{sec1-cones}, we further discuss open symmetric convex cones.

 



\section{Geometric optimal condition via topological obstruction}
\label{sec1-topo}

Unless otherwise indicated, 
$\Delta u$, $\nabla^2 u$ and $\nabla u$ refer to the Laplacian, Hessian, and gradient of 
$u$ under the Levi-Civita connection of 
$g$, respectively.

For the 
equations \eqref{main-equ0-modifiedSchouten} and \eqref{main-equ0-Schouten}, as we discussed in introduction: 
\begin{itemize}
	
	\item The equation	\eqref{main-equ0-modifiedSchouten} can be further reduced to the equation 
	of the form   \eqref{main-equ0-Schouten}, i.e.,
	\begin{equation}
		\begin{aligned}
			f(\lambda(- {g}_u^{-1} A_{ {g}_u})) =\psi, \nonumber
		\end{aligned}
	\end{equation}
	and verse vice. See also the identity \eqref{check-2}.

	\item Even for $f=\sigma_1$,  in general one could not expect the solvability of 
	$${f}(\lambda({g}_u^{-1} A_{ {g}_u})) =\psi$$ 
	in the conformal class of complete admissible metrics, due to the obstruction to the existence of complete metrics with positive scalar curvature.

\end{itemize}  
In addition,  we prove that
\begin{itemize}
	\item According to Theorem \ref{Y-k+1-4-coro},  $(0,\cdots,0,1)\in \mathring{\Gamma}_{\mathcal{G}}^{f} \Leftrightarrow $ 
	$f$ is of fully uniform ellipticity in $\Gamma$, i.e. 
	there exists a uniform positive constant $\theta$ such that
	\begin{equation}
		\label{fully-uniform2}
		\begin{aligned}
			f_{i}(\lambda)\geq  \theta\sum_{j=1}^n f_j(\lambda)>0 	\mbox{ in } \Gamma,	\,\, \forall 1\leq i\leq n. 
		\end{aligned}
	\end{equation}
	\item {\em Analytic optimal condition:} 
	\eqref{tau-alpha-sharp} is the optimal condition, under which the equation 
	\eqref{main-equ0-modifiedSchouten} is of uniform ellipticity according to Theorem \ref{Y-k+1-4-coro} and Corollary \ref{coro1-laplace1}. 
	Consequently,   \eqref{tau-alpha-sharp} is  the sharp condition ensuring that the equation \eqref{main-equ0-modifiedSchouten} can be reduced to \eqref{main-equ0-Schouten} being of uniform ellipticity.

	\item {\em Uniform ellipticity $\Rightarrow$ Solvability}: Theorems  \ref{thm2-completemetric}  
 	and \ref{thm3-completemetric} 
	reveal respectively that the uniform ellipticity implies the solvability of  \eqref{main-equ0-Schouten} 
	and \eqref{main-equ0-modifiedSchouten} in the conformal class of smooth  complete admissible metrics.

\end{itemize}

A natural question to raise  is as follows:
\begin{center}
	Can we drop the uniform ellipticity assumption 
	in Theorems \ref{thm2-completemetric} and  \ref{thm3-completemetric}?
 \end{center}

 The rest of this section is devoted to answering this problem via presenting topological obstructions.
To this end, below we prove  the following 
key ingredients.

\begin{lemma}\label{lemma1}
	Let $\Gamma$ be an open symmetric convex cone  $\Gamma\subsetneq\mathbb{R}^n$,  with vertex at origin and  nonempty 
	boundary 
	$\partial \Gamma\neq\emptyset$,  containing  the positive cone.
	For any $\lambda\in\Gamma$,
		\begin{equation}		\begin{aligned}		\lambda_{i_1}+\cdots +\lambda_{i_{1+\kappa_\Gamma}}>0, \mbox{   } \forall 1\leq i_1<\cdots<i_{1+\kappa_\Gamma}\leq n, \nonumber		\end{aligned}	\end{equation}
	where $\kappa_\Gamma$ is as in \eqref{kappa_1} below. \end{lemma}

\begin{proof}  
	The proof is based on the  convexity, symmetry and openness of $\Gamma$. 
\end{proof}

Specifically, we obtain
\begin{lemma}
	\label{lemma1-type1-yuan}
	Suppose $\Gamma$ is of type 1 (i.e. $(0,\cdots,0,1)\in\partial\Gamma$). Then 
	for any	$\lambda\in\Gamma$, 
	$$\sum_{j\neq i}\lambda_j>0, \,\, \forall 1\leq i\leq n. $$
	
\end{lemma}
\begin{proof}
Note  $(0,\cdots,0,1)\in\partial\Gamma$  if and only if $0\leq \kappa_\Gamma\leq n-2$.
	Assume $\lambda_1\leq \lambda_2\leq \cdots \leq \lambda_n$.
	By the definition of $\kappa_\Gamma$ in \eqref{kappa_1} and  the openness of $\Gamma$,
	we see $\lambda_{1+\kappa_\Gamma}>0.$ Combining with Lemma \ref{lemma1}, we know $\sum_{j=1}^{n-1}\lambda_j>0$.
\end{proof}

Applying Lemma \ref{lemma1-type1-yuan} to the cone $\mathring{\Gamma}_{\mathcal{G}}^f$,  we derive
\begin{proposition}
	\label{corollary-2}
	Suppose 
	$(0,\cdots,0,1)\in 
	\partial{\Gamma}_{\mathcal{G}}^f$.
	If $\lambda(-g^{-1}A_g)\in \mathring{\Gamma}_{\mathcal{G}}^f$ then 
	$G_g
	>0$.  
\end{proposition}  

Moreover, we review a relation between sectional curvature and the Einstein tensor in dimension three. (The author acquired this formula from \cite{Gursky-Streets-Warren2010}).

\begin{lemma}
	\label{prop1-einstein-sectional}
	Let $K_g$ be the sectional curvature of $g$.
	Fix $x\in M^3$,  let $\Sigma\subset T_xM$ be a tangent $2$-plane,   $\vec{\bf n}\in T_xM$ the unit normal vector to $\Sigma$, then  $	G_g(\vec{\bf n},\vec{\bf n})=-K_g(\Sigma).$
\end{lemma}



With this at hand,
we may give some obstruction 
to show that, 
when $(0,\cdots,0,1)\in \partial{\Gamma}_{\mathcal{G}}^f$  
the equation
\eqref{main-equ0-Schouten} 
is in general  unsolvable in the conformal class of smooth complete  admissible metrics.

\vspace{2mm}
 \noindent
 {\bf Obstruction}.
	{\em Let $\Omega\subset \mathbb{R}^3$ be a  smooth  simply connected bounded domain with  
	 $\pi_i(\Omega)\neq 0$ for some $i\geq 2$  (e.g. $\Omega=B_2(0)\setminus \overline{B_1(0)}$).
Let $g=g_E$ be the standard Euclidean metric.
	Fix	$ \sup_{\partial \Gamma_{\mathcal{G}}^f} f< \psi<\sup_{\Gamma}f$. 
	Let  ${f}(\lambda(-{g}_u^{-1}A_{{g}_u}))= \psi$   be an equation	on $\Omega$ 
	with  
	$(0,\cdots,0,1)\in \partial{\Gamma}_{\mathcal{G}}^f$.
Suppose by contradiction that 
	Theorem \ref{thm2-completemetric} holds 
	for  such an equation. Then, 	 
	according to Proposition  \ref{corollary-2} and Lemma \ref{prop1-einstein-sectional},
 the resulting metric
	 $g_u$
 is  a complete 
	metric  
	with negative sectional curvature, 
which contradicts to the Cartan-Hadamard theorem. 
This
	reveals that 
	$(0,\cdots,0,1)\in \mathring{\Gamma}_{\mathcal{G}}^f$ 
 imposed in  
	Theorem \ref{thm2-completemetric}   cannot be further
	 dropped.}
	 \vspace{2mm}
	

	In summary,  
we  show that the uniform ellipticity assumption 
(i.e. $(0,\cdots,0,1)\in \mathring{\Gamma}_{\mathcal{G}}^{f}$) imposed in  
Theorem \ref{thm2-completemetric}   cannot be further dropped in general.
Also, this shows that  assumption \eqref{tau-alpha-sharp} 
in Theorem \ref{thm3-completemetric}  is sharp.
	
	\begin{remark}
		Denote $B_{r}(a)=\left\{x\in \mathbb{R}^n: |x-a|^2<r^2\right\}.$ Suppose 
		$\overline{B_{r_1}(a_1)}, \cdots, \overline{B_{r_m}(a_m)}$ are pairwise disjoint.
	Let   $\cup_{i=1}^m B_{r_i}(a_i)\subset B_r(0)$, and 
	 denote
		$\Omega=B_{r+1}(0)\setminus (\cup_{i=1}^m \overline{B_{r_i}(a_i)}).$
		In general one could not expect to solve
		the following problem 
		\begin{equation}  f(e^{-u}\lambda[D^2 u])= \psi    		\mbox{ in }\Omega, \,\, \lim_{x\to\partial\Omega}u(x)=+\infty 
			\end{equation} 
		with some type 1 cone $\mathring{\Gamma}_{\mathcal{G}}^f$.
		Suppose that  it 
		has a  solution $u\in C^\infty(\Omega)$
		 with $\lambda(D^2u)\in\mathring{\Gamma}_{\mathcal{G}}^f$. According to Lemma \ref{lemma1},
		$u$ is a strictly  $(\kappa_{ {\Gamma}_{\mathcal{G}}^f}+1)$-plurisubharmonic function. 
		This is a contradiction by   some results of Harvey-Lawson \cite{Harvey2012Lawson-Adv,Harvey2013Lawson-IUMJ}. 
	\end{remark}

\begin{remark}
	By the formula (see \cite{Gursky2003Viaclovsky})
	\begin{equation}
		\begin{aligned}
			A_{{g}_u}^{\tau,\alpha}
			= A_{g}^{\tau,\alpha}
			+\frac{\alpha(\tau-1)}{n-2}\Delta u g-\alpha  \nabla^2 u
			+\frac{\alpha(\tau-2)}{2}|\nabla u|^2 g
			+\alpha  \dd{u}\otimes\dd{u},  \nonumber
		\end{aligned}
	\end{equation}
	we know that the equation \eqref{main-equ0-modifiedSchouten} 
	has the form
	\begin{equation}
		\label{equation1-modi-Shouten}
		\begin{aligned}
			f(\lambda[g_u^{-1}(\Delta u \cdot g- \varrho\nabla^2u+\chi(x,\nabla u))])=\psi. \nonumber
		\end{aligned}
	\end{equation}
	By Corollary \ref{coro1-laplace1} we may deduce that the equation
	is of uniform ellipticity at solutions with $\lambda[g^{-1}(\Delta u\cdot g- \varrho\nabla^2u+\chi(x,\nabla u))]\in\Gamma$, provided that $\varrho$ satisfies \eqref{assumption-4}.
	
On the other hand, 
as a consequence of Corollary \ref{coro1-laplace1}  
	we may 
	answer a question 
	posed in \cite[Page 5]{GGQ2022}.
 We shall mention that an analogous result 
 was also obtained by previous draft \cite{yuan-PUE2-note}  	when $f$ satisfies \eqref{addistruc}. 
\end{remark}

\section{Construction of admissible conformal metrics via Morse theory}
\label{sec-construction-Morse}

To complete the proof of the results for the fully nonlinear Loewner-Nirenberg and Yamabe problems, the first step is to construct admissible conformal metrics on the manifolds. In this section, we introduce Morse theory techniques for this purpose, following the author's earlier drafts  \cite{yuan-PUE-conformal,yuan-PUE2-note} of this paper.

\subsection{Construction of metrics on closed manifolds}

\begin{lemma}
	\label{lemma2-closed-construction} 
	Let $\Gamma$ be an open symmetric convex cone  $\Gamma\subsetneq\mathbb{R}^n$,  with vertex at origin and  nonempty 
	boundary 
	$\partial \Gamma\neq\emptyset$,  containing  the positive cone.
	Assume $(0,\cdots,0,1) \in\Gamma.$
	Let $(M,g)$ be a closed connected Riemannian manifold with
	\begin{equation} \label{def2-quasi-admissible}
		\begin{aligned}	\lambda(-g^{-1}A_{g}) \in \bar \Gamma  \mbox{ in }  M, 	 
			\mbox{ and } 
			\lambda(-g^{-1}A_{g}) \in 
			{\Gamma} \mbox{ at some } x_0\in  M.  
	\end{aligned}\end{equation}
	Then $M$ admits a smooth conformal metric $g_{\underline{u}}$ with 
	\begin{equation}
		\lambda(-g^{-1}A_{g_{\underline{u}}}) \in \Gamma \mbox{ in } M.
	\end{equation}
\end{lemma}

\begin{proof}  
	Given a $C^2$-smooth  function $w$ on $M$,  
	denote the critical set by 
	$\mathcal{C}(w)=\{x\in  M: dw(x)=0\}.$
	By 
	the openness of $\Gamma$ and the assumption
	\eqref{def2-quasi-admissible} 
	of $\Gamma$-{\em quasi-admissible} metric, there exists a uniform positive constant $r_0$
	such that  
	\begin{equation}
		\label{key1-metric}
		\begin{aligned}
			\lambda(-g^{-1}A_{g})\in\Gamma \mbox{ in } \overline{B_{r_0}(x_0)}.
		\end{aligned}
	\end{equation}
	
	Take a smooth Morse function $w$ with
	the critical set  
	$$\mathcal{C}(w)=\{p_1,\cdots, p_m, p_{m+1}\cdots p_{m+k}\}$$
	among which $p_1,\cdots, p_m$ are all the critical points  being in $M\setminus \overline{B_{r_0/2}(x_0)}$. 
	Pick $q_1, \cdots, q_m\in {B_{r_0/2}(x_0)}$ but not the critical point of $w$. By the homogeneity lemma 
	(see e.g. \cite{Milnor-1997}), 
	one can find a diffeomorphism
	$h: M\to M$, which is smoothly isotopic to the identity, such that 
	\begin{itemize}
		\item $h(p_i)=q_i$, $1\leq i\leq m$.
		\item $h(p_i)=p_i$, $m+1\leq i\leq m+k$.
	\end{itemize}
	Then we obtain a  smooth Morse function 
	\begin{equation}
		\label{Morse1-construction}
		\begin{aligned} 
			v=w\circ h^{-1}.  \nonumber
		\end{aligned}
	\end{equation} 
	We can check that 
	\begin{equation}
		\label{key2}
		\begin{aligned}
			\mathcal{C}(v)=\{q_1,\cdots, q_m, p_{m+1}\cdots p_{m+k}\}\subset \overline{B_{r_0/2}(x_0)}.
		\end{aligned}
	\end{equation}
	
	Next we complete the proof. We can assume $w\leq -1$. Then $v=w\circ h^{-1}\leq 1$.   	Take $\underline{u}=\mathrm{e}^{Nv},$ $g_{\underline{u}}=\mathrm{e}^{2\underline{u}}g,$
	then  
	\begin{equation} 		\label{key3-2}
		\begin{aligned}
			-A_{g_{\underline{u}}}
			=\,&   -A_{g} +N\mathrm{e}^{Nv} \nabla^2 v 	+N^2 \mathrm{e}^{Nv} \left[ \dd{v}\otimes \dd{v} + \mathrm{e}^{Nv} \left(\frac{1}{2}|\nabla v|_{g}^2 g -\dd{v}\otimes \dd{v} \right) \right] \\ 
			=\,&  
			-A_{g} +  N \mathrm{e}^{Nv} \left[  \nabla^2 v + N(1-\mathrm{e}^{Nv})\dd{v}\otimes \dd{v} \right] + \frac{1}{2}N^2 \mathrm{e}^{2Nv}
			|\nabla v|_{g}^2g.
	\end{aligned}\end{equation}  
	Notice that  
	\begin{equation}
		\label{key3}
		\begin{aligned}
			\lambda(g^{-1} \dd{v}\otimes \dd{v} )
			=  |\nabla v|_{g}^2  (0,\cdots,0,1).
		\end{aligned}
	\end{equation}

	\noindent{\bf Case 1}: 
	By \eqref{key1-metric} and the openness of $\Gamma$,  
	\begin{equation}
		\begin{aligned}
			\lambda(g^{-1}( -A_{g} +N\mathrm{e}^{Nv} \nabla^2 v )) \in \Gamma \mbox{ in } \overline{B_{r_0}(x_0)}.
		\end{aligned}
	\end{equation} 
Notice $v\leq -1$, we know $e^{Nv}\ll1$ if $N\gg1.$	  
	Combining with \eqref{key3}  we know
	\begin{equation}
		\begin{aligned}
			\lambda(-g^{-1}A_{g_{\underline{u}}})
			\in \Gamma \mbox{ in } \overline{B_{r_0}(x_0)}.  \nonumber
		\end{aligned}
	\end{equation}
	
	\noindent{\bf Case 2}:  
	By \eqref{key2} there is a uniform positive constant $m_0$ such that $|\nabla v|^2\geq m_0$ in $M\setminus\overline{B_{r_0}(x_0)}$.
	Hence by \eqref{key3}, $(0,\cdots,0,1)\in\Gamma$ and the facts that $v\leq -1$
	and  $\Gamma$ is open, for $N\gg1$
	\begin{equation}
		\lambda(g^{-1}( \nabla^2 v + N(1-\mathrm{e}^{Nv}) \dd{v}\otimes \dd{v}))
		\in \Gamma \mbox{ in } M\setminus\overline{B_{r_0}(x_0)}.
	\end{equation} 
	Together with 
	 \eqref{def2-quasi-admissible},
$ \lambda(-g^{-1}A_{g_{\underline{u}}})
			\in \Gamma \textrm{ in } M\setminus\overline{B_{r_0}(x_0)}.  $
\end{proof}

\subsection{Construction of metrics on compact manifolds with boundary}
\label{subsec1-construction-admissiblemetric}

\begin{lemma} 
	\label{lemma1-construct0}
	
	Let $(\bar{M},g)$ be a smooth, compact,  connected Riemannian manifold with smooth boundary $\partial M$.  
	Let $\Gamma$ be an open symmetric convex cone  $\Gamma\subsetneq\mathbb{R}^n$,  with vertex at origin and  nonempty 
	boundary 
	$\partial \Gamma\neq\emptyset$,  containing  the positive cone.
	In addition to $(0,\cdots,0,1)\in\Gamma$, we assume one of the following conditions  holds
	\begin{enumerate}
		\item[$(1)$] $\lambda(-g^{-1}A_g)\in \bar{\Gamma}$ in $\bar{M}$.
		\item[$(2)$] $(1,\cdots,1,-1)\in\bar{\Gamma}$.
	\end{enumerate}
	Then we have a smooth conformal metric $g_{\underline{u}}=e^{2\underline{u}}g$ such that $
			\lambda(-g^{-1}A_{g_{\underline{u}}})\in  {\Gamma}  \mbox{ in } \bar{M}.
	$
\end{lemma}


\begin{remark}
When $\Gamma$ is of type 1  (i.e. $(0,\cdots,0,1)\in\partial{\Gamma}$), the author     \cite{yuan-PUE3-construction} further proved that the construction   still holds if $(1,\cdots,1,-1)\in {\Gamma}$.
\end{remark}

This is equivalent to the following lemma.
\begin{lemma}
	 \label{lemma1-construct1-modifiedSchouten}
	 Let $(\bar{M},g)$ be a smooth, compact,  connected Riemannian manifold with smooth boundary $\partial M$.  
	 Let $\Gamma$ be an open symmetric convex cone  $\Gamma\subsetneq\mathbb{R}^n$,  with vertex at origin and  nonempty 
	 boundary 
	 $\partial \Gamma\neq\emptyset$,  containing  the positive cone.
	 In addition to \eqref{tau-alpha-sharp}, we assume one of the following conditions holds
	 \begin{itemize}
	 	\item[$(1)'$]  $\lambda(g^{-1}A_{g}^{\tau,\alpha}) \in \bar \Gamma $ in $\bar{M}$. 
	 	
	 	\item[$(2)'$]  
	 	$(\frac{\tau-2}{\tau-1}, \cdots, \frac{\tau-2}{\tau-1}, \frac{\tau}{\tau-1})\in\bar{\Gamma}$.
	 \end{itemize}
 Then we have a smooth conformal metric $g_{\underline{u}}=e^{2\underline{u}}g$ such that
$
 	\lambda( g^{-1}A_{g_{\underline{u}}}^{\tau,\alpha})\in  {\Gamma}  \mbox{ in } \bar{M}.
 $
\end{lemma}


%
  Before presenting the proof, we state the following lemma, which asserts that any compact manifold with boundary admits some Morse function without any critical points.
  (This lemma is believed to hold, but the author is unable to find an explicit reference in the literature. Below, we provide a detailed proof of the result).

\begin{lemma}
	\label{lemma-diff-topologuy}
	Let 
	$\bar{M}$
	be a compact connected 
	manifold of real dimension $n\geq 2$ with smooth boundary. Then there is a smooth function $v$ without any critical points. 
\end{lemma}

\begin{proof} 
	
	Let $X$ be the double of $M$. Let $w$ be a smooth Morse function on $X$ with the critical set $\{p_i\}_{i=1}^{m+k}$, among which $p_1,\cdots, p_m$ are all the critical points  being in $\bar{M}$. 
	Pick $q_1, \cdots, q_m\in X\setminus \bar{M}$ but not the critical point of $w$. By homogeneity lemma 
	(see \cite{Milnor-1997}), 
	one can find a diffeomorphism
	$h: X\to X$, which is smoothly isotopic to the identity, such that  
	$h(p_i)=q_i$ for $1\leq  i\leq m$, and moreover 	$h(p_i)=p_i$ for $m+1\leq  i\leq  m+k$.
	Then $v=w\circ h^{-1}\big|_{\bar{M}}$ is the desired 
	function.
	
\end{proof}

\begin{proof}
	[Proof of Lemma \ref{lemma1-construct0}]
	Let 
	$v$ be a smooth function as  in Lemma \ref{lemma-diff-topologuy}, then $\nabla v\neq 0$ in $\bar{M}$. 
	 For such $v$, we take 
 	$$\underline{u}=e^{Nv}, \,\, {g}_{\underline{u}}=e^{2\underline{u}}g.$$ 
	The straightforward calculation gives that
	\begin{equation}
		\begin{aligned}
			-A_{g_{\underline{u}}}
			=  -A_g +Ne^{Nv} \nabla^2 v 
			+N^2 e^{Nv} \dd{v}\otimes \dd{v} 
			+ 
			N^2 e^{2Nv} \left(\frac{1}{2}|\nabla v|^2 g -\dd{v}\otimes \dd{v} \right).	\nonumber
		\end{aligned}
	\end{equation}

\noindent
{\bf Case 1}: $\lambda(-g^{-1}A_g)\in \bar{\Gamma}$ in $\bar{M}$. In this case  
we assume
\begin{equation}
	\begin{aligned}
		v\leq -1 \mbox{ in } \bar{M}. \nonumber
	\end{aligned}
\end{equation}
The straightforward computation shows that
\begin{align*}
\,&	\lambda
	\left(g^{-1}\left[	Ne^{Nv} \nabla^2 v 
	+N^2 e^{Nv} \dd{v}\otimes \dd{v}
	+ 
	N^2 e^{2Nv} \left(\frac{1}{2}|\nabla v|^2 g -\dd{v}\otimes \dd{v} \right) \right]\right)
	\\
	=\,&
	Ne^{Nv} \lambda \left(g^{-1}\left[	  \nabla^2 v 
	+N \dd{v}\otimes \dd{v} 
	+ 
	N  e^{ Nv} \left(\frac{1}{2}|\nabla v|^2 g -\dd{v}\otimes \dd{v} \right) \right]\right).
\end{align*} 
By $(0,\cdots,0,1)\in\Gamma$, 
$\lambda(g^{-1}[ \nabla^2 v 
+N\dd{v}\otimes \dd{v}]) \in \Gamma$ if $N\gg1.$
Since 	$v\leq -1 $,   when $N\gg1 $ we can check that  $Ne^{Nv}\ll1$, and so
$$\lambda \left(g^{-1}\left[	  \nabla^2 v 
+N \dd{v}\otimes\dd{v} 
+ 
N  e^{ Nv} \left(\frac{1}{2}|\nabla v|^2 g -\dd{v}\otimes\dd{v} \right) \right]\right)\in\Gamma \mbox{ in } \bar{M}. $$ 
Together with $\lambda(-g^{-1}A_g)\in \bar{\Gamma}$,
we see $\lambda(-g^{-1}A_{g_{\underline{u}}})\in\Gamma$ whenever $N\gg1.$

\vspace{2mm}

\noindent
{\bf Case 2}:
$(1,\cdots,1,-1)\in\bar{\Gamma}$.  It is easy to check that 
\[\lambda\left(g^{-1}(\frac{1}{2}|\nabla v|^2 g -\dd{v}\otimes \dd{v}) \right)=|\nabla v|^2\left(\frac{1}{2},\cdots,\frac{1}{2},-\frac{1}{2} \right).\]
Hence
\[\lambda\left(g^{-1}(\frac{1}{2}|\nabla v|^2 g -\dd{v}\otimes \dd{v}) \right)\in\bar{\Gamma} \mbox{ in } \bar{M}.\] 
In addition, we assume
\begin{equation}
	\begin{aligned}
		v\geq 1 \mbox{ in } \bar{M}. \nonumber
	\end{aligned}
\end{equation} 
Since  $\Gamma$ is of type 2, i.e.  $(0,\cdots,0,1)\in\Gamma$, we know
$$\lambda(g^{-1}(-A_g +Ne^{Nv} \nabla^2 v +N^2 e^{Nv} \dd{v}\otimes\dd{v}))\in\Gamma$$ provided $N\gg1.$ Therefore, when $N\gg1$ we have $\lambda(-g^{-1}A_{g_{\underline{u}}})\in\Gamma \mbox{ in } \bar{M}.$  

\end{proof}

 \subsection{Construction of metrics on  complete noncompat manifolds}



We construct a complete noncompact  metric satisfying an asymptotic property.
\begin{proposition}\label{thm4-construction} 
	In the setting of Theorem \ref{thm1-noncompact}, there is a  complete noncompact   conformal metric ${g}_{\hat{u}}=e^{2\hat{u}}g$ with  
	\begin{equation}
		\label{key-assum1-2}
		\begin{aligned}
			{f(\lambda(-{g}_{\hat{u}}^{-1}A_{{g}_{\hat{u}}}))}   \geq  {\psi}, \,\, \lambda(-{g}^{-1}A_{{g}_{\hat{u}}})\in
			\Gamma \mbox{ in } M 
		\end{aligned}
	\end{equation}
	Moreover, $\hat{u}\geq \underline{u}-C_0$ for some constant $C_0>0$, where  $\underline{u}$ is as in \eqref{key-assum1}.
\end{proposition}

\begin{proof}
Let $\underline{u}$ be as in \eqref{key-assum1}. 
Without loss of generality,  it suffices to consider the case $\underline{u}=0$, i.e.,  $g$ satisfies 	\eqref{key-assum1}.
As in \eqref{key-assum1} let $K_0$ be  the compact subset.  
From \eqref{key-assum1} and  $\sup_{\partial \Gamma} f< \psi<\sup_{\Gamma}f$, we know 
$\lambda(-g^{-1}A_g)\in  {\Gamma}$ in     $M\setminus K_0$.  

Pick  two $n$-dimensional compact 
submanifolds $M_1$, $M_2$   with smooth boundary  and with $K_0\subset\subset M_1\subset\subset M_2$. 
Let $v$ be a smooth 
function  with $dv\neq0$ and 
$v\leq 0$ 
on $\bar{M}_2$ (by Lemma \ref{lemma-diff-topologuy}). By straightforward computation, we know that 
$g_{\underline{w}}=e^{2\underline{w}}g$ satisfies
$$\lambda(-g^{-1}A_{g_{\underline{w}}})\in \Gamma \mbox{ in }\bar{M}_2 \mbox{ if }\underline{w}=e^{t(v-1)}, \, t\gg1.$$ 
Choose a cutoff function satisfying
\begin{equation}
	\begin{aligned}
		\zeta\in C^{\infty}_0(M_2), \, 0\leq\zeta\leq 1 \mbox{ and } \zeta\Big|_{M_1}=1. \nonumber
	\end{aligned}
\end{equation}   
Take $\tilde{g}=e^{2\tilde{u}}g$ where $\tilde{u}=e^{Nh}$,
$
	h=
	\begin{cases}
		\zeta  v-1\,& \mbox{ if } x\in M_2,\\
		-1 \,& \mbox{ otherwise.}\nonumber
	\end{cases}
$
Notice that $\lambda(-g^{-1}A_g)\in \Gamma$ in $M\setminus K_0$, and $\lambda(-g^{-1}A_{g_{\tilde{u}}})\in \Gamma$   when restricted to $M_1\cup (M\setminus M_2)$.
By straightforward  computation, we can check for $N\gg1$ that $\lambda(-g^{-1}A_{g_{\tilde{u}}})\in \Gamma $ in $M$. Note that $\tilde{u}=e^{-N}<1$ in $M\setminus M_2$.  Let $\hat{u}=\tilde{u}-\Lambda_2$,
under this scaling,  by Lemma \ref{lemma3.4-2} we know that ${g}_{\hat{u}}$ satisfies \eqref{key-assum1-2} for some    constant $\Lambda_2$.

\end{proof}


\section{Description of the  test cone}
\label{sec-description-MC}


\subsection{Test cone} 

Assume as in Definition \ref{def-testcone} that $\mathcal{S}$  is  a 
symmetric  subset  of  
$\{\lambda\in\Gamma: f(\lambda)<\sup_\Gamma f \}.$
From 
 Definition \ref{def-testcone} 
\begin{equation} 
	\label{key-inequ0}
	\begin{aligned}
		\sum_{i=1}^n f_i(\lambda)\mu_{i}\geq 0, \,\, \forall\lambda\in \mathcal{S}, \, \forall\mu\in\mathcal{C}_{\mathcal{S},f}. 
	\end{aligned}
\end{equation}
Moreover, the 
\textit{maximal test cone} shall obey  
that

\begin{lemma}
	\label{lemma-key3-observ}
	If $\mathcal{S}_1\subseteq \mathcal{S}_2$  then
	$  \mathcal{C}^{\mathfrak{m}}_{\mathcal{S}_2,f} \subseteq \mathcal{C}^{\mathfrak{m}}_{\mathcal{S}_1,f}. $ Moreover, 	if $\mathcal{S}=\mathcal{S}_1\bigcup \mathcal{S}_2$ then 	
	$\mathcal{C}^{\mathfrak{m}}_{\mathcal{S},f} = \mathcal{C}^{\mathfrak{m}}_{\mathcal{S}_1,f}\bigcap \mathcal{C}^{\mathfrak{m}}_{\mathcal{S}_2,f}.$
\end{lemma}
Next, we   prove that 
the interior of any test cone is contained in $\Gamma_1.$


\begin{lemma}
	\label{lemma-22}
	%
	In the presence of \eqref{concave} and \eqref{elliptic-weak-3}, 
	we have
	\begin{enumerate} 
		
		\item[(1)] 
		For any $\mu\in \mathcal{C}_{\mathcal{S},f}$, 
		we have $\sum_{i=1}^n \mu_i\geq 0.$ 
		
		\item[(2)] $0\leq\kappa_{\mathcal{C}_{\mathcal{S},f}}\leq n-1.$
		
		\item[(3)]   
		If, in addition, there exists $\mu\in \mathcal{C}_{\mathcal{S},f}\setminus\{\vec{\bf 0}\}$ with
		$\sum_{i=1}^n \mu_i=0$, then 
		\begin{equation}
			\label{inequ0-permu}
			f_1(\lambda)=f_2(\lambda)=\cdots =f_n(\lambda),\,\, \forall \lambda\in\mathcal{S}.
		\end{equation}

	\end{enumerate}
\end{lemma}
\begin{proof}


	Prove (1) and (2):
	By Lemma \ref{lemma-key3-observ} it is trivial  if $c\vec{\bf 1}\in \mathcal{S}$ for some $c>0$. 
	The remaining case is not  obvious.
	Let $\mu\in {\mathcal{C}_{\mathcal{S},f}}$. 
	By 	\eqref{key-inequ0} and   symmetry of ${\mathcal{C}_{\mathcal{S},f}}$, 	
	for any permutation $\tau=(\tau_1, \cdots, \tau_n)$ of $(1, \cdots, n)$ we have
	\begin{equation}
		\label{inequ1-permu}
		\sum_{i=1}^n	f_i(\lambda)\mu_{\tau_i} \geq 0,   \,\, \forall \lambda\in\mathcal{S}.
	\end{equation} 
	Taking sums over permutations and applying \eqref{inequ1-permu}, we obtain
	\begin{equation}	\label{inequ2-permu}
		\sum_{i=1}^n f_i(\lambda) \sum_{j=1}^n \mu_j 
		= \frac{1}{(n-1)!}
		\sum_{(\tau_1,\cdots,\tau_n)}\sum_{i=1}^n f_i(\lambda)\mu_{\tau_i}
		\geq 0.  
	\end{equation}
 By 
  Lemma \ref{coro-4} below,  $\sum f_i(\lambda)>0$. So 
	$\sum \mu_i\geq 0$.
	As a result, $0\leq\kappa_{\mathcal{C}_{\mathcal{S},f}}\leq n-1.$
	
	Prove the third statement:
	It follows from \eqref{inequ1-permu} and \eqref{inequ2-permu} that if there is $\mu
	\in {\mathcal{C}_{\mathcal{S},f}}$ with $ \sum_{j=1}^n \mu_j=0$, then  
$
	\sum_{i=1}^n f_i(\lambda)\mu_{\tau_i} = 0$ 	for any permutation  $(\tau_1, \cdots, \tau_n)$ of $(1, \cdots, n)$.
	Note that $\mu\neq {\vec{\bf 0}},$  we can deduce \eqref{inequ0-permu}.

\end{proof}


To complete the proof of Lemma \ref{lemma-22},  we need to prove the monotonicity 
of $f$, assuming 
\eqref{concave} and \eqref{elliptic-weak-3}. 
First of all, from \eqref{concave}  we deduce the standard inequality
\begin{equation}\label{concave-1}
	\begin{aligned}
		\sum_{i=1}^n f_i(\lambda)(\mu_i-\lambda_i)\geq f(\mu)-f(\lambda), \,\, \forall \lambda, \, \mu\in\Gamma. 
\end{aligned}\end{equation} 

\begin{lemma}
	\label{coro-4}
	For any $(f,\Gamma)$ satisfying \eqref{concave} and \eqref{elliptic-weak-3}, we have \eqref{elliptic-weak}. 
	Moreover,	
	\begin{equation}
		\label{lemma1-GQY}
		\left\{\lambda\in\Gamma:f(\lambda)<\sup_\Gamma f\right\}=\left\{\lambda\in\Gamma: \sum_{i=1}^n f_i(\lambda)>0\right\}. 
	\end{equation} 
\end{lemma}
\begin{proof}
	\eqref{elliptic-weak-3} implies 	  $\Gamma_n\subseteq	\Gamma_{\mathcal{G}}^{f}$.
	Combining with Lemma \ref{lemma2-key} below, we obtain  \eqref{elliptic-weak}, i.e., $f_i(\lambda)\geq0$ in $\Gamma$, $\forall  1\leq i\leq n$.  
	On the other hand, from $\sum_{i=1}^n f_i(\lambda)=\frac{\dd}{\dd t}f(\lambda+t\vec{\bf1})\big|_{t=0}$
	we see  
	$$ 
	\left\{\lambda\in\Gamma: \sum_{i=1}^n f_i(\lambda)>0\right\}\subseteq
	\left\{\lambda\in\Gamma:f(\lambda)<\sup_\Gamma f\right\}.
	$$  
	
	Below we verify that
	$ \sum_{i=1}^n f_i >0$ when
	$f(\lambda)<\sup_\Gamma f$, 
	$\lambda\in\Gamma$. 
	The rest of proof  is similar to that of \cite[Lemma 2.1]{GQY2018}.
	Assume by contradiction
	that there is some  $\lambda\in\Gamma$ with
	$\frac{\partial  {f}}{\partial \lambda_i}(\lambda)=0$, $\forall  i.$ 
	Then by \eqref{concave-1},
	$ {f}(\mu) \leq {f}(\lambda),$ $\forall \mu\in\Gamma.$
	A contradiction.
	
	
	

	
\end{proof}

Consequently, from Lemmas \ref{lemma-22} and \ref{coro-4} we know
\begin{lemma}
	\label{lemma4.3}
	Suppose $f$	satisfies \eqref{concave} and \eqref{elliptic-weak-3}. Then  $\Gamma_n\subseteq  
	\mathring{\mathcal{C}}^{\mathfrak{m}}_{\mathcal{S},f}\subseteq \Gamma_1$, and
	$\mathring{\mathcal{C}}^{\mathfrak{m}}_{\mathcal{S},f}$  is an open  symmetric convex cone with vertex at origin.  
	Moreover, \eqref{inequ0-permu} holds if and only if $\mathring{\mathcal{C}}^{\mathfrak{m}}_{\mathcal{S},f}=\Gamma_1$.
\end{lemma}

\subsection{Determine the maximal test cone} 

Let $\Gamma_{\mathcal{G}}^{f}$ be as in 
\eqref{component1}.
First of all, we shall review a lemma observed by \cite{yuan1-closed}. 
\begin{lemma}
	[\cite{yuan1-closed}]
	\label{lemma2-key}  
	Suppose $f$ satisfies  
	\eqref{concave} in $\Gamma$. Then for any $\lambda\in\Gamma$ and $\mu\in  \Gamma_{\mathcal{G}}^{f}$,
	\begin{equation}
		\label{lemma315} 
		\begin{aligned}
			f(\lambda+\mu)\geq f(\lambda) \mbox{ and }
			\sum_{i=1}^n f_i(\lambda)\mu_i \geq 0. 
		\end{aligned}
	\end{equation}
	
\end{lemma}


\begin{proof}
	The concavity 	assumption \eqref{concave} yields that
	\begin{equation}
		\label{010} 
		\begin{aligned}
			\sum_{i=1}^n f_i(\lambda)\mu_i \geq \limsup_{t\rightarrow+\infty} f(t\mu)/t, \,\, \forall \lambda, \, \mu\in\Gamma.    
	\end{aligned}\end{equation}
So for any 
	$\lambda\in\Gamma$ and $\mu\in 	\Gamma_{\mathcal{G}}^{f}$, 
	$\sum_{i=1}^n f_i(\lambda)\mu_i \geq 0$, thereby 
	$f(\lambda+\mu)\geq f(\lambda) $ using \eqref{concave-1}.
	
\end{proof}

When $f$ satisfies \eqref{addistruc}, 
 we have a   stronger conclusion.  
	\begin{lemma}
		\label{lemma3.4}
		If $f$ satisfies \eqref{concave} and \eqref{addistruc}, then 
			\begin{equation}
	 	\label{addistruc-4}
			\begin{aligned}
				\sum_{i=1}^n f_i(\lambda)\mu_i>0, \mbox{   } \forall \lambda, \, \mu\in \Gamma.  
			\end{aligned}
		\end{equation}  
	In particular, $f_i(\lambda)\geq 0$
	and $\sum_{i=1}^n f_i(\lambda)>0$ in $\Gamma$.
\end{lemma}
\begin{proof}
	Fix $\lambda,$ $\mu\in \Gamma$.
	Pick $t_0>0$, such that $f(t_0\mu)>f(\lambda)$. 
	 By \eqref{concave-1}, 
	 $$\sum f_i(\lambda)(t_0\mu_i-\lambda_i)\geq f(t_0\mu)-f(\lambda)>0.$$
Take $\mu=\lambda$, we see $\sum f_i(\lambda)\lambda_i> 0$. Thus $\sum  f_i(\lambda)\mu_i>0$, thereby yielding \eqref{elliptic-weak}-\eqref{sumfi>0}.
\end{proof}

\begin{proof}
 	[Proof of Theorem \ref{lemma-Y6}]
	The proof is based on Lemma \ref{lemma2-key} and  Proposition \ref{lemma-Y4} below.
	 From  
	 Lemma \ref{lemma2-key}, we may derive that
 	\begin{equation}
 		\label{Y-asymptotical -charac2}
 		\begin{aligned}	\Gamma_{\mathcal{G}}^{f} 	=\left\{\mu\in\Gamma:	\sum_{i=1}^nf_i(\lambda)\mu_i\geq 0, \mbox{ } \forall \lambda\in\Gamma\right\},	 
 		\end{aligned}	\end{equation}
	and so $\bar{\Gamma}_{\mathcal{G}}^{f}\subseteq \mathcal{C}^{\mathfrak{m}}_{\Gamma,f}.
	$ 
	On the other hand, 
	by Proposition \ref{lemma-Y4} 
	we see 
	$\mathcal{C}^{\mathfrak{m}}_{\partial\Gamma^\sigma,f}=\bar{\Gamma}_{\mathcal{G}}^{f}.$  	
	Since $\partial\Gamma^\sigma\subseteq\mathcal{S}\subseteq\Gamma$, by Lemma \ref{lemma-key3-observ} we know  
	$\mathcal{C}^{\mathfrak{m}}_{\Gamma,f}\subseteq\mathcal{C}^{\mathfrak{m}}_{\mathcal{S},f}
	\subseteq \mathcal{C}^{\mathfrak{m}}_{\partial\Gamma^\sigma,f}
	$. 
Combining with these results, we may conclude that	$\mathcal{C}^{\mathfrak{m}}_{\mathcal{S},f}=\bar{\Gamma}_{\mathcal{G}}^{f}.$

\end{proof}






To complete the proof of Theorem \ref{lemma-Y6},  
it suffices to prove that $	\bar \Gamma_{\mathcal{G}}^{f}$ is the  maximal test cone  
for the level set  case. 
That is

\begin{proposition} 
	\label{lemma-Y4}
	
	Suppose that  $f$ satisfies   \eqref{concave} and  \eqref{elliptic-weak-3}. Then  
	\begin{align*}  \mathcal{C}^{\mathfrak{m}}_{\partial\Gamma^\sigma,f}=
		\bar \Gamma_{\mathcal{G}}^{f}, \,  \forall \sigma\in (\sup_{\partial\Gamma}f, \sup_\Gamma f).
	\end{align*}
\end{proposition}


 The rest of this section is to prove Proposition  \ref{lemma-Y4}.
A key and difficult step is to prove $\mathcal{C}^{\mathfrak{m}}_{\partial\Gamma^\sigma,f}\subseteq\bar{\Gamma}.$
The proof is a combination of the following lemmas.

\begin{lemma}
	\label{lemma-key1-levelset}
	Let  $\mathrm{L}_\mu=\{t\mu: t>0\}$ denote
	the ray generated by $\mu\in\mathbb{R}^n$. 
Assume that \eqref{concave} and \eqref{elliptic-weak-3} hold. Then for $\sigma\in (\sup_{\partial\Gamma}f, \sup_\Gamma f)$ we have
	\begin{enumerate}

		
		\item 	$\mathring{\mathcal{C}}^{\mathfrak{m}}_{\partial\Gamma^\sigma,f} \subseteq \left\{\mu\in\mathbb{R}^n: \sum_{i=1}^n f_i(\lambda)\mu_i>0, \forall \lambda\in\partial\Gamma^\sigma\right\};$
		
		\item $	\left\{\mu\in\mathbb{R}^n: \sum_{i=1}^n f_i(\lambda)\mu_i>0, \forall \lambda\in\partial\Gamma^\sigma\right\}\subseteq
		\left\{\mu\in \mathbb{R}^n: \partial\Gamma^\sigma+\mathrm{L}_\mu\subseteq \Gamma^\sigma\right\};$

				\item $\left\{\mu\in \mathbb{R}^n: \partial\Gamma^\sigma+\mathrm{L}_\mu\subseteq \Gamma^\sigma\right\}\subseteq   	\mathcal{C}^{\mathfrak{m}}_{\partial\Gamma^\sigma,f}.$
	\end{enumerate} 
In particular, 	\begin{enumerate}
	\item[(4)] $	\mathcal{C}^{\mathfrak{m}}_{\partial\Gamma^\sigma,f}		 
	=\overline{\left\{\mu\in\mathbb{R}^n: \sum_{i=1}^n f_i(\lambda)\mu_i>0, \forall \lambda\in\partial\Gamma^\sigma\right\}};$
	\item[(5)] 
	$\mathcal{C}^{\mathfrak{m}}_{\partial\Gamma^\sigma,f} = \overline{	\left\{\mu\in \mathbb{R}^n: \partial\Gamma^\sigma+\mathrm{L}_\mu\subseteq \Gamma^\sigma\right\}}.$
	
\end{enumerate}
\end{lemma}

\begin{proof}
First,	$\mathrm{(3)}$ follows from
	\begin{equation} 
		\label{key-oberv1}
		\frac{\dd}{\dd t} f(\lambda+t\mu)\big|_{t=0}=\sum_{i=1}^n f_i(\lambda)\mu_i.
	\end{equation}
	
	Next, we   prove $(2)$.
	Fix $\mu \in \mathbb{R}^n$ with 
	\begin{equation}
		\label{def-mu-1}
		\begin{aligned}
			\sum_{i=1}^n f_i(\lambda)\mu_i>0,\,\, \forall \lambda\in\partial\Gamma^\sigma.
		\end{aligned}
	\end{equation}
Denote $I=\left\{t>0: \partial\Gamma^\sigma+\{t\mu\}\subseteq \Gamma^\sigma\right\},$
and for $ \lambda\in\partial\Gamma^\sigma$
 $$ I_\lambda=\left\{t>0: \lambda+t\mu\in \Gamma^\sigma\right\}.
 $$ 
	
	It suffices to prove that for all $\lambda\in\partial\Gamma^\sigma$, 
	\begin{equation}
		\begin{aligned}
		I_\lambda=\mathbb{R}^+ :=\{t:t>0\}.
		\end{aligned}
	\end{equation}
	 Fix $\lambda\in\partial\Gamma^\sigma$. By \eqref{key-oberv1} and \eqref{def-mu-1} there is  some $\epsilon>0$ such that  
	$(0,\epsilon) \subseteq I_\lambda$.  
	By the convexity of $\Gamma^\sigma$,  
	$I_\lambda$ is a 
	connected  subset  
	of $\mathbb{R}^+$, i.e., 	$I_\lambda$ is an interval. 
	Suppose by contradiction that 
	$ \sup_{t\in I_\lambda}  t<+\infty$ and denote $t_0=\sup_{t\in I_\lambda}t$. 
	Then    
	$\lambda+t_0\mu\in\partial\Gamma^\sigma$ and $\lambda+t\mu\notin {\Gamma}^\sigma$ for any $t>t_0$. 	That is, 
$f(\lambda+t_0\mu)=\sigma \mbox{ and } f(\lambda+t\mu)\leq\sigma,  \forall t>t_0$,
which implies by \eqref{key-oberv1} that 
	$
	\sum_{i=1}^n	f_i(\lambda+t_0\mu)\mu_i\leq 0.
	$
	It is a contradiction to \eqref{def-mu-1}. So $I_\lambda=\mathbb{R}^+$, 
	and 
	$I =\mathbb{R}^+$.
	

	

\vspace{1mm}
	Finally we  prove $\mathrm{(1)}$. Note that 	$\mathcal{C}^{\mathfrak{m}}_{\partial\Gamma^\sigma,f}$ is a 
	 symmetric convex cone with vertex at origin.  
	For an  interior point $\mu$  of $\mathcal{C}^{\mathfrak{m}}_{\partial\Gamma^\sigma,f}$,
	there is $0<\epsilon\ll1$ (depending on $\mu$) such that
	$\mu-\epsilon\vec{\bf1}\in  \mathcal{C}^{\mathfrak{m}}_{\partial\Gamma^\sigma,f}$.
	This means that
	$\sum_{i=1}^n f_i(\lambda)\mu_i\geq \epsilon \sum_{i=1}^n f_i(\lambda)> 0$, $\forall\lambda\in\partial\Gamma^\sigma.$ 
\end{proof}

Consequently, we obtain  
the following key lemma. 
\begin{lemma}
	\label{lemma-Y3}
	In the presence of \eqref{concave} and \eqref{elliptic-weak-3}, 
 we have $$\mathcal{C}^{\mathfrak{m}}_{\partial\Gamma^\sigma,f}	\subseteq\bar{\Gamma}, \,\, \forall\sigma\in (\sup_{\partial\Gamma}f, \sup_\Gamma f).$$

\end{lemma} 

\begin{proof}
	
	It suffices to prove $\mathring{\mathcal{C}}^{\mathfrak{m}}_{\partial\Gamma^\sigma,f} \subseteq\bar{\Gamma}.$ 
	 %
	Fix $\lambda\in\partial\Gamma^\sigma$. 
Assume by contradiction that   there is $\mu\in \mathring{\mathcal{C}}^{\mathfrak{m}}_{\partial\Gamma^\sigma,f} \setminus\bar{\Gamma}$. Then  $\mu\neq\vec{\bf0}$ and $\lambda+t\mu\notin\bar{\Gamma}$ for $t\gg1$. This contradicts to
$\mathrm{(1)}$ and $\mathrm{(2)}$ of
	Lemma \ref{lemma-key1-levelset}. 
\end{proof}

\begin{proof}
	[Proof of Proposition \ref{lemma-Y4}]
First by Lemma \ref{lemma2-key},
	$
	\bar \Gamma_{\mathcal{G}}^{f}\subseteq \mathcal{C}^{\mathfrak{m}}_{\partial\Gamma^\sigma,f}.$
	The rest is to prove
 $\mathring{\mathcal{C}}^{\mathfrak{m}}_{\partial\Gamma^\sigma,f}\subseteq	\bar \Gamma_{\mathcal{G}}^{f}.$
	Fix 
 $\mu\in \mathring{\mathcal{C}}^{\mathfrak{m}}_{\partial\Gamma^\sigma,f}$
	 and $\lambda^\sigma \in \partial\Gamma^\sigma\cap \Gamma_{\mathcal{G}}^{f}$.
	It requires  to prove that 
	$$\mu+\epsilon\lambda^\sigma\in  \Gamma_{\mathcal{G}}^{f}, \,\, \forall \epsilon>0.$$
By Lemma \ref{lemma-Y3},  $\mu\in \bar{\Gamma}$. So $\mu+\epsilon\lambda^\sigma\in\Gamma$ and $t\mu+\lambda^\sigma\in \Gamma.$ Then by  
	Lemma \ref{lemma2-key}, 
	\begin{align*}
		f(t(\mu+\epsilon\lambda^\sigma))\geq f(t\mu+\lambda^\sigma), \, \forall t\geq  {\epsilon}^{-1}.
	\end{align*}
From Lemma \ref{lemma-key1-levelset},
$
		f(\lambda^\sigma+t\mu)\geq  \sigma, \, \forall t>0.
$
	Thus $\mu+\epsilon\lambda^\sigma\in  \Gamma_{\mathcal{G}}^{f}$. So $\mu\in \bar \Gamma_{\mathcal{G}}^{f}$. 

\end{proof}


\section{On the structure of nonlinear operators}
\label{sec-structure}


 \subsection{Determination of partial uniform ellipticity via test cones}
\label{sec2-PUE}

To understand the structure of fully nonlinear equations, it is natural to introduce the following:
\begin{definition}  
	[Partial uniform ellipticity] 
	\label{def-PUE}  
	Let $\mathcal{S}$ be 
	a symmetric  subset  of $\{\lambda\in\Gamma: f(\lambda)<\sup_\Gamma f \}$. 
	We say that $f$ is of 
	\textit{$\mathrm{m}$-uniform ellipticity} in $\mathcal{S}$,
	if \eqref{elliptic-weak} holds 
	and  there exists a uniform 
	constant $\vartheta$ such that 
	for any $\lambda\in \mathcal{S}$ with 
	$f_1(\lambda) \geq \cdots \geq f_n(\lambda)$,
	\begin{equation}
		\label{PUE1}
		\begin{aligned}
			f_{{i}}(\lambda) \geq   \vartheta \sum_{j=1}^{n}f_j(\lambda)>0, \mbox{ } \forall 1\leq i\leq \mathrm{m}.
			\nonumber
		\end{aligned}
	\end{equation}
	In particular, \textit{$n$-uniform ellipticity} is also called fully uniform ellipticity.  
\end{definition}  

  In 1990s, Lin-Trudinger  \cite{Lin1994Trudinger}
explicitly  verified that $\sigma_k^{1/k}$ is of $(n-k+1)$-unifrom ellipticity in $\Gamma_k$. 
 However, their method does not generalize to broader settings. This raises the following question:
\begin{problem}
	\label{Problem1}
	Given $f$ and $\mathcal{S}$, how to determine the integer $\mathrm{m}$ in	Definition \ref{def-PUE}?
\end{problem} 

In the case  $\mathcal{S}=\partial\Gamma^\sigma$, Guan-Nie \cite{Guan21Nie} reformulated a related problem by defining the rank of the tangent cone at infinity to 
$\Gamma^\sigma$.
 The global case  
($\mathcal{S}=\Gamma$) was solved by the author in a previous draft \cite{yuan2020conformal} under the condition that 
$f$ satisfies \eqref{addistruc}.
Unfortunately, Problem \ref{Problem1} is still unresolved in the general case. This subsection focuses on tackling this problem.

Let $\mathcal{C}_{\mathcal{S},f}$ be a test cone as in Definition \ref{def-testcone}.
 %
%
Denote $\mathring{C}_{\mathcal{S},f}$ the interior  of $\mathcal{C}_{\mathcal{S},f}$. 
For the test cone $\mathcal{C}_{\mathcal{S},f}$, 
  define
\begin{equation}	\label{kappa_1}	\begin{aligned}	 \kappa_{\mathcal{C}_{\mathcal{S},f}}:=\max \left\{k: ({\overbrace{0,\cdots,0}^{k-entries}},{\overbrace{1,\cdots, 1}^{(n-k)-entries}})\in \mathring{C}_{\mathcal{S},f} \right\}. 
\end{aligned}\end{equation}


Below we bridge $\kappa_{\mathcal{C}_{\mathcal{S},f}}$ and
the partial uniform ellipticity of $f$ in $\mathcal{S}$.
\begin{proposition}
	\label{thm1-confi}
	Suppose $f$ satisfies \eqref{concave} and \eqref{elliptic-weak-3} in $\Gamma$.  
	Denote 
	\begin{equation} 
		\vartheta_{\mathcal{C}_{\mathcal{S},f}}=
		\begin{cases}
			1/n,   \,& 
			\mbox{ if } \kappa_{\Gamma_{\mathcal{C}_{\mathcal{S},f}}}
			=0, \\
			\underset{(-\alpha_1,\cdots,-\alpha_{\kappa_{\mathcal{C}_{\mathcal{S},f}}}, \alpha_{1+\kappa_{\Gamma_{\mathcal{C}_{\mathcal{S},f}}}},\cdots, \alpha_n)\in \mathcal{C}_{\mathcal{S},f}, 	 
				\, \alpha_i\geq0}{\sup} 
			\,\, \frac{\alpha_1/n}{\sum_{i=1+\kappa_{\mathcal{C}_{\mathcal{S},f}}}^n		\alpha_i-\sum_{i=2}^{\kappa_{\mathcal{C}_{\mathcal{S},f}}}\alpha_i}, 
			\,&\mbox{ if }\kappa_{\Gamma_{\mathcal{C}_{\mathcal{S},f}}}\geq 1. 
		\end{cases}  \nonumber
	\end{equation}
	Then  \eqref{elliptic-weak} holds, and  
	for any $\lambda\in \mathcal{S}$ with
	$\lambda_1 \leq \cdots \leq\lambda_n$  we have
	\begin{equation} 
		\begin{aligned}
			f_{{i}}(\lambda) 
			\geq \vartheta_{\mathcal{C}_{\mathcal{S},f}} \sum_{j=1}^{n}f_j(\lambda),  \mbox{ }
			\forall  1\leq i\leq 1+ 
			\kappa_{\mathcal{C}_{\mathcal{S},f}}. \nonumber
		\end{aligned}
	\end{equation}
\end{proposition}



\begin{proof}

Take $\lambda\in\mathcal{S}$. 
From Lemma \ref{coro-4},   
$f_i(\lambda)\geq0$, $\forall1\leq i\leq n.$ Assume in addition  that  $\lambda_1\leq\cdots\leq\lambda_n$. The concavity and symmetry of $f$ imply that
\begin{equation}
	\label{asymptotical -concave1}
	f_1(\lambda)\geq\cdots\geq f_n(\lambda), \mbox{ and so } f_1(\lambda)\geq \frac{1}{n}\sum_{i=1}^n f_i(\lambda).
\end{equation}
When $\kappa_{\mathcal{C}_{\mathcal{S},f}}=0$, we  obtain Proposition \ref{thm1-confi} immediately.

From now on we assume $\kappa_{\mathcal{C}_{\mathcal{S},f}}\geq1$. Denote 
$\kappa=\kappa_{\mathcal{C}_{\mathcal{S},f}}$ for simplicity. 
Since $\mathring{\mathcal{C}}_{\mathcal{S},f}$ is open, we may pick  $n$ nonnegative constants $\alpha_1, \cdots, \alpha_n$  with $\alpha_1>0$  such that  
\begin{equation}
	(-\alpha_1,\cdots,	-\alpha_{\kappa}, \alpha_{1+\kappa},\cdots, \alpha_n)\in {\mathcal{C}_{\mathcal{S},f}}.  \nonumber
\end{equation}
According to \eqref{key-inequ0},  
$\sum_{i=1+\kappa}^n \alpha_i f_i(\lambda) \geq \sum_{i=1}^{\kappa} \alpha_i f_i(\lambda).$
Thus 
$f_{1+\kappa}(\lambda)	\geq\frac{\alpha_1}{\sum_{i=1+\kappa}^n \alpha_i}f_1(\lambda).$
In addition,     by  using 
iteration we can derive
\begin{equation}
	\label{theta2}
	\begin{aligned}
		f_{1+\kappa}(\lambda)\geq\frac{\alpha_1}{\sum_{i=1+\kappa}^n
			\alpha_i-\sum_{i=2}^{\kappa}\alpha_i}f_1(\lambda). \nonumber
	\end{aligned}
\end{equation}
\end{proof}

Proposition \ref{thm1-confi} describes  partial uniform ellipticity via a test cone. Moreover, the maximal test cone 
effectively captures the information of partial uniform ellipticity. That is

\begin{lemma}
\label{prop-Y2}
In addition to 
\eqref{concave}, 
we assume that
$f$ is of $(k+1)$-uniform ellipticity in 
$\mathcal{S}$ for some $0\leq k\leq n-1$.  Then there exists $R>0$ such that $$
({\overbrace{-1,\cdots,-1}^{k-entries}},{\overbrace{R,\cdots, R}^{(n-k)-entries}})\in \mathcal{C}^{\mathfrak{m}}_{\mathcal{S},f}.$$
In particular, 
$\kappa_{\mathcal{C}^{\mathfrak{m}}_{\mathcal{S},f}}\geq  k$.
\end{lemma}

\begin{proof}
The case $k=0$ is obvious. Assume now that $k\geq1.$
Let $\vartheta$ be as in Definition \ref{def-PUE}. 
Fix $R>\frac{1}{\vartheta}$. 
Set  $\mu:=({\overbrace{-1,\cdots,-1}^{k-entries}},{\overbrace{R,\cdots, R}^{(n-k)-entries}}).$
By the $(k+1)$-uniform ellipticity,
\begin{equation}
	\begin{aligned}
		\sum_{i=1}^n f_i(\lambda)\mu_i
		\geq  R\vartheta \sum_{i=1}^n f_i(\lambda)-\sum_{i=1}^k f_i(\lambda)>0, 
		\, \forall \lambda\in\mathcal{S}. \nonumber
	\end{aligned}
\end{equation}
This implies $\mu\in \mathcal{C}^{\mathfrak{m}}_{\mathcal{S},f}$. And so $\mu+\vec{\bf 1}\in \mathring{\mathcal{C}}^{\mathfrak{m}}_{\mathcal{S},f}$, as required.

\end{proof}

As a result, we obtain
\begin{proposition} 
\label{coro37}
When the test cone imposed in Proposition \ref{thm1-confi} is the maximum one,
the statement of $(1+ \kappa_{\mathcal{C}^{\mathfrak{m}}_{\mathcal{S},f}})$-uniform ellipticity   is sharp and cannot be  improved.

\end{proposition}

When $\mathcal{S}$ contains a level set,
	$\mathcal{C}^{\mathfrak{m}}_{\mathcal{S},f}=\bar{\Gamma}_{\mathcal{G}}^{f}$  according to Theorem \ref{lemma-Y6}.
As a result, we can address Problem \ref{Problem1} in generic case.  

\begin{theorem}  
	\label{Y-k+1-4} 
	In addition to 
	\eqref{concave} and \eqref{elliptic-weak-3}, we assume that the symmetric set $\mathcal{S}$ contains a level set $\partial\Gamma^\sigma$ with some
	$\sup_{\partial\Gamma}f <\sigma<\sup_\Gamma f$, i.e., $\partial\Gamma^\sigma\subseteq\mathcal{S}$.
	Then  \eqref{elliptic-weak} holds and there exists a uniform positive 
	constant $\vartheta_{\Gamma_{\mathcal{G}}^{f}}$ depending only 
	on $\Gamma_{\mathcal{G}}^{f}$ such that 
	for any $\lambda\in \mathcal{S}$ with
	$\lambda_1 \leq \cdots \leq\lambda_n$,
	\begin{equation}
		\label{inequ2-k+1}
		\begin{aligned}
			f_{{i}}(\lambda) 
			\geq \vartheta_{\Gamma_{\mathcal{G}}^{f}} \sum_{j=1}^{n}f_j(\lambda)>0,  \mbox{ }
			\forall  1\leq i\leq 1+ 
			\kappa_{\Gamma_{\mathcal{G}}^{f}}. 
		\end{aligned}
	\end{equation} 
	Moreover,   the statement of $(1+\kappa_{\Gamma_{\mathcal{G}}^{f}})$-uniform ellipticity 
	cannot be further improved.
	

\end{theorem}

\begin{proof}

	According to Theorem \ref{lemma-Y6},
	$\mathcal{C}^{\mathfrak{m}}_{\Gamma,f}= 
	\bar{\Gamma}_{\mathcal{G}}^{f}.$
Together with Proposition \ref{thm1-confi} we get \eqref{inequ2-k+1}. By Proposition \ref{coro37}, the statement of $(1+\kappa_{\Gamma_{\mathcal{G}}^{f}})$-uniform ellipticity is sharp.
\end{proof}

\begin{remark}
	When considering global version ($\mathcal{S}=\Gamma$) of partial uniform ellipticity,  by Lemma \ref{coro-4}  
	we shall impose the natural condition \eqref{sumfi>0}
	to  ensure $\{\lambda\in\Gamma: f(\lambda)<\sup_{\Gamma}f\}=\Gamma$ and exclude  constant functions.     
\end{remark}




\subsection{Further  structures of operators}
\label{subsec63-application}
We may confirm some inequalities.
The first is as follows, which is also   important   for deriving local gradient estimate.
\begin{theorem}
	\label{thm2-inequalities2}
	In the presence of \eqref{concave}, \eqref{elliptic-weak-3} and \eqref{sumfi>0}, we have 
	\begin{equation} 
		\label{inequ1-laplace}	\begin{aligned} 	
			f_{i}(\lambda) \leq	\frac{1}{\varrho_{\Gamma_{\mathcal{G}}^{f}}}\sum_{j=1}^n f_j(\lambda) 
			\mbox{ in }\Gamma, \,\,  \forall 1\leq i\leq n, 
	\end{aligned}\end{equation} 
	where $\varrho_{\Gamma_{\mathcal{G}}^{f}}$ is as in \eqref{varrho_1}.
	Suppose in addition  that 
	\begin{equation}
		\label{condition-key100-1-2}
		\begin{aligned} 
			\sum_{i=1}^n f_i(\lambda)\lambda_i\geq -K_0\sum_{i=1}^n f_i(\lambda) \mbox{ in } \partial\Gamma^\sigma 
		\end{aligned}
	\end{equation} 
	holds for some $\sigma\in (\sup_{\partial\Gamma}f,\sup_\Gamma f)$ and $K_0\in \mathbb{R}$.
	Then 
	\begin{equation} 	\label{condition-key100-3}
		\begin{aligned}		\sum_{i=1}^n f_i(\lambda)\lambda_i\geq -K_0\sum_{i=1}^n f_i(\lambda) \mbox{ in }  \bar{\Gamma}^\sigma.  
	\end{aligned}	\end{equation}
\end{theorem} 

\begin{proof}
	It is a consequence of Lemmas   \ref{lemma-Y7} and \ref{lemma1-varrho-gamma} below. 
\end{proof}

\begin{lemma}
	\label{lemma-Y7} 
	In addition to \eqref{concave} and \eqref{elliptic-weak-3}, suppose  that \eqref{condition-key100-1-2} holds for some $\sigma\in (\sup_{\partial\Gamma}f,\sup_\Gamma f)$ and $K_0 \in \mathbb{R}$.  Then for such $\sigma$ and $K_0$, 
	there holds
	\begin{equation}
		\label{condition-key100-2}
		\begin{aligned} 
			\sum_{i=1}^n f_i(\lambda) \widetilde{\lambda}_i \geq -K_0\sum_{i=1}^n f_i(\lambda), \,\, \forall \lambda\in \Gamma, \, \forall \widetilde{\lambda}\in \bar{\Gamma}^\sigma.  
		\end{aligned}
	\end{equation}  
	In particular,  we have  \eqref{condition-key100-3}. 
	
\end{lemma}

\begin{proof}
	Denote $\mu = \widetilde{\lambda}+K_0\vec{\bf1} $ for  
	$ \widetilde{\lambda}\in\bar{\Gamma}^\sigma$. Let $\lambda\in\partial\Gamma^\sigma.$
	Using \eqref{concave-1}  we can verify that  $\sum_{i=1}^n f_i(\lambda)  (\widetilde{\lambda}_i -\lambda_i)\geq 0$. Combining with \eqref{condition-key100-1-2} we get  
	\begin{align*}
		\sum_{i=1}^nf_i(\lambda)\mu_i
		\geq 0, \, \forall \lambda\in\partial\Gamma^\sigma.
	\end{align*} 
	This implies that 
	$\mu  \in  \mathcal{C}^{\mathfrak{m}}_{\partial\Gamma^\sigma,f}.$
Combining with the result $\mathcal{C}^{\mathfrak{m}}_{\partial\Gamma^\sigma,f}= \bar{\Gamma}_{\mathcal{G}}^{f}$ 
asserted 
in Theorem \ref{lemma-Y6}, we have
$\mu  \in  \bar{\Gamma}_{\mathcal{G}}^{f}.$   By  
Lemma \ref{lemma2-key} 
or \eqref{Y-asymptotical -charac2},
we obtain \eqref{condition-key100-2}.
\end{proof}

\begin{remark}
Lemma \ref{lemma-Y7} claims that when $\sum  f_i(\lambda)\lambda_i\geq -K_0\sum  f_i(\lambda)$ holds on a level set, then it remains valid in the corresponding superlevel set, with the same $K_0$. Such an inequality, particularly for $K_0=0$, appears naturally as a vital assumption in
literature on fully nonlinear equations; see e.g. \cite{CNS3,GGQ2022,Guan12a,Guan-Dirichlet,Gabor}.
\end{remark}

\begin{lemma}
\label{lemma1-varrho-gamma}
Suppose \eqref{concave}, \eqref{elliptic-weak-3} and  \eqref{sumfi>0} hold. Then for any $\lambda\in\Gamma,$ we have \eqref{inequ1-laplace}.
Moreover,  $\varrho_{\Gamma_{\mathcal{G}}^f}$ is the best constant such that the inequality \eqref{inequ1-laplace} holds.

\end{lemma}  

\begin{proof}
Let $\lambda_1\geq \cdots\geq \lambda_n$, then 
$f_1(\lambda)\leq \cdots \leq f_n(\lambda)$.
Pick $\mu=	 (1,\cdots,1,1-\varrho_{\Gamma_{\mathcal{G}}^{f}}).$
From Theorem \ref{lemma-Y6}, $\mathcal{C}^{\mathfrak{m}}_{\Gamma,f}= \bar \Gamma_{\mathcal{G}}^{f}.$  
Thus
$\sum f_i(\lambda)\mu_i\geq0$. This gives \eqref{inequ1-laplace}.

If there exists  $c_0>\varrho_{\Gamma_{\mathcal{G}}^f}$ so that 
$\sum f_i(\lambda)-c_0f_n(\lambda)
\geq 0$  for all $\lambda\in\Gamma$.
This means that $(1,\cdots,1,1-c_0)\in\mathcal{C}^{\mathfrak{m}}_{\Gamma,f}=\bar{\Gamma}_{\mathcal{G}}^f,$ which implies that $c_0\leq \varrho_{\Gamma_{\mathcal{G}}^f}.$
A contradiction.
\end{proof}

Moreover, we   present additional insights into the structural aspects of operators. 



\begin{lemma}	\label{thm1-PUE-function}
	Assume, in addition to \eqref{concave} and \eqref{elliptic-weak-3}, that   $\kappa_{\Gamma}\leq 1+\kappa_{\Gamma_{\mathcal{G}}^{f}}.$  
	Then 
	\begin{equation}
		\label{key4-Y-5}
		\begin{aligned}
			f_{{i}}(\lambda) \geq   \vartheta_{\Gamma_{\mathcal{G}}^{f}} \sum_{j=1}^{n}f_j(\lambda) 
			,  
			\mbox{ if  } 
			\lambda_i\leq 0.   
		\end{aligned}
	\end{equation}
	
\end{lemma}

\begin{proof} 
	Suppose	$\lambda_1\leq \cdots\leq 
	\lambda_n$. 
	Without loss of generality, assume $\sum f_i(\lambda)>0$ (otherwise $  f_i(\lambda)=0$ $\forall i$).  
	If $\lambda_i\leq 0$  then $1\leq i\leq \kappa_\Gamma$.
	Thus we complete the proof by  Theorem \ref{Y-k+1-4}.
\end{proof}

When $\lambda\in\Gamma_{\mathcal{G}}^f$,
the assumption $\kappa_{\Gamma}\leq 1+\kappa_{\Gamma_{\mathcal{G}}^{f}}$ can be dropped. 
\begin{lemma}	\label{thm2-PUE-function}
	Suppose \eqref{concave}  
	and \eqref{elliptic-weak-3} hold.
	Then  \eqref{key4-Y-5} holds for any $\lambda\in\Gamma_{\mathcal{G}}^f.$
	
\end{lemma}

\begin{proof}
	Let $\lambda\in\Gamma_{\mathcal{G}}^f$, assume	$\lambda_1\leq \cdots\leq 
	\lambda_n$ and $\sum f_i(\lambda)>0$. By the definition of $\kappa_{\Gamma_{\mathcal{G}}^f}$,
	$\lambda_{1+\kappa_{\Gamma_{\mathcal{G}}^f}}\geq0.$   
	If $\lambda_i\leq 0$  then $f_i \geq f_{1+\kappa_{\Gamma_{\mathcal{G}}^f}}$.
	By Theorem \ref{Y-k+1-4}, we complete the proof.
\end{proof}

 The inequality of the form \eqref{key4-Y-5} was commonly imposed as a key assumption to derive  local and global gradient estimate for real fully nonlinear equations;  see e.g.  \cite{SChen2007,Guan2007AJM,Guan1991Spruck,LiYY1991,ShengUrbasWang-Duke,Trudinger90,Urbas2002}. 	

\begin{remark}
	The assumption $\kappa_\Gamma \leq 1+\kappa_{\Gamma_{\mathcal{G}}^f}$ holds in the following cases:
	\begin{enumerate}
		\item $f$  satisfies \eqref{addistruc} (i.e., ${\Gamma}_{\mathcal{G}}^{f}=\Gamma$). 

		\item $f$ is the function as in \eqref{Guan-Zhang-3}  ($\mathring{\Gamma}_{\mathcal{G}}^{f}=\Gamma_{k+1}$ and so $\kappa_{\Gamma}= 1+\kappa_{\Gamma_{\mathcal{G}}^{f}}$).

		\item $\kappa_\Gamma\leq 1$, i.e. $(0,0,1,\cdots,1)\notin\Gamma$.
		
	\end{enumerate}
\end{remark}

\section{Setup and global $C^0$-estimate}
\label{sec-C0-estimate}

 Let $\chi$ be a smooth symmetric $(0,2)$-tensor.
 We consider the equation
 \begin{equation}
 	\label{main-equ2-Schouten}
 	\begin{aligned}
 		F(\mathfrak{g}[u])	\equiv f(\lambda(g^{-1}\mathfrak{g}[u])) =\psi,
 	\end{aligned} 
 \end{equation}
 where $\mathfrak{g}\equiv \mathfrak{g}[u]=e^{-2u}W$ and
 \begin{equation}
	\label{def1-W-tensor}
	\begin{aligned}
 	W\equiv  W[u]=   \nabla^2 u	+\frac{1}{2}|\nabla u|^2 g-\dd{u} \otimes \dd{u}+\chi.
 	\end{aligned} 
\end{equation}
 Clearly $f(\lambda(g^{-1}\mathfrak{g}[u]))=f(\lambda(g_u^{-1}W[u])).$ 
 This coincides with \eqref{main-equ0-Schouten} when $\chi=-A_g$, since  the formula (\cite{Viaclovsky2000})  
 \begin{equation}
 	\label{conformal-formula2}
 	\begin{aligned}
 		-A_{{g_u}} =  \nabla^2 u	+\frac{1}{2}|\nabla u|^2 g
 		- \dd{u}\otimes\dd{u}-A_{g}.  
 \end{aligned}\end{equation}

 In addition to \eqref{concave}, \eqref{elliptic-weak-3} and \eqref{sumfi>0}, we further assume $(0,\cdots,0,1)\in \mathring{\Gamma}_{\mathcal{G}}^f$ and  \eqref{condition-key100-1}  
 holds for some
 $\sup_{\partial\Gamma}f<\sigma\leq \inf_M\psi$ and $K_0 \in \mathbb{R}$. 
 
 \begin{definition}
 	
 	For the equation \eqref{main-equ2-Schouten}, we say $u$ is {\em admissible} if $\lambda(g^{-1}W[u])\in\Gamma.$ 
 \end{definition}

 We will give global $C^0$-estimate.
 Firstly we need to prove the following   lemmas.

 \begin{lemma}
 	\label{lemma3.4-3}
 	Suppose \eqref{concave}, \eqref{elliptic-weak-3} and \eqref{sumfi>0} 
 	hold. Then
 	\begin{equation}
 		\label{inequ2-3.4}
 		\sum_{i=1}^n f_i(\lambda)\mu_i >0,\,\, \forall \lambda\in\Gamma, \,\, \forall \mu\in \mathring{\Gamma}_{\mathcal{G}}^{f}.
 	\end{equation}
 	In particular  for any   $\mu\in \mathring{\Gamma}_{\mathcal{G}}^{f}$, $\sum f_i(\mu)\mu_i >0$ and $f(t\mu)$ is a monotone strictly increasing function of $t\in\mathbb{R}^+$.
 	
 \end{lemma}
 \begin{proof}
  Pick $\epsilon>0$ with $\mu-\epsilon\vec{\bf1}\in \mathring{\Gamma}_{\mathcal{G}}^f$. Thus
we have
 $\sum f_i(\lambda)\mu_i \geq \epsilon \sum f_i(\lambda)$ by Lemma \ref{lemma2-key}.
 \end{proof}
 
 \begin{lemma}
 	\label{lemma3.4-2}
 	Suppose \eqref{concave} and \eqref{elliptic-weak-3} 
 	hold.
 	Then for any $ \lambda\in \mathring{\Gamma}_{\mathcal{G}}^{f}$, we have
 	\begin{align*}
 		\lim_{t\to +\infty}f(t\lambda) =\sup_\Gamma f.
 	\end{align*}

 \end{lemma}
 
 \begin{proof}
 	Fix $R_0>0$ and $\lambda\in \mathring{\Gamma}_{\mathcal{G}}^{f}$. We have 
 	$t_0>0$ such that for any $t\geq t_0$,
 	 $t\lambda-R_0 \vec{\bf 1} \in \mathring{\Gamma}_{\mathcal{G}}^{f}. $ 
 	By Lemma \ref{lemma2-key}, we have
 	$	f(t\lambda)\geq f(R_0 \vec{\bf 1})$
 	for $t\geq t_0.$
 \end{proof}

 \begin{remark}
 	
 	Under the assumptions $	\sup_{\partial \Gamma} f< \psi<\sup_{\Gamma}f$
 	and 
 	\begin{equation}
 		\label{admissible2-strictly}
 		\begin{aligned}
 			\lambda(g^{-1} W[\underline{u}])\in \mathring{\Gamma}_{\mathcal{G}}^{f},
 		\end{aligned}
 	\end{equation}
 	according to Lemma  \ref{lemma3.4-2}, 
 	we know that there exist constants $B_1$ and $B_2$ such that
 	\begin{equation}
 		\label{def1-b1b2}
 		\left\{
 		\begin{aligned}
 			f(e^{-2B_1}\lambda( g_{\underline{u}}^{-1} W[\underline{u}]))  \geq \psi
 			\mbox{ in } M, \\
 			f(e^{-2B_2}\lambda(  g_{\underline{u}}^{-1} W[\underline{u}])) \leq \psi
 			\mbox{ in } M.
 		\end{aligned}
 		\right.
 	\end{equation}
 	
 \end{remark}

 We can prove the following $C^0$-estimate.
 \begin{proposition}
 	\label{lemma-c0general}
 	Suppose $\sup_{\partial \Gamma} f< \psi<\sup_{\Gamma}f$ and 
 	$\lambda(g^{-1} W[\underline{u}])\in \mathring{\Gamma}_{\mathcal{G}}^{f}$. 
 	Then 
 	\begin{equation}
 		\begin{aligned}
 			\min\left\{B_1, \inf_{\partial  M}(u-\underline{u})\right\}\leq (u-\underline{u})(x) \leq \max\left\{ B_2, \sup_{\partial M}(u-\underline{u})\right\}, 
 			\,\, \forall x\in\bar{M}.
 		\end{aligned}
 	\end{equation}
 	Here the constants $B_1$ and $B_2$ satisfy \eqref{def1-b1b2}.
 	In particular, when $M$ is closed 
 	\begin{equation}
 		\begin{aligned}
 			B_1 \leq u-\underline{u} \leq B_2 
 			\mbox{ in }  M.
 		\end{aligned}
 	\end{equation}
 \end{proposition}
 
 \begin{proof}
 	Suppose $(u-\underline{u})(x_0)=\sup_M (u-\underline{u}).$
 	If $x_0\in M$ then at $x_0$ we have	\begin{align*} 
 		\nabla u=\nabla \underline{u}, \,\, \nabla^2 u\leq \nabla^2\underline{u}.
 	\end{align*}
 	Then $W[u]\leq  W[\underline{u}] $ and so
 	\begin{align*} 
 		\psi=f(\lambda( g_u^{-1}W[u])) \leq f(\lambda( g_u^{-1}W[\underline{u}]))
 		=f(e^{-2(u-\underline{u})}\lambda(  g_{\underline{u}}^{-1} W[\underline{u}])).
 	\end{align*}
 	Combining Lemma  \ref{lemma3.4-3} we have
 	$	u-\underline{u} \leq B_2.$
 	
 	Suppose $(u-\underline{u})(x_0')=\inf_M (u-\underline{u}).$
 	If $x_0'\in M$, then at $x_0'$ we get	\begin{align*} 
 		\nabla u=\nabla \underline{u}, \,\, \nabla^2 u\geq \nabla^2\underline{u}.
 	\end{align*}
 	Then  $W[u]\geq  W[\underline{u}] $ and so
 	\begin{align*} 
 		\psi=f(\lambda( g_u^{-1}W[u])) 
 		\geq f(\lambda( g_u^{-1}W[\underline{u}])) =f(e^{-2(u-\underline{u})}
 		\lambda( g_{\underline{u}}^{-1} W[\underline{u}])).
 	\end{align*}
 	Again by   Lemma  \ref{lemma3.4-3} we have
 	$	u-\underline{u} \geq B_1.$
 	
 \end{proof}


\section{Interior gradient and Hessian estimates}
\label{sec4-estimate-local}


According to Theorem \ref{Y-k+1-4-coro}, under the assumption  
$(0,\cdots,0,1)\in \mathring{\Gamma}_{\mathcal{G}}^f$,
$f$ is of fully uniform ellipticity in $\Gamma$, i.e. 
there exists a uniform   constant $\theta$ such that
\begin{equation}
	\begin{aligned}
		f_{i}(\lambda)\geq  \theta\sum_{j=1}^n f_j(\lambda)>0 	\mbox{ in } \Gamma,	\,\, \forall 1\leq i\leq n. \nonumber
	\end{aligned}
\end{equation}
Theorem \ref{thm2-inequalities2} asserts that if  \eqref{condition-key100-1} holds for some $\sup_{\partial\Gamma}f<\sigma\leq\inf_M\psi$ and $K_0\in\mathbb{R}$, then
\begin{equation} 	 \label{condition-key100-4}
	\begin{aligned}		\sum_{i=1}^n f_i(\lambda)\lambda_i\geq -K_0\sum_{i=1}^n f_i(\lambda) \mbox{ in }  \{\lambda\in\Gamma:  f(\lambda)\geq \inf_M\psi\}.  
\end{aligned}	\end{equation}

 Let $e_1,...,e_n$ be a local  frame on $M$.    
 Denote
 $$  \langle X,Y\rangle=g(X,Y),\,\  g_{ij}= \langle e_i,e_j\rangle,\,\   \{g^{ij} \} =  \{g_{ij} \}^{-1}.$$
 Under Levi-Civita connection  of $(M,g)$, $\nabla_{e_i}e_j=\Gamma_{ij}^k e_k$, and $\Gamma_{ij}^k$ denote the Christoffel symbols.  
 For simplicity we write 
 $$\nabla_i=\nabla_{e_i}, \nabla_{ij}=\nabla_i\nabla_j-\Gamma_{ij}^k\nabla_k, 
 \nabla_{ijk}=\nabla_i\nabla_{jk}-\Gamma_{ij}^l\nabla_{lk}-\Gamma^l_{ik}\nabla_{jl}, \cdots,\mbox{ etc}.$$ 
 We have the standard formula
 \begin{equation}
 	\label{formula-3}
 	\begin{aligned}
 		\nabla_{ij}u=\nabla_{ji}u,  
 		\,\,
 		\nabla_{iik}u -\nabla_{kii} u =R^l_{iik}\nabla_l u.
 	\end{aligned}
 \end{equation}
Let
 $W[u]= \nabla^2 u	+\frac{1}{2}|\nabla u|^2 g-\dd{u}\otimes\dd{u}+\chi$ be   as in \eqref{def1-W-tensor}.
 Denote   \begin{equation}
 	\label{def1-lambda1}
 	\begin{aligned}
 		\lambda=\lambda(g^{-1}\mathfrak{g}[u])=\lambda(g_u^{-1}W[u]),
 	\end{aligned}
 \end{equation}
and $W_{ij}=W(e_i,e_j)$, $\mathfrak{g}_{ij}=\mathfrak{g}(e_i,e_j)$,  
 \begin{equation}
 	F^{ij}=\frac{\partial F(\mathfrak{g})}{\partial\mathfrak{g}_{ij}}, \,\,
 	F^{pq,rm}=\frac{\partial^2 F(\mathfrak{g})}{\partial\mathfrak{g}_{pq} \partial\mathfrak{g}_{rm}}.
 \end{equation}
In particular
\begin{equation}
	\begin{aligned}
		W_{ij}
		=  
		\nabla_{ij}u+\frac{1}{2}|\nabla u|^2 g_{ij}-\nabla_i u\nabla_j u+\chi_{ij}.
	\end{aligned}
\end{equation} 

 The linearized operator of \eqref{main-equ2-Schouten} is given by
 \begin{equation}
 	\label{linearized-operator}
 	\begin{aligned}
 		\mathcal{L}v
 		= 
 		F^{ij}\nabla_{ij}v 
 		+ F^{ij}W_{ij,p_l} \nabla_l v  
 	\end{aligned}
 \end{equation}
 for $ v\in C^{2}(M)$,
 where we denote for simplicity 
 \begin{equation}
 	\label{def1-Wij}
 	\begin{aligned}
 		W_{ij,p_l}=g^{kl}\nabla_k u g_{ij} - \delta_{jl} \nabla_{i} u
 		-\delta_{il} \nabla_{j} u. 
 	\end{aligned}
 \end{equation} 

The computations give the following 
\begin{equation}	\label{def1-Wij-3}	\begin{aligned}		W_{ij,p_l}\nabla_l u = |\nabla u|^2 g_{ij}-2\nabla_i u\nabla_ju,	\end{aligned}\end{equation} 
\begin{equation}
	\label{Lu-1}
	\begin{aligned}
		\mathcal{L}u = \,& 
		F^{ij}\nabla_{ij}u+ F^{ij}W_{ij,p_l}\nabla_l u
		\\
		=\,& \frac{1}{2}|\nabla u|^2 F^{ij}g_{ij}
	-F^{ij}\nabla_iu\nabla_j u 
		+F^{ij}W_{ij} 
		-F^{ij}\chi_{ij}.		\end{aligned}
\end{equation}

\begin{equation}
	\label{inequ27}
	\begin{aligned} 
		\nabla_k W_{ij} 
		= \,&  2e^{2u} \nabla_k u \cdot \mathfrak{g}_{ij}
		+e^{2u} \nabla_k \mathfrak{g}_{ij} 
		 	   = 
		\nabla_{kij}u+ W_{ij,p_l}\nabla_{lk}u  
		+\nabla_k \chi_{ij} 
		\\  =\,& 
		\nabla_{kij}u+ 
		g^{pq}\nabla_{kp} u\nabla_q u g_{ij}
		-\nabla_{ki}u\nabla_j u -\nabla_{kj} u\nabla_i u +\nabla_k \chi_{ij}.	\end{aligned}
\end{equation}  
%
Differentiating the equation \eqref{main-equ2-Schouten}, we obtain
\begin{equation}
	\label{diff1-equ} 
	\begin{aligned}
		F^{ij}\nabla_k W_{ij} 
		=e^{2u}\nabla_k \psi +2 \nabla_k u F^{ij}W_{ij},
	\end{aligned}
\end{equation} 
and then 
\begin{equation}
	\label{inequ28}
	\begin{aligned} 
		F^{ij} \nabla_{kij} u = \,& 
		 e^{2u}\nabla_k \psi  +2\nabla_k u F^{ij}W_{ij}
		-F^{ij}\nabla_k\chi_{ij}
		\\\,&
		-g^{pq}\nabla_{kp}u\nabla_q u  F^{ij}g_{ij}
		+2F^{ij}\nabla_{i}u\nabla_{kj}u.
	\end{aligned}
\end{equation} 


When $e_1,...,e_n$ is a local orthonormal frame, under this local frame, the
straightforward computation shows that 
\begin{equation}
	\label{wi1}
	\begin{aligned}
		\nabla_i (|\nabla u|^2)=2\nabla_k u \nabla_{ik}u, \,\,
		\nabla_{ij} (|\nabla u|^2)=2\nabla_{ik}u\nabla_{jk}u+2\nabla_k u\nabla_{ijk}u,
	\end{aligned}
\end{equation} 
Together with \eqref{diff1-equ}, we have 
\begin{equation}	  	
\label{Lw1} 
\begin{aligned}
	\mathcal{L}(|\nabla u|^2) =\,&
	2 F^{ii}|\nabla_{ki}u|^2 +2 F^{ii} (\nabla_{iik}u+W_{ii,p_l}\nabla_{lk}u) \nabla_k u
	\\
	=\,&
	2 F^{ii}|\nabla_{ki}u|^2 +2 F^{ii} (\nabla_{k }W_{ii} - \nabla_k\chi_{ii}) \nabla_k u
	+2F^{ii} (\nabla_{iik}u-\nabla_{kii}u)\nabla_k u	
	\\
	=\,&
	2 F^{ii}|\nabla_{ki}u|^2   
	+4 |\nabla u|^2 F^{ii}W_{ii}
	+2e^{2u}\nabla_k\psi \nabla_ku
	\\\,&
	-2 F^{ii}\nabla_k \chi_{ii} \nabla_k u
	+2F^{ii} R^l_{iik}\nabla_l u\nabla_k u.
\end{aligned}
\end{equation}




\subsection{Local gradient estimate}
\label{subsection-local-C1}


Denote $w=|\nabla u|^2$. 
We consider the quantity $$Q:=\eta w e^\phi$$ 
where 
$\phi$ is determined later.  Following \cite{Guan2003Wang} 
let $\eta$ be a smooth function 
with compact support in $B_r \subset M$ and 
\begin{equation}
	\label{eta1}
	\begin{aligned}
		0\leq \eta\leq 1, \mbox{  } \eta|_{B_{\frac{r}{2}}}\equiv1, 
		\mbox{  } |\nabla\eta|\leq   {C\sqrt{\eta}}/{r},
		\mbox{  } |\nabla^2\eta|\leq {C}/{r^2}.
	\end{aligned}
\end{equation}

The quantity $Q$ attains its maximum at an interior point $x_0\in M$. 
We may assume $|\nabla u|(x_0)\geq 1$.
By maximum principle, at $x_0$
\begin{equation}
	\label{mp1}
	\begin{aligned}
		\frac{\nabla_i\eta}{\eta}+\frac{\nabla_i w}{w}+\nabla_i\phi=0, 
		\,\,
		\mathcal{L}(\log\eta+\log w+\phi)\leq 0.
	\end{aligned}
\end{equation}
Around $x_0$ we choose a local orthonormal frame $e_1,\cdots, e_n$;
for simplicity,  we further assume $e_1,\cdots,e_n$ have been chosen so that at $x_0$,  
$W_{ij}$ is diagonal (so is $F^{ij}$).

\vspace{1mm}
\noindent{\em Step 1. Computation and estimation for $\mathcal{L}(\log w)$.}
By \eqref{wi1},
\begin{equation}
	\label{wiwi1}
	\begin{aligned}
		F^{ii}|\nabla_iw |^2\leq 4|\nabla u|^2 F^{ii} |\nabla_{ik}u|^2.  
	\end{aligned}
\end{equation}
Using Cauchy-Schwarz inequality, one derives by \eqref{mp1}
\begin{equation}
	\label{wiwi2}
	\begin{aligned}
		\frac{ F^{ii} |\nabla_i w|^2}{w^2}\leq (1+\epsilon)
		\left(\frac{1}{\epsilon}F^{ii}\frac{
			|\nabla_i \eta|^2}{ \eta^2}
		+F^{ii} |\nabla_i \phi|^2 \right), \nonumber
	\end{aligned}
\end{equation}
which, together with \eqref{wiwi1}, yields 
\begin{equation}
	\label{wiwi3}
	\begin{aligned}
		\frac{	F^{ii} |\nabla_i w|^2}{w^2}
		\leq (1-\epsilon^2)
		\left(\frac{1}{\epsilon}\frac{
			F^{ii} |\nabla_i \eta|^2}{ \eta^2}
		+F^{ii} |\nabla_i \phi |^2 \right)+\frac{4\epsilon}{w}F^{ii}
		|\nabla_{ik}u |^2. \nonumber
	\end{aligned}
\end{equation}
Set $\epsilon=\frac{1}{4}$. Combining with \eqref{Lw1} we now obtain 
\begin{equation}
	\label{Llogw}
	\begin{aligned}
		\mathcal{L}(\log w) 
		\geq \,&
		\frac{1}{w}F^{ii} |\nabla_{ik}u |^2
		-4F^{ii}\frac{ |\nabla_i \eta|^2}{ \eta^2}
		- F^{ii} |\nabla_i\phi |^2
		+4 F^{ii}W_{ii} -C_0\sum F^{ii}
		-\frac{C_0}{\sqrt{w}}.
	\end{aligned}
\end{equation}

\vspace{2mm}
\noindent{\em Step 2. Construction and computation of $\phi$.}
As in \cite{Guan2008IMRN},  
let $\phi=v^{-N}$, where $v=u-\inf_{B_r} u+2$ (note $v\geq 2$ in $B_r$)  and $N\geq 1$ to be chosen later. 
By straightforward computation 
and \eqref{Lu-1},
\begin{equation}
	\label{Lphi}
	\begin{aligned}
		\mathcal{L}\phi 
		=\,&
		(\phi''-\phi')F^{ii}|\nabla_i u|^2
		+\phi' F^{ii} W_{ii}
		-\phi'F^{ii}\chi_{ii}
		+\frac{1}{2} \phi' w \sum F^{ii}.
	\end{aligned}
\end{equation} 

\subsubsection*{Step 3. Completion of the proof.}

Note that 
$
	\phi'=-Nv^{-N-1}, \, \phi''=N(N+1)v^{-N-2}.
$
We choose $N\gg1$ so that 
\begin{equation}
	\begin{aligned}
		\frac{N}{v^{N}}\leq 1, \,\, 
		\frac{N\theta}{2v} \geq \frac{1}{2}, \,\, Nv^{-N+1}=-\phi'<1,
		\nonumber
	\end{aligned}
\end{equation}
\begin{align*}
	\phi''-\phi'-\phi'^2 = 
	Nv^{-N-1} 
	\left(1+\frac{N+1}{v}-\frac{N}{v^{N+1}} \right) 
	\geq 
	\frac{N^2}{v^{N+2}}.
\end{align*}
From Theorem \ref{Y-k+1-4-coro},  $(0,\cdots,0,1)\in \mathring{\Gamma}_{\mathcal{G}}^{f}$  yields that 
 $f$ obeys \eqref{fully-uniform2}. Hence 
\begin{equation}
	\label{fully-application1}
	\begin{aligned}
		F^{ii} |\nabla_i u |^2 \geq \theta w\sum F^{ii}.
	\end{aligned}
\end{equation}

Plugging \eqref{Llogw},
\eqref{Lphi}, \eqref{fully-application1}  
 into \eqref{mp1}, one has
\begin{equation}
	\begin{aligned}
		0 	
		\geq \,&	\frac{1}{w} F^{ii}|\nabla_{ki}u|^2
		+(\phi''-\phi'-\phi'^2) F^{ii}|\nabla_i u|^2
		+(4+\phi') F^{ii}W_{ii}	
		\\ \,&  
		+\frac{1}{2}\phi'w\sum F^{ii}
		-\phi'F^{ii}\chi_{ii}
		-C_0\sum F^{ii}-\frac{C_0}{\sqrt{w}}  
		\\\,&
		-5F^{ii}\frac{|\nabla_i\eta|^2}{\eta^2}
		+F^{ii}\frac{\nabla_{ii}\eta}{\eta}
		+\frac{F^{ii}W_{ii,p_l}\nabla_l\eta}{\eta}
		\\\geq\,& 
		\frac{\theta   N^2}{2v^{N+2}} w\sum F^{ii}
		-C\left(1+\frac{w}{r^2\eta w}+\frac{w}{\sqrt{r^2\eta w}}
		\right)\sum F^{ii}	
		\\\,&
		+ (4+\phi') F^{ii}W_{ii}
		-\frac{C}{\sqrt{w}}.
	\end{aligned}
\end{equation} 
By Lemma \ref{lemma-Y7}, under the assumption \eqref{condition-key100-1} 
we have
\begin{equation} 
	\label{condition-key100-5} 
	\begin{aligned}
	   F^{ii}W_{ii}  \geq -K_0e^{2u}\sum  F^{ii}
	\end{aligned}
\end{equation}
and, together with \eqref{concave-1},  then 
\begin{equation}
	\label{sumfi-1}
	\begin{aligned}
		\sum F^{ii} 
		\geq \kappa_0>0  
	\end{aligned}
\end{equation}
for some constant $\kappa_0$. 
This gives $\eta w\leq \frac{C}{r^2}$, as required.



\subsection{Interior estimate for second order derivative}
\label{subsection-local-C2}

Let's consider 
$$P(x)=\max_{\xi\in T_x\bar{M}, |\xi|=1} \eta   (W(\xi,\xi)+B|\nabla u|^2),
$$
where 
 $\eta$ is the cutoff function as given by \eqref{eta1}, and 
$B$ is a positive constant to be chosen later.
One knows $P$ achieves maximum at an interior point $x_0\in B_r$ and for $\xi\in T_{x_0}\bar{M}$. For simplicity, set
 $$Q =W +B|\nabla u|^2 g.$$
Around $x_0$ we choose a smooth local orthonormal  frame $e_1,\cdots, e_n$ such that 
 $\Gamma_{ij}^k(x_0)=0$, and
$\{W_{ij}(x_0)\}$ is diagonal (so is $\{F^{ij}(x_0)\}$),
\begin{align*}
	W_{11}(x_0) \geq W_{22}(x_0) \geq \cdots \geq W_{nn}(x_0).
\end{align*}
From this we see  $P(x_0)=\eta(x_0)    Q_{11}(x_0) $, $e_1(x_0)=\xi$.
Since $\Gamma_{ij}^k(x_0)=0$,   at $x_0$
we have $\nabla_i (\Gamma_{jk}^k)=0$, and $\nabla_{i}(W_{11})=\nabla_i W_{11}$, 
$\nabla_{ij}(W_{11})=\nabla_{ij}W_{11}$.
At the
point $x_0$ where the function $ \eta(x) Q_{11}(x)$   (defined near $x_0$) attains its maximum,   we
have 
\begin{equation}
	\label{mp2nd-1}
	\begin{aligned}
 \frac{\nabla_i Q_{11}}{Q_{11}}+\frac{\nabla_i \eta}{\eta}=0, \, \forall 1\leq i\leq n,
		\,\,
		\mathcal{L}( \log\eta+\log Q_{11}) \leq 0.
	\end{aligned}
\end{equation}

In what follows, we indicate $C$ to be the constant (which may vary from line to line) depending only on the quantities specified in the proposition.


\vspace{1mm}
\noindent{\em Step 1. Estimation for 
	$\mathcal{L}(\log Q_{11})$.}
The straightforward computation shows that 
\begin{equation}	\label{inequality-330}
	\begin{aligned}
		F^{ii}	W_{ii,p_l} \nabla_l W_{11}
		=\,& 
		(\nabla_{l11}u + \nabla_{kl}u\nabla_{k}u-2\nabla_{1l}\nabla_1 u+\nabla_l \chi_{11}) \nabla_l u \sum F^{ii}
		\\\,&
		-2F^{ii} (\nabla_{i11}u
		+\nabla_{ik}\nabla_k u -2\nabla_{1i}u\nabla_1u+\nabla_i\chi_{11}) \nabla_i u
		\\ \geq \,& 
		\nabla_{l11}u   \nabla_l u \sum F^{ii} 
		-2F^{ii}  \nabla_{i11}u  \nabla_i u
		-C W_{11} \sum F^{ii},
	\end{aligned}
\end{equation}
\begin{equation}\label{Wiikk1}	\begin{aligned}		\nabla_{kk}W_{ij} 	=\,& e^{2u} \nabla_{kk} \mathfrak{g}_{ij} +4 e^{2u} \nabla_k u \nabla_k \mathfrak{g}_{ij} +2 e^{2u} \nabla_{kk} u \mathfrak{g}_{ij} + 4e^{2u} |\nabla_k u|^2 \mathfrak{g}_{ij} \\	 =\,&\nabla_{kkij}u+|\nabla_{kl}u|^2g_{ij} + \nabla_{kkl} u \nabla_l u g_{ij} -\nabla_{kki}u \nabla_j u 	 \\\,& -\nabla_{kkj}u \nabla_i u  -2 \nabla_{ki}u \nabla_{kj}u	+\nabla_{kk}\chi_{ij},\end{aligned}\end{equation}
and so by \eqref{formula-3}
\begin{equation}
	\label{inequality-329}
	\begin{aligned}
		F^{ii}	(\nabla_{ii}W_{11} -\nabla_{11}W_{ii} )
		=\,& F^{ii}(\nabla_{ii11}u-\nabla_{11ii}u)
		+F^{ii}(\nabla_{ii}\chi_{11}-\nabla_{11}\chi_{ii})	
		\\\,&
		-(|\nabla_{1l}u|^2+\nabla_{11l}\nabla_l u)\sum F^{ii} 
		+F^{ii}\nabla_{iil}u\nabla_l u 
		\\\,&  
		+F^{ii}|\nabla_{il}u|^2
		-2F^{ii}\nabla_{ii1}u \nabla_1 u
		+2F^{ii}\nabla_{11i}u\nabla_i u
		\\ \geq \,& F^{ii}|\nabla_{il}u|^2
		-\left[C W_{11} +|\nabla_{1l}u|^2+\nabla_{l11}u\nabla_l u \right]\sum F^{ii} 
		\\\,&
		+F^{ii}\nabla_{lii}\nabla_l u  
		-2F^{ii}\nabla_{1ii}u \nabla_1 u
		+2F^{ii}\nabla_{i11}u\nabla_i u.
	\end{aligned}
\end{equation}

By 
  \eqref{inequ27} and \eqref{diff1-equ},
we have \eqref{inequ28} and
\begin{equation}
	\label{inequality-331}
	\begin{aligned}
		F^{ii}\nabla_{kii}u\nabla_k u
		=\,&
		2 |\nabla_ku|^2 F^{ii}W_{ii}
		+(e^{2u}\nabla_k\psi -F^{ii}W_{ii,p_l}\nabla_{lk}u-F^{ii}\nabla_k \chi_{ii})\nabla_k u.
	\end{aligned}
\end{equation}
Differentiating the equation twice,
we obtain 
\begin{equation}	
	\label{diff2-equ}
	\begin{aligned}
		F^{ij}\nabla_{kk}W_{ij}  
		=\,&	4e^{2u}\nabla_k u\nabla_k \psi +(2\nabla_{kk}u+4|\nabla_k u|^2)
		F^{ij}W_{ij}
		\\\,&+ e^{2u}\nabla_{kk}\psi
		-e^{2u}F^{pq,rm}\nabla_{k}\mathfrak{g}_{pq} \nabla_{k}\mathfrak{g}_{rm}.
	\end{aligned}
\end{equation}
Together with the concavity of $f$, we get
\begin{equation}
	\label{key2nd-1}
	\begin{aligned}
		F^{ij}\nabla_{kk} W_{ij}  
		\geq 	4e^{2u}\nabla_k u\nabla_k \psi +(2\nabla_{kk}u+4|\nabla_k u|^2)
		F^{ij}W_{ij}
		+ e^{2u}\nabla_{kk}\psi.
	\end{aligned}
\end{equation}  
Thus, from \eqref{inequality-330}, \eqref{inequality-329}, \eqref{key2nd-1}, \eqref{inequality-331} and \eqref{inequ28}  we can derive 
		\begin{equation}
			\label{L-W11-1}
		\begin{aligned}
			\mathcal{L}(W_{11})
			= \,& F^{ii}\nabla_{11} W_{ii}
			+F^{ii}	(\nabla_{ii}W_{11} -\nabla_{11}W_{ii} )
			+F^{ii}W_{ii,p_l} \nabla_l W_{11}
			\\\geq \,&
			F^{ii}|\nabla_{il}u|^2 
			+ F^{ii}\nabla_{lii}u\nabla_l u -2F^{ii}\nabla_{1ii}u\nabla_1u
			-C_3'
			\\\,&
			+(2\nabla_{11}u+4|\nabla_1 u|^2) F^{ii}W_{ii}
			-(CW_{11}+|\nabla_{1l}u|^2)
			\sum F^{ii} 
			\\\geq \,&
			F^{ii}|\nabla_{il}u|^2  
			+2(\nabla_{11}u+|\nabla u|^2) F^{ii}W_{ii}
			\\\,&
			-(CW_{11}+|\nabla_{1l}u|^2)
			\sum F^{ii}
			-C.
		\end{aligned}
	\end{equation} 
We have by   \eqref{Lw1} the following 
\begin{equation} 
	\begin{aligned}
		\mathcal{L} (|\nabla u|^2)
		\geq\,&
		2 F^{ij} \nabla_{ki}u \nabla_{kj}u  
		+4|\nabla u|^2 F^{ij}W_{ij}
		-C(1+\sum F^{ii}).
	\end{aligned}
\end{equation}

Using Cauchy-Schwarz inequality, we have
 $$|\nabla_{ii} u|^2\geq \frac{1}{2}W_{ii}^2-[\frac{1}{2}|\nabla u|^2  -|\nabla_i u|^2 +\chi_{ii}]^2.$$
Without loss of generality, assume 
\begin{equation}
	\label{inequality-452} 	
	\begin{aligned}
	1\leq \frac{1}{2} \nabla_{11}u \leq W_{11}\leq \frac{3}{2}\nabla_{11}u, 	
\end{aligned}
\end{equation} 
 \begin{equation}	\label{inequality-453}
 	\begin{aligned}
 		|\nabla_{11} u|^2\geq \frac{1}{4}W_{11}^2, \,\, 	
 		W_{11} \geq \frac{16 |\nabla u|^2}{\theta},
 	\end{aligned}
 \end{equation} 
where $\theta$ is as in \eqref{fully-uniform2}. 
In particular,
\begin{equation}
	\label{inequ-532}
	\begin{aligned}
		|\nabla u|^2 F^{ii}W_{ii} 
		\leq 	|\nabla u|^2 W_{11} \sum F^{ii} \leq \frac{\theta}{16}W_{11}^2\sum F^{ii}.
		\end{aligned}
\end{equation} 
Since $f$ satisfies \eqref{fully-uniform2},
we further know 
\begin{equation}	\label{inequality-454} 	
	\begin{aligned}
		 F^{ii} |\nabla_{ki}u|^2 
		  \geq  
		  \frac{\theta}{4} W_{11}^2 \sum F^{ii}.
	\end{aligned}
\end{equation}	
Note 
 also that $W_{ii}>-(n-1)W_{11}$.
In conclusion, by \eqref{mp2nd-1}, \eqref{L-W11-1}, \eqref{inequality-452}, \eqref{inequality-453}, \eqref{inequ-532} and \eqref{inequality-454}   we derive
\begin{equation}
\begin{aligned}	
	\mathcal{L}(Q_{11}) 
	\geq \,& 
	(2B+1) F^{ii} |\nabla_{ki}u|^2 
	+\left(2\nabla_{11}u+(4B+2)|\nabla u|^2 \right) F^{ii}W_{ii}
	\\ \,&	   
	- \left[CW_{11} + |\nabla_{1l}u|^2  + BC \right] \sum F^{ii} 
	-BC-C
		\\	\geq \,& 
\frac{\theta(2B+1)}{8} W_{11}^2 \sum F^{ii}    
-BC-C  
\\\,&
- \left[CW_{11} + (1+4(n-1)) W_{11}^2  + BC \right] \sum F^{ii}.
\end{aligned}
\end{equation}
By the concavity of $f$, $|\lambda|\sum f_i(\lambda) \geq \frac{1}{2} (f(|\lambda|\vec{\bf1})-f(\lambda))$. Thus 
there exists $\Lambda_1>0$ such that if $\|W\|>\Lambda_1$ then
\begin{equation}
	 \label{sumfi-2}
 W_{11}  \sum F^{ii} \geq  \kappa_{1}>0.
\end{equation}
Take $B=32(4n-3)\theta^{-1}$, and we assume 
\begin{align*}
	W_{11}  \geq \max\left\{\Lambda_1, 32  \kappa_1^{-1}\theta^{-1} C, 4\sqrt{2\theta^{-1}C}, 
	32\theta^{-1}C
	\right\},
\end{align*}
then
\begin{equation}
	\begin{aligned}	
		\mathcal{L}(Q_{11})  
	 	\geq  
		\frac{\theta(2B+1)}{16} W_{11}^2 \sum F^{ii}
		\geq 4(4n-3)W_{11}^2 \sum F^{ii}.
	\end{aligned}
\end{equation} 
Hence, we obtain
\begin{equation}
	\begin{aligned}
		\mathcal{L}(\log Q_{11}) 
	\geq 
		\frac{4(4n-3)W_{11}^2}{Q_{11}} \sum F^{ii} -\frac{F^{ii}|\nabla_iQ_{11}|^2}{Q_{11}^2} .
	\end{aligned}
\end{equation}

\noindent{\em Step 2. Completion of the proof.}
From \eqref{eta1},
\begin{equation}
	\begin{aligned}
		F^{ii}\frac{|\nabla_i\eta|^2}{\eta^2} \leq \frac{C}{r^2 \eta}\sum_{i=1}^n F^{ii},
		\,\,
		\mathcal{L}(\log\eta)\geq -\frac{C}{r^2 \eta}\sum_{i=1}^n F^{ii} -\frac{C}{r\sqrt{\eta}} 
		\sum_{i=1}^n F^{ii}. 
		\nonumber
	\end{aligned}
\end{equation} 

Assume $Q_{11}\leq 2 W_{11}$.
Finally, at $x_0$  
we have by  \eqref{mp2nd-1}
$ \frac{\nabla_iQ_{11} }{Q_{11}}=-\frac{\nabla_i \eta}{\eta}$ and
\begin{equation}
	\label{yuan1-secondorder}
	\begin{aligned} 
	\mathcal{L}(\log\eta+\log Q_{11})\geq (4n-3)Q_{11}\sum F^{ii} -\frac{2C}{r^2 \eta}\sum_{i=1}^n F^{ii} -\frac{C}{r\sqrt{\eta}} 
		\sum_{i=1}^n F^{ii}.
	\end{aligned}
\end{equation}
Plugging this inequality  into \eqref{mp2nd-1}, we conclude that 
\begin{equation}\begin{aligned}
		\eta Q_{11}\leq \frac{C}{r^2} \mbox{ at } x_0.  \nonumber
\end{aligned}\end{equation}
The proof is complete.





\section{Equations on closed manifolds}
\label{sec-existence-closed}

First we prove the uniqueness result.  
\begin{lemma} 
	\label{lemma3-unique}
	Let $(M,g)$ be a closed, connected and smooth Riemannian manifold, let $\chi$ be a smooth symmetric $(0,2)$-tensor.  
As in \eqref{def1-W-tensor} we denote
	$$W[u] \equiv\nabla^2 u	+\frac{1}{2}|\nabla u|^2 g-\dd{u}\otimes \dd{u}+\chi.$$
	Let $w$, $v\in C^2(M) $ be admissible solutions to 
	\begin{equation}
		\begin{aligned}
			f(\lambda(g_{u}^{-1}W[u])) =\psi \mbox{ in } M. \nonumber
		\end{aligned}
	\end{equation}
	Suppose in addition that 
	$\lambda(g^{-1}W[w])\in \mathring{\Gamma}_{\mathcal{G}}^{f}$. Then
$
		w \equiv v \mbox{ in } M.
$
\end{lemma}

\begin{proof}

	Suppose $\sup_M(w-v)=(w-v)(x_0)>0$ for some $x_0\in M$. Then at $x_0$, 
	$\nabla v=\nabla w, \,\, \nabla^2 w \leq \nabla^2 v$ and then
	\begin{align*}
		W[w] \leq W[v].
	\end{align*} 
	Since $\lambda( g^{-1}W[w])\in \mathring{\Gamma}_{\mathcal{G}}^{f}$. 
	Combining with Lemma \ref{lemma3.4-3}, at $x_0 $ we have
	\begin{align*}
		f(\lambda( g_w^{-1}W[w])) < 	f(\lambda( g_v^{-1}W[w]))
		\leq  f(\lambda( g_v^{-1}W[v])).
	\end{align*}
	This is a contradiction. Hence $w-v\leq 0$ in $M.$
	
	Suppose now that $\inf_M(w-v)=(w-v)(x_0')<0$ for some $x_0'\in M$, where 
	\begin{align*}
		W[w] \geq W[v].
	\end{align*}
	Using Lemma \ref{lemma3.4-3} again, at $x_0' $ we have	\begin{align*}
		f(\lambda( g_w^{-1}W[w])) >f(\lambda( g_v^{-1}W[w]))
		\geq  f(\lambda( g_v^{-1}W[v])).
	\end{align*}
Again it is a contradiction. Hence $w-v\geq 0$ in $M$. 
In conclusion, $w\equiv v$. 
\end{proof}

Below, we prove existences more general than Theorem \ref{thm1-existence-closedmanifold}.
\begin{theorem} 
	\label{thm1-existence-closedmanifold-3}
	Suppose \eqref{concave}, \eqref{elliptic-weak-3},  \eqref{sumfi>0}, 
	\eqref{non-degenerate1},    \eqref{condition-key100-1} and 
	\eqref{assumption2-type2} hold.
	Then there is a smooth admissible conformal  metric $g_u=e^{2u}g$ satisfying 	 \eqref{main-equ0-Schouten}.  
\end{theorem}

\begin{proof}
The proof is somewhat standard.
	Applying Lemma \ref{lemma2-closed-construction} to the open cone $\mathring{\Gamma}_{\mathcal{G}}^f$, 
 there is a smooth function $\underline{u}$ such that $	\lambda(-g^{-1}A_{g_{\underline{u}}}) \in \mathring{\Gamma}_{\mathcal{G}}^f$ in $M.$
Without loss of generality, we may assume  $\underline{u}\equiv 0$, i.e.,
 $\lambda(-g^{-1}A_{g}) \in \mathring{\Gamma}_{\mathcal{G}}^f.$  
Fix
  $N>0$ with  $f(N\vec{\bf1})>\sup_{\partial\Gamma}f $. 
	Denote 
	\begin{equation}
		\label{def2-Wt}
		\begin{aligned}
			W^t[u]= \nabla^2 u	+\frac{1}{2}|\nabla u|^2 g
			- du\otimes du -tA_g+(1-t)N g. 
		\end{aligned}
	\end{equation}
	By \eqref{conformal-formula2}, $W^1[u]=-A_{g_u}$. 	Moreover, denote
	$$\psi^t:=t\psi+(1-t) f(N\vec{\bf1}).$$   	
	One can check that $\sup_{\partial\Gamma}f<\psi^t<\sup_\Gamma f$.
	Consider a family of equations
	\begin{equation}
		\label{equation-t-closed} 
		\begin{aligned}
			f(\lambda[g_u^{-1}W^t[u_t]])=\,& 
			\psi^t  \mbox{ in } M.
		\end{aligned} 
	\end{equation}
	
	Since $\lambda(-g^{-1}A_{g})\in \mathring{\Gamma}_{\mathcal{G}}^f$ and the cone $\mathring{\Gamma}_{\mathcal{G}}^f$ is of convexity, we see that
	\begin{align*}
		\lambda(g^{-1}[(1-t)Ng-tA_g]) \in \mathring{\Gamma}_{\mathcal{G}}^f, \,\, 
		\forall t\in [0,1]
	\end{align*}
	and then the $C^0$-estimate obtained in Proposition \ref{lemma-c0general} holds uniformly for the equations  \eqref{equation-t-closed}  for all $0\leq t\leq1.$   
	When picking $\eta\equiv1$ in the proof of interior estimates in Section \ref{sec4-estimate-local},
	together with the $C^0$-estimate, 
	we may derive 
	\begin{equation}
	\label{estimate-5alpha}
	\begin{aligned}|\nabla^2 u^t|\leq C, \mbox{ independent of } t. 		\end{aligned}
\end{equation}
	Using  Evans-Krylov theorem and Schauder theory, one can derive higher estimates
	\begin{equation}
		\label{estimate-4alpha}
		\begin{aligned}
			|u^t|_{C^{4,\alpha}(\bar{M})}\leq R_o, \mbox{ independent of } t 
		\end{aligned}
	\end{equation}
	for some   $0<\alpha<1$. 
	
	In what follows, we use 
	the degree theoretic argument
	to prove the existence.   
	Consider the operator
	\begin{equation}
		\begin{aligned}
			G(u,t)= f(\lambda[g_u^{-1}W^t[u]])-\psi^t.
		\end{aligned}
	\end{equation} 
	Fix $\alpha$ as in \eqref{estimate-4alpha}, and denote
	\begin{align*}
		\mathcal{O}_R=\{u\in C^{4,\alpha}(M): |u|_{C^{4,\alpha}(M)}<R\}.
	\end{align*}
	Obviously $\mathcal{O}_R$ is an open subset of $C^{4,\alpha}(M)$. By \eqref{estimate-4alpha} if $R>R_o+1$ then $G(u,t)=0$ has no solution with $\lambda(g^{-1}W^t[u])\in \Gamma$ on $\partial\mathcal{O}_R$. Therefore, following \cite{LiYY1989}  one can define the degree $\deg (G(\cdot,t),\mathcal{O}_R,0)$ for all $0\leq t\leq 1$, and 
	\begin{equation}
		\label{homotopic-invariance1}
		\begin{aligned} 
			\deg (G(\cdot,0),\mathcal{O}_R,0)=\deg (G(\cdot,1),\mathcal{O}_R,0).
		\end{aligned}
	\end{equation}
	
	Next we compute $\deg (G(\cdot,0),\mathcal{O}_R,0).$
	Take $t=0$, the equation reads as follows
	\begin{equation}
		\label{equation-0-closed} 
		\begin{aligned}
			f(\lambda[g_{{u}}^{-1}(\nabla^2 {u}	+\frac{1}{2}|\nabla {u}|^2 g-d{u}
			\otimes d{u}+N g)])=
			f(N\vec{\bf1}). 
		\end{aligned} 
	\end{equation} 
	Clearly $u_0\equiv0$ is a  
	solution to \eqref{equation-0-closed}.
	Next we need to compute $\deg (G(\cdot,0),\mathcal{O}_R,0)$, which is
	the Leray–Schauder degree of the solution $u_0\equiv0$ of
	\eqref{equation-0-closed}.
	
	\vspace{1mm} 
	\noindent {\em Uniqueness}:
	By  the maximum principle or
	Lemma \ref{lemma3-unique}, we can verify that 
	$u_0\equiv 0$ is the unique admissible solution of  \eqref{equation-0-closed}.
	
	\vspace{1mm} 
	\noindent {\em Invertibility}:
	We need to find the linearization of 
	\eqref{equation-0-closed}  at the unique solution 
	$u_0\equiv0$: 
	Let $\mathfrak{g}^t[u]=e^{-2u}W^t[u]$, and
	let 	$\mathfrak{L}_{u_0}v  
	= \frac{d}{ds}  \left(f(\lambda(g^{-1}\mathfrak{g}^0[u_{(s)}]))-f(N\vec{\bf1})\right)\Big|_{s=0}$, where $u_{(s)}= u_0+sv$. 
	Denote $F^{ij}_0=\frac{\partial f(\lambda[g^{-1}\mathfrak{g}^0[u]]) }{\partial \mathfrak{g}_{ij}}\Big|_{u=u_0\equiv0}
	$. Then  the straightforward computation gives the following
	\begin{equation}
		\begin{aligned} 
			\mathfrak{L}_{u_0}v 
			= F^{ij}_0 \frac{d}{d s} \left( e^{-2u_{(s)}}W_{ij}^0[u_{(s)}]  \right)\Big|_{s=0}
			=f_1(N\vec{\bf1})(\Delta v-2nN v ).
		\end{aligned}
	\end{equation}
	By the maximum principle, $\mathfrak{L}_{u_0}$ has trivial kernel.
	So the operator $\mathfrak{L}_{u_0}$ is invertible.  
	
	From the discussion above, we derive $ \deg (G(\cdot,0),\mathcal{O}_R,0)=\pm 1.$
	Together with \eqref{homotopic-invariance1},
	$$ \deg (G(\cdot,1),\mathcal{O}_R,0)=\pm 1.$$
	Therefore the equation \eqref{equation-t-closed}  with $t=1$ has a solution $u$ with $\lambda(-g^{-1}A_{g_{u}})\in\Gamma.$ 
	
	
	
\end{proof}

When $\psi$ satisfies \eqref{non-degenerate2}, we have the following existence and uniqueness. 
\begin{theorem} 
	\label{thm1-existence-closedmanifold-4}
	Suppose \eqref{concave}, \eqref{elliptic-weak-3},  \eqref{sumfi>0}, 
	\eqref{non-degenerate2} and    
	\eqref{assumption2-type2} hold.
	Then there is a unique smooth admissible conformal  metric $g_u=e^{2u}g$ satisfying 	 \eqref{main-equ0-Schouten}.  
\end{theorem}
\begin{proof}
Pick 	  $f(N\vec{\bf1})>\sup_{\partial \Gamma_{\mathcal{G}}^f} f$, and we 	consider the continuity path \eqref{equation-t-closed}.
	The proof  is based on the  continuity method (see e.g. \cite{Gursky2003Viaclovsky}) and Lemma \ref{lemma3-unique}.  
	The detail  is somewhat similar to that of Theorem \ref{thm2-finiteBVC} below, and we omit it here.
\end{proof}

\section{Dirichlet problem: boundary estimate}
\label{sec-DirichletProblem}


In this section we assume that $(\bar{M},g)$ is a smooth Riemannian manifold with smooth boundary.  
We consider the Dirichlet problem 
\begin{equation}
	\label{Dirichlet2-main-equ2-Schouten}
	\left\{ 
	\begin{aligned}  f(\lambda(g_u^{-1} W[u])) =\,& \psi  \,& \mbox{ in } \bar{M},
		\\
		u=\,& \varphi \,& \mbox{ on } \partial M.
	\end{aligned} 
\right.
\end{equation}
Here, as in \eqref{def1-W-tensor}, we denote
 $W[u]=\nabla^2 u	+\frac{1}{2}|\nabla u|^2 g-\dd{u}\otimes \dd{u}+\chi$.


We first prove comparision principle. The proof is similar to that of Lemma \ref{lemma3-unique}.
\begin{lemma}
	[Comparision principle]
	\label{lemma1-comparision}
	Let $w$, $v\in C^2(M)\cap C^0(\bar{M})$ be admissible functions satisfying 
	\begin{equation}
		\begin{aligned}
			f(\lambda(g_{w}^{-1}W[w])) \geq \,&
			f(\lambda(g_{v}^{-1} W[v])) \,& \mbox{ in } M, \\
			w\leq \,& v  \,& \mbox{ on } \partial M. \nonumber
		\end{aligned}
	\end{equation}
	Suppose in addition that either
	$\lambda(g^{-1}W[w])\in \mathring{\Gamma}_{\mathcal{G}}^{f}$ or 
	$\lambda( g^{-1}W[v])\in\mathring{\Gamma}_{\mathcal{G}}^{f}$. Then
	\begin{align*}
		w \leq v \mbox{ in } M.
	\end{align*}
\end{lemma}

\begin{proof}

	Suppose $\sup_M(w-v)=(w-v)(x_0)>0$ for some $x_0\in M$. Then at $x_0$, 
	$\nabla v=\nabla w, \,\, \nabla^2 w \leq \nabla^2 v$ and then
$
		W[w] \leq W[v].
$
	
	Case 1: $\lambda( g^{-1}W[w])\in \mathring{\Gamma}_{\mathcal{G}}^{f}$. 
	Using Lemma \ref{lemma3.4-3}, we have
	\begin{align*}
		f(\lambda( g_w^{-1}W[w])) < 	f(\lambda( g_v^{-1}W[w]))
		\leq  f(\lambda( g_v^{-1}W[v])).
	\end{align*}
	This is a contradiction. 
	
	Case 2: $\lambda( g^{-1}W[v])\in \mathring{\Gamma}_{\mathcal{G}}^{f}$. 
	Then by Lemma \ref{lemma3.4-3}
	\begin{align*}
		f(\lambda( g_v^{-1}W[v])) 
		> 		f(\lambda( g_w^{-1}W[v]))
		\geq  		
		f(\lambda( g_w^{-1}W[w])).
	\end{align*}
	Again it is a contradiction. 
\end{proof}

As a  consequence, we obtain the unique result. 
\begin{proposition}
	\label{pro1-unique}
	Assume that the Dirichlet problem \eqref{Dirichlet2-main-equ2-Schouten} admits a $C^2$ solution $u$ with 
	$\lambda(g^{-1}W[u])\in \mathring{\Gamma}_{\mathcal{G}}^{f}$. Then 
	$u$ is the unique admissible solution the Dirichlet problem.
\end{proposition}

This will help us to prove the existence result using degree theory.

 \subsection{Boundary gradient estimate}
 \label{subsect1-Boundarygradient}

 As in \eqref{distance-function}, $\mathrm{\sigma}(x)$ denotes the distance from $x$ to boundary.
We know $\mathrm{\sigma}(x)$ is smooth near $\partial M$.
Denote 
\begin{equation}
	\label{omega-delta}
	\Omega_\delta:=\{x\in M: \mathrm{\sigma}(x)<\delta\}.
\end{equation}  
We use  $\mathrm{\sigma}$ to construct local barriers. 

\vspace{1mm}
\noindent{\bf Local upper barrier}. 
We use the local upper barrier constructed in earlier work \cite{yuan-PUE-conformal}. For convenience, 
take 
$\bar{w}=\frac{1}{2(n-2)} \log(1+\frac{\mathrm{\sigma}}{\delta^2})
+\varphi.$ 
By straightforward computation and Cauchy-Schwarz inequality, one can see that 
\begin{equation}
	\begin{aligned}
		\mathrm{tr}(g^{-1}W[\bar w])  
		\leq \,&
		-\frac{1}{4(n-2)} \frac{|\nabla\sigma|^2}{(\delta^2+\sigma)^2}
		+\frac{\Delta \sigma}{2(n-2)(\delta^2+\sigma)}
		\\\,&
		+\Delta\varphi 
		+(n-2)|\nabla\varphi|^2
		+	\mathrm{tr}(g^{-1}\chi) <0
	\end{aligned}
\end{equation} 
provided $0<\delta\ll1$. Here we  use $|\nabla\sigma|=1$ on $\partial M$.
Notice also that 
$\underset{\delta\to 0^+}{\lim}\, \bar{w}\big|_{\mathrm{\sigma}=\delta} =+\infty.$
Combining the $C^0$-estimate in  Proposition \ref{lemma-c0general},  when   $\delta$ is sufficiently small, 
\begin{equation} 	\label{upper-barrier1}
	\begin{aligned}
		u\leq \bar{w} \mbox{ on } \Omega_{\delta}.  
	\end{aligned}
\end{equation}

\noindent{\bf Local lower barrier}. 
Fix $k\geq 1$. 
Inspired by \cite{Guan2008IMRN} we take
\begin{equation}
	\label{h-def-k}
	\begin{aligned}
		w= \log \frac{\delta^2}{ \delta^2+k\mathrm{\sigma}}+\varphi.   
	\end{aligned}
\end{equation}
The straightforward computation gives that 
\begin{equation}
	\label{compute-h-c1c2}
	\begin{aligned}
		\nabla w=-\frac{k \nabla \mathrm{\sigma}}{\delta^2+k\mathrm{\sigma}}+\nabla\varphi,\,\,
		\nabla^2 w= \frac{k^2  \dd{\mathrm{\sigma}}\otimes \dd{\mathrm{\sigma}}  }{(\delta^2+k\mathrm{\sigma})^2}-\frac{k \nabla^2 \mathrm{\sigma}}{\delta^2+k\mathrm{\sigma}}+\nabla^2\varphi.   
	\end{aligned}
\end{equation}  
By straightforward computation 
\begin{equation}
	\begin{aligned}  
		W[w] =\,& \frac{1}{2}
		\left[\frac{k^2  |\nabla\sigma|^2}{(\delta^2+k\sigma)^2}
		-\frac{2k}{\delta^2+k\sigma}
		\langle\nabla\varphi,\nabla\sigma\rangle
		+|\nabla\varphi|^2
		\right]\cdot g
		\\\,&
		+
		\frac{k }{\delta^2+k\sigma}
		\left[\dd{\varphi}\otimes \dd{\sigma}+\dd{\sigma}\otimes \dd{\varphi} -\nabla^2\sigma \right] 
		+\nabla^2\varphi-\dd{\varphi}\otimes \dd{\varphi}+
		\chi.
	\end{aligned}
\end{equation} 
In particular, if $\varphi$ is constant then
\begin{equation}
	\begin{aligned}  
		W[w] =  \frac{1}{2}
		\frac{k^2  |\nabla\sigma|^2}{(\delta^2+k\sigma)^2} 
		\cdot g 
		-\frac{k  \nabla^2\sigma}{\delta^2+k\sigma}  + 	\chi.
	\end{aligned}
\end{equation}

Note that $|\nabla\mathrm{\sigma}|=1$ on $\partial M$.
Take  $\delta$ small enough such that 
\begin{equation}	
	\begin{aligned}
		W[w] \geq  	  \frac{k^2|\nabla\sigma|^2}{4(\delta^2+k\sigma)^2} \cdot g.
	\end{aligned}
\end{equation} 
Therefore, for $0<\delta\ll1$ we have on $\Omega_\delta$
\begin{equation}
	\begin{aligned}
		f\left(\lambda(g_{w}^{-1} W[w]))\right) \geq  
		f\left( e^{-2w}\frac{ k^2 
			|\nabla\mathrm{\sigma}|^2 }{4(\delta^2+k\mathrm{\sigma})^2}
		\vec{\bf 1}\right)
		= 
		f\left( e^{-2\varphi}\frac{  k^2  	|\nabla\mathrm{\sigma}|^2}{4\delta^{4 }} 
		\vec{\bf 1}\right)>\psi.  \nonumber
	\end{aligned}
\end{equation} 
It follows from Lemma \ref{lemma1-comparision}, 
 Proposition \ref{lemma-c0general} and 
$\underset{\delta\to 0^+}{\lim}\,	w\big|_{\mathrm{\sigma}=\delta} =-\infty$
that 
\begin{equation}
	\label{low-barrier1}
	\begin{aligned}
		u\geq w
		\mbox{ on } \Omega_{\delta}, \mbox{ for some } 0<\delta\ll1.
	\end{aligned}
\end{equation}  

From \eqref{low-barrier1} (setting $k=1$) and \eqref{upper-barrier1} we derive
\begin{equation}
	\begin{aligned}
		|\nabla u|\leq C \mbox{ on } \partial M. \nonumber
	\end{aligned}
\end{equation}


\subsection{Boundary second estimate}
 \label{subsect1-Boundaryhessian}

For a point $x_0\in\partial M$ we choose a smooth  local orthonormal frame $e_1, \cdots, e_n$ around $x_0$ such that
$e_n$ is unit outer normal vector field when restricted to $\partial M$.   
We denote $\rho(x):=\mathrm{dist}_g(x,x_0)$  the distance function from $x$ to $x_0$
with respect to $g$, 
and $M_\delta:=\{x\in M: \rho(x)<\delta\}.$  
We know that  
$\mathrm{\sigma}(x)$ is smooth and $|\nabla \mathrm{\sigma}|\geq \frac{1}{2}$ in $M_\delta$ for small $\delta$.
Let $\lambda$ be as in 
\eqref{def1-lambda1}
\begin{align*}
	\lambda=\lambda(g^{-1}\mathfrak{g}[u])=\lambda(g_u^{-1}W[u]).
\end{align*}

\subsubsection*{Case 1. Pure tangential derivatives.}   
From the boundary value condition, 
\begin{equation}
	\label{morry-1}
	\begin{aligned}
		\nabla_\alpha u=\nabla_\alpha \varphi, 
		\,\, \nabla_{\alpha\beta}u=\nabla_{\alpha\beta}\varphi+\nabla_{n}(u-\varphi)\mathrm{II}(e_\alpha,e_\beta)
		\mbox{ on } \partial M  \nonumber
	\end{aligned}
\end{equation}
for $1\leq \alpha, \beta\leq n-1$, where $\mathrm{II}$ is the fundamental form of $\partial M$. 
This gives the bound of
second estimates for pure tangential derivatives
\begin{equation}
	\label{ineq2-bdy}
	\begin{aligned}
		|\mathfrak{g}_{\alpha\beta}(x_0)|\leq C.
	\end{aligned}
\end{equation}

\subsubsection*{Case 2. Mixed derivatives.}
For $1\leq\alpha<n$, the local barrier function 
is given by
\begin{equation}
	\begin{aligned}
		\Psi=\pm \nabla_\alpha (u-\varphi) 
		+A_1 \left(\frac{N \mathrm{\sigma}^2}{2} -t \mathrm{\sigma}\right)
		-A_2\rho^2+A_3 \sum_{k=1}^{n-1} |\nabla_k (u-\varphi)|^2 
		\mbox{ in } M_\delta,  \nonumber
	\end{aligned}
\end{equation}
where $A_1$, $A_2$, $A_3$, $N$, $t$ are all positive constants to be determined, and $N\delta-2t\leq 0.$
Such a construction 
is somewhat standard. 

To deal with local barrier functions, we need the following formula 
\begin{equation}
	\label{ineq1-bdy}
	\begin{aligned}
		\nabla_{ij}(\nabla_k u)=\nabla_{ijk}u+\Gamma_{ik}^l\nabla_{jl}u+\Gamma_{jk}^l\nabla_{il}u+\nabla_{\nabla_{ij}e_k}u,\nonumber
	\end{aligned}
\end{equation}
see e.g. \cite{Guan12a}.
Using \eqref{diff1-equ}  and \eqref{inequ27}
\begin{equation}
	\begin{aligned}
		F^{ij}(\nabla_{kij}u+W_{ij,p_l}\nabla_{lk}u)
		= e^{2u}\nabla_k\psi+2\nabla_ku
		F^{ij}W_{ij} -F^{ij}\nabla_k\chi_{ij}.
	\end{aligned}
\end{equation}
Thus we obtain
\begin{equation}
	\label{bdy-1}
	\begin{aligned}
		|\mathcal{L}(\nabla_k (u-\varphi))|\leq C\left(1+\sum_{i=1}^n f_i+\sum_{i=1}^n f_i |\lambda_i|\right).  \nonumber
	\end{aligned}
\end{equation}
As in \eqref{Lw1}, we can prove
\begin{equation}
	\label{bdy-2}
	\begin{aligned}
		\mathcal{L}(|\nabla_k (u-\varphi)|^2)
		\geq 
		e^{4u}F^{ij}\mathfrak{g}_{ik}\mathfrak{g}_{jk} -C\left(1+\sum_{i=1}^n f_i+\sum_{i=1}^n f_i |\lambda_i|\right).  \nonumber
	\end{aligned}
\end{equation}
By \cite[Proposition 2.19]{Guan12a}, there exists an index $1\leq r\leq n$ with
\begin{equation}
	\begin{aligned}
		\sum_{k=1}^{n-1} F^{ij}\mathfrak{g}_{ik}\mathfrak{g}_{jk} \geq \frac{1}{2} \sum_{i\neq r} f_i\lambda_i^2.  \nonumber
	\end{aligned}
\end{equation}

Let $\tilde{C}_\sigma$ be as in \eqref{def1-c-sigma} below. 
It follows from \eqref{concave-1} that 
$$ \sum_{i=1}^n f_i \lambda_i \leq \tilde{C}_\psi\sum_{i=1}^nf_i.$$ 
On the other hand, from \eqref{condition-key100-3} in Proposition \ref{lemma-Y7}, 
\begin{equation}
	\begin{aligned}
		-\sum_{i=1}^n f_i\lambda_i\leq  K_0\sum_{i=1}^n f_i.  \nonumber
	\end{aligned}
\end{equation}
Using Cauchy-Schwarz inequality and
$$\sum_{i=1}^n f_i |\lambda_i|= \sum_{i=1}^n f_i\lambda_i -2\sum_{\lambda_i<0} f_{i} \lambda_i =2\sum_{\lambda_i\geq 0} f_i\lambda_i -\sum_{i=1}^n f_{i} \lambda_i,$$ 
we have 
\begin{equation}
	\label{flambda}
	\begin{aligned}
		\sum_{i=1}^n f_i |\lambda_i|
		\leq   \epsilon \sum_{i\neq r} f_i\lambda_i^2  +\left(\frac{1}{\epsilon}+\max\{\tilde{C}_\psi,K_0\}\right)\sum_{i=1}^n f_i, \,\, \forall \epsilon>0.  
	\end{aligned}
\end{equation} 

Let $\kappa$ be as in \eqref{sumfi-1}. 
If $\delta $ and $t$ are chosen small enough such that 
\begin{equation}
	\begin{aligned}
		\mathcal{L}\left(\frac{N \mathrm{\sigma}^2}{2}-t \mathrm{\sigma}\right) 
		\geq  
		\frac{N \kappa \theta}{8(1+\kappa)}
		\left(1+\sum_{i=1}^n f_i\right).  \nonumber
	\end{aligned}
\end{equation}

Putting those inequalities together, if $A_1\gg A_2$, $A_1\gg A_3>1$, and $\epsilon$ in \eqref{flambda} is sufficiently small, then 
\begin{equation}
	\begin{aligned}
		\mathcal{L}(\Psi) 
		>0 \mbox{ in } M_\delta.  \nonumber
	\end{aligned}
\end{equation}
On the other hand, $\Psi|_{\partial M}=-A_2\rho^2\leq 0$; while 
$\Psi\leq 0$  on $\partial M_\delta\setminus \partial  M$  if  $N\delta-2t\leq 0$, $A_2\gg A_1$.
Note that $\Psi(x_0)=0$.  
Thus
\begin{equation}
	\label{mix-1}
	\begin{aligned}
		|\mathfrak{g}_{\alpha n}(x_0)|\leq C \mbox{ for } 1\leq \alpha\leq n-1.
	\end{aligned}
\end{equation}

\subsubsection*{Case 3. Double normal derivatives.}
Fix $x_0\in\partial M$.  
Since $\mathrm{tr}(g^{-1}\mathfrak{g})>0$,  \eqref{ineq2-bdy} and \eqref{mix-1}, we have $\mathfrak{g}_{nn}\geq -C$. 
We assume  
$\mathfrak{g}_{nn}(x_0)\geq 1$ (otherwise we are done).
According to \eqref{fully-uniform2},  
$F^{nn}\geq \theta \sum_{i=1}^n f_i. $ 
By the concavity of equation,   \eqref{ineq2-bdy} and \eqref{mix-1}, 
\begin{equation}
	\begin{aligned}
		\psi(x_0)-F(g) 
		\geq 
		F^{ij} (\mathfrak{g}_{ij}-\delta_{ij})  
		\geq 
		-C'\sum_{i=1}^n f_i (\lambda)+  (\theta\mathfrak{g}_{nn}-1)  \sum_{i=1}^n f_i(\lambda).  \nonumber
	\end{aligned}
\end{equation}
This gives $\mathfrak{g}_{nn}(x_0) \leq C.$


\section{Equations on manifolds with boundary}
\label{sec-existence-Dirichlet}

Throughout this section, we assume that $(\bar{M},g)$ is a smooth, compact,  connected Riemannian manifold with smooth boundary $\partial M$. 
Suppose in addition that one of the following assumptions holds
\begin{itemize}
	\item  $\lambda(-g^{-1}A_g)\in \bar{\Gamma}_{\mathcal{G}}^{f}$ in $\bar{M}$. 
	
	\item 
	 $(1,\cdots,1,-1)\in\bar{\Gamma}_{\mathcal{G}}^{f}$.
\end{itemize}
 Each of them is imposed as a key condition in Lemma \ref{lemma1-construct0} to  construct    $\mathring{\Gamma}_{\mathcal{G}}^f$-admissible conformal metric, using   some technique from Morse theory.

\subsection{The Dirichlet problem}
As above, denote  $g_u=e^{2u}g$.
In this subsection we consider the Dirichlet problem  
\begin{equation}
	\label{Dirichlet3-main-equ2-Schouten}
	\left\{ 
	\begin{aligned}  f(\lambda[-g_u^{-1}A_{g_u}]) =\,& \psi  
		 \,& \mbox{ in } \bar{M},
		\\
		u=\,& \varphi \,& \mbox{ on } \partial M.
	\end{aligned} 
	\right.
\end{equation}

We prove that 

\begin{theorem}
	\label{thm1-finiteBVC} 
	In addition to \eqref{concave}, \eqref{elliptic-weak-3}, \eqref{sumfi>0}, \eqref{non-degenerate1} and \eqref{assumption2-type2}, we assume   that \eqref{condition-key100-1} holds for some $\sup_{\partial\Gamma}f<\sigma\leq \inf_M\psi.$ 
	Let $\varphi\in C^\infty(M)$. 
	Then there is a smooth admissible conformal metric $g_u$ satisfying 	 \eqref{Dirichlet3-main-equ2-Schouten}.

\end{theorem}

When $\psi$ satisfies \eqref{non-degenerate2}, we obtain the uniqueness result. 
\begin{theorem}
		\label{thm2-finiteBVC} 
Suppose that all the assumptions in Theorem \ref{thm1-finiteBVC}  hold.
In addition, we assume \eqref{non-degenerate2} holds.
 Then \eqref{Dirichlet3-main-equ2-Schouten} admits a unique smooth admissible solution $u$.
 Moreover,	$\lambda(-g^{-1}A_{g_u})\in \mathring{\Gamma}_{\mathcal{G}}^f.$
\end{theorem}

In the proofs of Theorems \ref{thm1-finiteBVC} and \ref{thm2-finiteBVC}, applying Lemma \ref{lemma1-construct0} to the open cone $\mathring{\Gamma}_{\mathcal{G}}^f$, without loss of generality we may assume $\underline{u}\equiv0$ and then
\begin{equation}
	\begin{aligned} 
		\lambda(-g^{-1}A_g)\in\mathring{\Gamma}_{\mathcal{G}}^f.
	\end{aligned}
\end{equation}

\begin{proof}
	[Proof of Theorem \ref{thm1-finiteBVC}] We use  degree theory to prove existence.	
	Up to certain conformal deformation, we   assume 
	$\varphi\equiv 0$. 
Take $\eta\equiv1$ in the proof of interior estimate, we have global gradient and Hessian estimates.
The rest of proof is similar to that of Theorem \ref{thm1-existence-closedmanifold-3}, and also to that of \cite{Guan2007AJM,Guan2008IMRN,Li2011Sheng}.
And we omit the details. 
\end{proof}

\begin{proof}
	[Proof of Theorem \ref{thm2-finiteBVC}]

	We use the continuity method. 
	Denote $\psi_0=f(\lambda(-g^{-1}A_g))$ and $\psi^t=t\psi  +(1-t)\psi_0$. 
	By the assumptions \eqref{non-degenerate2} and $\lambda(-g^{-1}A_g)\in\mathring{\Gamma}_{\mathcal{G}}^f$,  there exists some $\delta_0>0$ such that 
\begin{equation} 
	\label{nondegenerate-3t}
	\begin{aligned}
	\sup_{\partial \Gamma_{\mathcal{G}}^f} f+\delta_0< \psi^t<\sup_\Gamma f, \,\, \forall 0\leq t\leq 1.
	\end{aligned}
\end{equation} 
	For $t\in [0,1]$, we consider the Dirichlet problems \begin{equation}
		\label{Dirichlet3-main-equt-Schouten}
		\left\{ 
		\begin{aligned}  f(\lambda[-g_u^{-1}A_{g_u}]) =\,& \psi^t\,& \mbox{ in } \bar{M},
			\\
			u=\,& t\varphi \,& \mbox{ on } \partial M.
		\end{aligned}   
		\right.
	\end{equation}
Note that $\lambda(-g^{-1}A_g)\in \mathring{\Gamma}_{\mathcal{G}}^{f}$. Denote $\mathrm{S}\subseteq [0,1]$ by
\begin{align*}
	\mathrm{S}=\left\{t: \eqref{Dirichlet3-main-equt-Schouten} \textrm{  is uniquely solvable in class of  $C^{\infty}$-functions with  } \lambda[-g^{-1}A_{g_u}]\in    \mathring{\Gamma}_{\mathcal{G}}^{f}\right\}.
\end{align*}
When $t=0$,  the equation has a unique solution $u\equiv0$ according to Proposition \ref{pro1-unique}.
 This is the starting point of the continuity path along \eqref{Dirichlet3-main-equt-Schouten}.
By  implicit function theorem 
and the openness of $\mathring{\Gamma}_{\mathcal{G}}^{f}$, we know that $\mathrm{S}$ is open. 
(Through linear elliptic equation theory and Lemma \ref{lemma3.4-3} or \ref{lemma2-key}, we may prove that the Fr\'echet derivative is an isomorphism, thereby justifying the use of the implicit function theorem).

Next we prove that $\mathrm{S}$ is closed.
Let $t_i\in \mathrm{S}$ and $t_0=\underset{i\to +\infty}{\lim} t_i$.
 We remark  that \eqref{nondegenerate-3t}  ensures   $\lambda[-g^{-1}A_{g_{u^{t_0}}}]\in    \mathring{\Gamma}_{\mathcal{G}}^{f}$, and by Lemma \ref{lemma2-key} $\sum f_i(\lambda)\lambda_i\geq0$ for $\lambda=\lambda[-g_{u^{t_i}}^{-1}A_{g_{u^{t_i}}}]$. 
 Thus 
 the estimates established in this paper works.  
Combining with Proposition \ref{pro1-unique},  $u=u^{t_0}$ is the unique solution to the Dirichlet problem  with $\lambda[-g^{-1}A_{g_{u}}]\in    \mathring{\Gamma}_{\mathcal{G}}^{f}$.
  Thus we conclude that  $\mathrm{S}$ is closed and so $\mathrm{S}=[0,1]$.

\end{proof}

\subsection{The Dirichlet problem with infinity boundary value}

We consider the Dirichlet problem with infinity boundary data
\begin{equation}
	\label{Dirichlet2-infinity-main-equ0-Schouten}
 \left\{
	\begin{aligned} 
		f(\lambda(-g_{u}^{-1}A_{g_{u}})) = \,& \psi,
	 \,\,\lambda(-g^{-1}A_{g_{u}})\in \Gamma,   \mbox{ in }   M,  \\
\lim_{x\to\partial M}{u}(x)	= \,& +\infty.   
	\end{aligned} 
 \right.
\end{equation}

\begin{theorem}
	\label{existence1-compact-2}

	Suppose, in addition to \eqref{concave}, \eqref{elliptic-weak-3}, \eqref{sumfi>0}, \eqref{non-degenerate1} and \eqref{assumption2-type2},  that \eqref{condition-key100-1} holds for some $\sup_{\partial\Gamma}f<\sigma\leq \inf_M\psi.$  
	Then there is a smooth admissible function ${u_\infty}$ (defined on the interior $M$ of $\bar{M}$) 
	satisfying \eqref{Dirichlet2-infinity-main-equ0-Schouten}.
	
If, in addition,  	\eqref{non-degenerate2} holds then the solution    satisfies 
	$\lambda(-g^{-1}A_{g_{u_\infty}})\in \mathring{\Gamma}_{\mathcal{G}}^f$. Moreover,   \eqref{condition-key100-1} can be dropped in this case.
\end{theorem}

\subsubsection{Local  bound of $u$ from above}

To achieve this, we use some result of 
Aviles-McOwen \cite{Aviles1988McOwen}, extending some  famous result of Loewner-Nirenberg
\cite{Loewner1974Nirenberg}.  They show that $(M,g)$ admits a   complete conformal metric $g_{\tilde{u}}=e^{2\tilde{u}}g$ with negative constant scalar curvature $-1$. Namely, there is $\tilde{u}\in C^\infty(M)$ such that
\begin{equation}	\label{scalar-equ1}
	\begin{aligned}
		\mathrm{tr}(-g^{-1}A_{g_{\tilde{u}}}) =   \frac{1}{2(n-1)}e^{2\tilde{u}}   \mbox{ in } M, \,\,
		\lim_{x\to\partial M} \tilde{u} = +\infty.
	\end{aligned}
\end{equation}
See also further results on semilinear elliptic equations of McOwen \cite{McOwen1993}.

%

\begin{lemma}
	\label{lemma1-C0-local-upper}
	Let $u$ be an admissible solution to the Dirichlet problem. Then 
	\begin{equation}
		u\leq \tilde{u}-\frac{1}{2}\log(2n(n-1)\epsilon_0) \mbox{ in } M.
	\end{equation}
	Here  $\epsilon_0$ is a positive constant satisfying 
	$ f(\epsilon_0\vec{\bf 1})\leq \inf_M \psi.$
\end{lemma}

\begin{proof}
	
	  $u-\tilde{u}$ must achieve its maximum at an interior 
	point  $x_0\in M$, namely
	\begin{equation}
		\begin{aligned}
			(u-\tilde{u}) (x_0)= \sup_M (u-\tilde{u}).
		\end{aligned}
	\end{equation}
	At such an interior point $x_0$, we have 
	$\nabla u =\nabla \tilde{u}, \,\, \nabla^2 u\leq \nabla^2\tilde{u}$ 
	and so  
	\begin{equation}
		\begin{aligned}
			\mathrm{tr}(-g^{-1}A_{g_u}) \leq 	\mathrm{tr}(-g^{-1}A_{g_{\tilde{u}}}).
		\end{aligned}
	\end{equation}	 

On the other hand, 
from \eqref{concave-1} we get
	\begin{equation}	\label{key2-main}	\begin{aligned}		\sum_{i=1}^n \lambda_i\geq 		n\epsilon_0 + A_{f,\epsilon_0} \left(f(\lambda)-f(\epsilon_0 \vec{\bf 1})\right), \mbox{ }\forall \lambda\in\Gamma,  \nonumber	\end{aligned}\end{equation}
 where   $ A_{f,\epsilon_0}= {n}\left(\sum_{i=1}^n f_i(\epsilon_0 \vec{\bf 1})\right)^{-1}$,  
and so 
	\begin{equation} 
		\begin{aligned}
			\mathrm{tr}(-g^{-1}A_{g_{u}})  \geq 
			e^{2u} \left[ n\epsilon_0+ A_{f,\epsilon_0} \left(\psi-f(\epsilon_0\vec{\bf 1})\right) \right] \geq  e^{2u+\log (n\epsilon_0)}.
		\end{aligned}
	\end{equation} 
	Thus we can derive that $(u-\tilde{u})(x_0)\leq -\frac{1}{2}\log(2n(n-1)\epsilon_0).$

\end{proof}

\subsubsection{Proof of Theorem \ref{existence1-compact-2}}

 By Lemma \ref{lemma3.4-2} and the construction of admissible conformal metric in Lemma  \ref{lemma1-construct0}  
 (applying it to the 
cone $\mathring{\Gamma}_{\mathcal{G}}^f$), there is a smooth conformal metric 
$g_{\underline{u}}=e^{2\underline{u}}g$, up to certain rescaling, such that
\begin{equation}
	\begin{aligned}
		f(\lambda(-g_{\underline{u}}^{-1}A_{g_{\underline{u}}})) \geq \,& \psi, \,\, \lambda(-g_{\underline{u}}^{-1}A_{g_{\underline{u}}}) \in \mathring{\Gamma}_{\mathcal{G}}^f\mbox{ in } \bar{M}.
	\end{aligned}
\end{equation}

We consider the Dirichlet problem
\begin{equation}
	\label{mainequ-k}
	\begin{aligned}
		f(\lambda(-g_{u^{(k)}}^{-1}A_{g_{u^{(k)}}})) =\,& \psi \mbox{ in } \bar{M}, \,\,
		u^{(k)}\big|_{\partial M}= \log k.
	\end{aligned}
\end{equation}
The solvability follows from Theorem   \ref{thm1-finiteBVC}. 
By the comparision principle (Lemma \ref{lemma1-comparision}), we deduce that 
\begin{equation}
	\begin{aligned}
	u^{(k)} \geq \underline{u} \mbox{ in } \bar{M}, \mbox{ whenever  } \log k\geq \sup_{\partial M}\underline{u}.
	\end{aligned}
\end{equation}
Namely, $u^{(k)}$ has a uniform lower bound. On the other hand, we have a local upper bound for $u^{(k)}$ asserted in Lemma \ref{lemma1-C0-local-upper}. Together with interior estimate for  first and second derivatives in Section \ref{sec4-estimate-local}, for any compact subset $K$ of $M$, we have 
\begin{equation}
	\begin{aligned}
	|u^{(k)} |_{C^2(K)} \leq C(K), \mbox{ independent of } k.
	\end{aligned}
\end{equation}
Using Evans-Krylov theorem and Schauder theory, we have 
\begin{equation}
	\begin{aligned}
		|u^{(k)} |_{C^{l,\alpha}(K)} \leq C(K,l,\alpha), \mbox{ independent of } k.
	\end{aligned}
\end{equation}
for some $l\geq 2$ and $0<\alpha<1$. Utilizing the diagonal process, we can obtain a smooth admissible solution ${u_\infty}$ to \eqref{Dirichlet2-infinity-main-equ0-Schouten}.

 This completes the proof of   Theorem \ref{existence1-compact-2}.

\subsection{Proof of Theorem \ref{thm2-completemetric} and completeness of resulting metrics}

We consider the Dirichlet problem \eqref{mainequ-k}. 
Let ${u_\infty}$ be the solution constructed as 
above. 

It suffices to prove the completeness of $g_{u_\infty}=e^{2u_{\infty}}g$.
 From \eqref{low-barrier1} we have
\begin{equation}
	\label{low-barrier1-2}
	\begin{aligned}
		u^{(k)}\geq \log \frac{k\delta^2}{\delta^2+k\sigma}
		\mbox{ on } \Omega_{\delta}, \mbox{ for some } 0<\delta\ll1.
	\end{aligned}
\end{equation} 
(Notice $w=\log \frac{k\delta^2}{\delta^2+k\sigma}$ in \eqref{h-def-k}, i.e.,  $\varphi=
\log k$).
So near the boundary, we have 
\begin{align*}
	{u_\infty} \geq -\log\sigma+\log\delta^2 \mbox{ on } \Omega_{\delta}, \mbox{ for some } 0<\delta\ll1.
\end{align*}
This implies that the conformal metric $g_{u_\infty}=e^{2u_{\infty}}g$ is complete in the interior.


\section{Asymptotic behavior and uniqueness  of solutions to fully nonlinear Loewner-Nirenberg problem}
\label{sec-Asymp-Uniquess}

For the conformal scalar curvature equation \eqref{scalar-equ1} 
on a smooth bounded domain  $M=\Omega\subset\mathbb{R}^n$, 
Loewner-Nirenberg \cite{Loewner1974Nirenberg} 
proved  
the asymptotic ratio 
\begin{equation}
	\label{asymptotic-rate1-002}
	\begin{aligned}
		\lim_{x\to\partial M}\left(\tilde{u}(x)+\log\mathrm{\sigma}(x) \right)=\frac{1}{2} \log{n(n-1)}.
	\end{aligned}
\end{equation}
This asymptotic property 
can be extended to general Riemannian manifolds, as pointed out by McOwen 
in \cite[Page 230]{McOwen1993}.  

  In this section, let $(\bar{M},g)$
  be a compact Riemannian manifold with smooth boundary $\partial M$, and we extend \eqref{asymptotic-rate1-002} to fully nonlinear case. 
\begin{theorem}
	\label{thm1-unique}
	In addition to  \eqref{concave}, \eqref{elliptic-weak-3} and \eqref{sumfi>0},
	we assume $(0,\cdots,0,1)\in\mathring{\Gamma}_{\mathcal{G}}^{f}$ and $(1,\cdots,1,-1)\in\bar{\Gamma}_{\mathcal{G}}^{f}$. Suppose that
	$\psi$ satisfies \eqref{non-degenerate2} and
	$\psi\big|_{\partial M}\equiv   f(D_o\vec{\bf1})$. Then  the interior $M$ admits a unique smooth admissible complete metric $g_u=e^{2u}g$ satisfying \eqref{main-equ0-Schouten}. Moreover, the solution obeys the following asymptotic behavior
	\begin{equation}
		\begin{aligned}
			\lim_{x\rightarrow\partial M} ({u}(x)+\log \mathrm{\sigma}(x)) =  \frac{1}{2}\log\frac{1}{2D_o}. 
		\end{aligned}
	\end{equation}
	
\end{theorem}

\begin{proof}
	The existence is a consequence of Theorem \ref{thm2-completemetric}. 
	The   asymptotic property and uniqueness 
	follow  from  Proposition \ref{unique-prop}, which  is a combination of Propositions \ref{super123} and \ref{proposition-asymptotic}. 
\end{proof}

When replacing \eqref{non-degenerate2} by \eqref{non-degenerate1}, using  Proposition \ref{super123}  and Lemma \ref{lemma2-asymptotic} we have the following existence and asymptotic behavior result. We shall stress  that the uniqueness  possibly fails. 
This is in contrast with Theorem \ref{thm1-unique}.
\begin{theorem}
	\label{thm1-unique-2}
	In addition to  \eqref{concave}, \eqref{elliptic-weak-3} and \eqref{sumfi>0},
	we assume $(0,\cdots,0,1)\in\mathring{\Gamma}_{\mathcal{G}}^{f}$ and $(1,\cdots,1,-1)\in\bar{\Gamma}_{\mathcal{G}}^{f}$. 	  Suppose 
	$\psi$ satisfies \eqref{non-degenerate1} and
	$\psi\big|_{\partial M}\equiv   f(D_o\vec{\bf1})$. Assume in addition that \eqref{condition-key100-1} holds for some $\sup_{\partial\Gamma}f<\sigma\leq \inf_M\psi.$  
	Then  the interior $M$ admits a smooth admissible complete metric $g_{u_\infty}=e^{2u_\infty}g$ satisfying \eqref{main-equ0-Schouten} and
	\begin{equation}	
		\label{inequa-asymptotic4}
		\begin{aligned}
			\lim_{x\rightarrow\partial M} ({u}_\infty(x)+\log \mathrm{\sigma}(x)) =  \frac{1}{2}\log\frac{1}{2D_o}. 
		\end{aligned}
	\end{equation}
	
\end{theorem}

 \begin{proof}
	Fix a   constant $0<\epsilon\ll1.$	Let $u_{k,\epsilon}$ be an admissible solution to
	\eqref{001} below. 
	Let $u_\infty=\underset{k\to+\infty}\lim u_{k,\epsilon}$, then $g_{u_\infty}$ is complete and smooth (by Theorem \ref{thm2-completemetric}). 
	Then the asymptotic ratio \eqref{inequa-asymptotic4} follows as a combination of \eqref{inequa-asymptotic} and \eqref{asymptotic-rate1}.
\end{proof}

  \begin{remark}
	Theorem \ref{thm1-unique-0} 
	is a combination of Theorems  \ref{thm1-unique} and \ref{thm1-unique-2}
\end{remark}

The rest of this section is to complete the proof of Lemma \ref{lemma2-asymptotic}, Propositions \ref{super123}, \ref{proposition-asymptotic} and  \ref{unique-prop}. The proofs follow previous draft \cite{yuan-PUE-conformal}. 
 Write
 \begin{equation}	V[u]= \nabla^2 u	+\frac{1}{2}|\nabla u|^2 g - \dd{u}\otimes \dd{u}-A_g.\end{equation}
From \eqref{conformal-formula2} we know 
$V[u]=-A_{ {g}_u}$. 
Then 
\eqref{main-equ0-Schouten}
reads as follows
\begin{equation}
	\label{mainequ-02-2-1}
	\begin{aligned}
		f(\lambda(g_u^{-1}V[u]))	=\psi. 
	\end{aligned}
\end{equation}

For $\sup_{\partial\Gamma}f<\sigma<\sup_{\Gamma}f$, denote by $\tilde{C}_\sigma$ the positive constant 
with
\begin{equation}
	\label{def1-c-sigma}
	\begin{aligned}
		f(\tilde{C}_\sigma \vec{\bf 1}) =\sigma.
	\end{aligned}
\end{equation}

\begin{proposition}
	\label{super123} 
		Let $\tilde{C}_\sigma $ be as in \eqref{def1-c-sigma}.
	Let $\widetilde{g}_\infty=e^{2\widetilde{u}_\infty}g$
	be a complete  metric  
	obeying the equation  \eqref{main-equ0-Schouten}
	with \eqref{concave}, \eqref{elliptic-weak-3}, \eqref{sumfi>0} and \eqref{non-degenerate1},   then
	\begin{equation} 
		\label{asymptotic-rate1}
		\begin{aligned}
			\lim_{x\rightarrow\partial M} (\widetilde{u}_\infty(x)+\log \mathrm{\sigma}(x))\leq \frac{1}{2}\log\frac{1}{2  \tilde{C}_{\inf_{\partial M}\psi}}. 
		\end{aligned}
	\end{equation}
\end{proposition}

\begin{proof}  
	Let $\Omega_\delta =\{x\in M: \mathrm{\sigma}(x)<\delta\} $ be as in \eqref{omega-delta}, and let $$\Omega_{\delta,\delta'}=\{x\in M: \delta'<\mathrm{\sigma}(x)<\delta-\delta'\}.$$
	We choose  $0<\delta'<\frac{\delta}{2}\ll1$ so that $\Omega_\delta$ and $\Omega_{\delta,\delta'}$ are both smooth.

	By some result of \cite{Aviles1988McOwen}, 
	 $\Omega_{\delta}$ admits a smooth complete conformal metric $
	 g_{w_\infty^{\delta}}=e^{2w_\infty^{\delta}}g$ with constant scalar curvature $-1$.
Similarly,  
	there is a smooth complete metric $
	g_{u_\infty^{\delta,\delta'}}=e^{2u_\infty^{\delta,\delta'}}g$  on $\Omega_{\delta,\delta'}$ with constant
	scalar curvature $-1$, that is, 
	\begin{equation}
		\begin{aligned}	
			\mathrm{tr} (g^{-1} V[{u_\infty^{\delta,\delta'}}])=\frac{ 1}{2(n-1) }e^{2u_\infty^{\delta,\delta'}}
			\mbox{ in } \Omega_{\delta,\delta'}, 
			\,\, 
			\lim_{x\to \partial \Omega_{\delta,\delta'}}{u_\infty^{\delta,\delta'}(x)} =+\infty.
		\end{aligned}
	\end{equation} 
	By the maximum principle 
	\begin{equation} \label{buchong78} \begin{aligned}
			u_\infty^{\delta,\delta'}\geq w_\infty^{\delta}
			\mbox{ in } \Omega_{\delta,\delta'}.
	\end{aligned}\end{equation}
	Furthermore, for any $0<\delta'<\delta_0'<\frac{\delta}{2}$, we have $\Omega_{\delta,\delta_0'}\subset \Omega_{\delta,\delta'}$ and
	\begin{equation} \label{buchong79} \begin{aligned}
			u_\infty^{\delta,\delta'}\leq u_\infty^{\delta,\delta_0'} \mbox{ in } \Omega_{\delta,\delta_0'}.\end{aligned}\end{equation}
		From \eqref{concave-1} we have
		\begin{equation}
			\label{key1-main2}
			\begin{aligned}
				\sum_{i=1}^n \lambda_i \geq n \tilde{C}_\sigma, \mbox{ }\forall \lambda\in\partial\Gamma^\sigma.
			\end{aligned}
		\end{equation}
	Therefore, the solution  $\widetilde{u}_\infty$ 
	satisfies that 
	$$\mathrm{tr} (g^{-1} V[\widetilde{u}_\infty]) \geq n \tilde{C}_\psi 
	e^{2 \widetilde{u}_\infty}.$$
	The maximum principle implies 
	\begin{equation}  \begin{aligned}
			u_\infty^{\delta,\delta'} \geq \widetilde{u}_\infty+\frac{1}{2}\log \left( 2n(n-1) \tilde{C}_{\inf_{\Omega_{\delta,\delta'}} \psi} \right) 
			\mbox{ in } \Omega_{\delta,\delta'}. \nonumber
	\end{aligned}\end{equation}
	
	Take $u_\infty^{\delta}(x)= \underset{\delta'\rightarrow0}{\lim}\, u_\infty^{\delta,\delta'}(x)$ (such a limit exists by \eqref{buchong78}-\eqref{buchong79}). Then 	\begin{equation} 
		\label{f-equ1}
		\begin{aligned}
			u_\infty^{\delta} \geq w_\infty^{\delta}, 
			\,\, 
			u_\infty^{\delta} \geq  \widetilde{u}_\infty+  
			\frac{1}{2}\log \left( 2n(n-1) \tilde{C}_{\inf_{\Omega_{\delta}} {\psi}}\right)  \mbox{ in } \Omega_{\delta}, 
		\end{aligned}
	\end{equation}
	and $g_{u_\infty^\delta}=e^{2u_\infty^\delta}g$ is a smooth complete 
	metric on $\Omega_{\delta}$ of scalar curvature $-1$.  
	(Note that $u_\infty^{\delta}$ is smooth, since one has the local estimates).
	Furthermore, \eqref{asymptotic-rate1-002} gives
	$$\lim_{x\rightarrow \partial M} (u_\infty^{\delta}(x)+\log\mathrm{\sigma}(x))= \frac{1}{2}\log [n(n-1)].$$
	Putting them together,
	we get \eqref{asymptotic-rate1}. 
\end{proof}

\begin{lemma}
	\label{lemma2-asymptotic}
	
	Assume \eqref{concave}, \eqref{elliptic-weak-3},
	\eqref{sumfi>0} and  \eqref{non-degenerate1} hold. 
		Suppose  in addition that $(0,\cdots,0,1)\in\mathring{\Gamma}_{\mathcal{G}}^{f}$ and $(1,\cdots,1,-1)\in\bar{\Gamma}_{\mathcal{G}}^{f}$. 	Fix a constant $0<\epsilon<1$. 
	Let $u_{k,\epsilon}$ be an admissible solution to
	\begin{equation}
		\label{001}
		\left\{
		\begin{aligned}
			\,& f(\lambda(g_{u_{k,\epsilon}}^{-1} V[u_{k,\epsilon}]))=\psi  \,&\mbox{ in } \bar{M}, \\
			\,& u_{k,\epsilon}=\log {k} +\frac{1}{2}\log\frac{(1-\epsilon)^2 }{2 \tilde{C}_{\sup_{\partial M}\psi+\epsilon}} \,& \mbox{ on } \partial M.
		\end{aligned}
		\right.
	\end{equation}
	Then 
	for  any sufficiently small $\delta$ and $k\geq\frac{1}{\delta}$, there holds that
		\begin{equation}
		\label{inequa-asymptotic}
		\begin{aligned}
			u_{k,\epsilon}\geq  \log\frac{k}{k\mathrm{\sigma}+\delta}+\frac{1}{2}\log\frac{(1-\epsilon)^2 }{2\tilde{C}_{\sup_{\partial M} {\psi}+\epsilon}}+\frac{1}{\mathrm{\sigma}+\delta}-\frac{1}{\delta}  \mbox{ in } \Omega_{\delta}. 
		\end{aligned}
	\end{equation}
 
\end{lemma}

\begin{proof}
	The key ingredients 
	are  our Lemma \ref{lemma2-key}    and 
	$|\nabla \mathrm{\sigma}|=1$ on $\partial M$.    
 Without loss of generality, by applying Lemma 
 \ref{lemma1-construct0} to the cone $\mathring{\Gamma}_{\mathcal{G}}^f$  we can construct $\mathring{\Gamma}_{\mathcal{G}}^f$-admissible metrics. Hence we may assume  
	\begin{align*}
		\lambda(-g^{-1}A_g)\in \mathring{\Gamma}_{\mathcal{G}}^f \mbox{ in } \bar{M}.
	\end{align*}  
	According to Lemma \ref{lemma3.4-2} there exists some uniform constant $B_3$ such that
	\begin{equation}
		\begin{aligned}
			f(e^{-2B_3}\lambda [-g^{-1}A_{g}]) \geq \psi.  \nonumber
		\end{aligned}
	\end{equation}
Combining the maximum principle and Lemma \ref{lemma3.4-3}, we can derive that
	\begin{equation}
		\label{1114-3}
		\begin{aligned} 
			\inf_M u_{k,\epsilon} \geq \min\left\{ B_3,  \frac{1}{2}\log\frac{k^2(1-\epsilon)^2 }{2\tilde{C}_{\sup_{\partial M} {\psi}+\epsilon}} \right\}. 
		\end{aligned}
	\end{equation}

	The local barrier is given by 
	\begin{equation}
		\label{1115-0}
		\begin{aligned}
			h_{k,\epsilon,\delta}(\mathrm{\sigma})=\log\frac{k}{k\mathrm{\sigma}+1}
			+\frac{1}{2}\log\frac{(1-\epsilon)^2  }{2 \tilde{C}_{\sup_{\partial M}  {\psi}+\epsilon}}+
			\frac{1}{\mathrm{\sigma}+\delta}- \frac{1}{\delta}. 
		\end{aligned}
	\end{equation} 
By some straightforward computation, we can infer that
	\begin{equation}
		\label{1114-1}
		\begin{aligned}
			h_{k,\epsilon,\delta}'=-\frac{k}{k\mathrm{\sigma}+1}-\frac{1}{(\mathrm{\sigma}+\delta)^2},
			\,\,
			h_{k,\epsilon,\delta}''=\frac{k^2}{(k\mathrm{\sigma}+1)^2}+\frac{2}{(\mathrm{\sigma}+\delta)^3}, \nonumber
		\end{aligned}
	\end{equation}
	\begin{equation}
		\begin{aligned}
			h_{k,\epsilon,\delta}'^2
			=\frac{k^2}{(k\mathrm{\sigma}+1)^2}+\frac{1}{(\mathrm{\sigma}+\delta)^4}+\frac{2k}{(k\mathrm{\sigma}+1)(\mathrm{\sigma}+\delta)^2}, \nonumber
		\end{aligned}
	\end{equation} 
	
	\begin{equation}
		\label{yuan-39}
		\begin{aligned}
			V[h_{k,\epsilon,\delta}] = \,&  \frac{1}{2} h_{k,\epsilon,\delta}'^2 |\nabla \mathrm{\sigma}|^2 g 
			+(h_{k,\epsilon,\delta}''-h_{k,\epsilon,\delta}'^2) \dd{\mathrm{\sigma}}\otimes \dd{\mathrm{\sigma}}
			+h_{k,\epsilon,\delta}'   \nabla^2 \mathrm{\sigma}-A_g 
			\\ =\,&  
			\left( \frac{1}{2(\mathrm{\sigma}+\delta)^4}+ \frac{k}{(k\mathrm{\sigma}+1)(\mathrm{\sigma}+\delta)^2} 
			\right) (|\nabla \mathrm{\sigma}|^2 g-2 \dd{\mathrm{\sigma}}\otimes \dd{\mathrm{\sigma}})	 
			\\ 	\,&  
			+\frac{1}{(\mathrm{\sigma}+\delta)^2}
			\left(- \nabla^2\mathrm{\sigma} +\frac{2}{\mathrm{\sigma}+\delta} \dd{\mathrm{\sigma}}\otimes \dd{\mathrm{\sigma}} \right)
		 \\ \,&
			+
			\frac{k^2}{2(k\mathrm{\sigma}+1)^2}   |\nabla \mathrm{\sigma}|^2 g  
			- \frac{k}{k\mathrm{\sigma}+1}  \nabla^2 \sigma
			-A_g.
		\end{aligned}
	\end{equation}
	
	Note that $|\nabla \mathrm{\sigma}|^2=1$ on ${\partial M}$, 
	$\Delta\mathrm{\sigma} \cdot g-\varrho\nabla^2\mathrm{\sigma}$ is bounded in $\bar{M}$, and
	$\frac{k}{k\mathrm{\sigma}+1} \geq \frac{1}{\mathrm{\sigma}+\delta} \mbox{ for } k\geq\frac{1}{\delta}$ as well as
  $\frac{1}{\mathrm{\sigma}+\delta}\gg1$ in $\Omega_\delta$ 
	if $0<\delta\ll1$.
	For the  constant $\epsilon>0$ fixed, there is a $\delta_\epsilon$ depending on $\epsilon$ and other known data  
	such that 
	\begin{enumerate}
		\item $\mathrm{\sigma}(x)$ is smooth, and $|\nabla \mathrm{\sigma}|^2\geq 1-\epsilon$;
		\item $\sup_{\partial M} {\psi}+\epsilon\geq \sup_{\Omega_{\delta}} {\psi}$;

		\item $  \frac{\epsilon k^2}{2(k\mathrm{\sigma}+1)^2}   |\nabla \mathrm{\sigma}|^2 g 
		- \frac{k}{k\mathrm{\sigma}+1}  \nabla^2 \sigma -A_g\geq0$;
		
		\item $\lambda \left[  g^{-1}\left(- \nabla^2\mathrm{\sigma} +\frac{2}{\mathrm{\sigma}+\delta} \dd{\mathrm{\sigma}}\otimes \dd{\mathrm{\sigma}} \right)\right] \in \mathring{\Gamma}_{\mathcal{G}}^{f}$; 
		\item $\lambda \left[g^{-1} \left(|\nabla \mathrm{\sigma}|^2 g-2 \dd{\mathrm{\sigma}}\otimes \dd{\mathrm{\sigma}}\right) \right]\in\bar {\Gamma}_{\mathcal{G}}^{f}$,
	\end{enumerate}
in $\Omega_\delta$ for 	$0<\delta<\delta_{\epsilon}$.
	To derive the last two claims, we  need to use 		$(0,\cdots,0,1)\in \mathring{\Gamma}_{\mathcal{G}}^{f}$ and $(1,\cdots,1,-1)\in\bar {\Gamma}_{\mathcal{G}}^{f}$, respectively.
	 In particular, we have
	 \begin{equation}
	 	\label{key5-comparision}
	 	\begin{aligned}
	 		\lambda(g^{-1}	V[h_{k,\epsilon,\delta}])
	 		\in  \mathring{\Gamma}_{\mathcal{G}}^{f} \mbox{ in } \Omega_{\delta} \mbox{ for }
	 		0<\delta\ll1.
	 	\end{aligned}
	 \end{equation}

	From now on we fix such $0<\delta<\delta_\epsilon$ 
  and $k\geq\frac{1}{\delta}$.
	By   \eqref{1115-0}, $e^{-2h_{k,\epsilon,\delta}}\geq  \frac{2\tilde{C}_{\sup_{\partial M}\psi+\epsilon} \cdot (k\sigma+1)^2}{(1-\epsilon)^2k^2}.$
	Using  Lemma \ref{lemma315},  
	we can conclude that 	in $\Omega_\delta$ 	 
	\begin{equation}
		\label{buchong1}
		\begin{aligned}
			f(\lambda(g_{h_{k,\epsilon,\delta}}^{-1}
			V[h_{k,\epsilon,\delta}]))
			\geq \,&
			f\left(\frac{k^2{ (1-\epsilon)} }{2(k\mathrm{\sigma}+1)^2}\cdot|\nabla \mathrm{\sigma}|^2 e^{-2h_{k,\epsilon,\delta}}\vec{\bf1}\right)  \\
			\geq \,& 
			f(\tilde{C}_{\sup_{\partial M}\psi+\epsilon} \vec{\bf1})
			=\sup_{\partial M}\psi+\epsilon	 
			\geq 
			\psi. 
		\end{aligned}
	\end{equation} 
	From  \eqref{1114-3}, $u_{k,\epsilon}$ has a uniform lower bound. 
	Notice that 
	\begin{equation}
		\begin{aligned}
			\frac{\log\frac{k}{k\delta+1}}{\frac{1}{2\delta}}\leq\frac{\log\frac{1}{\delta}}{\frac{1}{2\delta}}\rightarrow 0  \mbox{ and   } 
			\log \frac{k}{k\delta+1}-\frac{1}{2\delta} \to -\infty
		\end{aligned}
	\end{equation}  as $\delta\rightarrow0^+$.
	Thus there is $\delta_0$ such that for any $0<\delta<\delta_0$, we get
	$u_{k,\epsilon}\big|_{\mathrm{\sigma}=\delta} \geq h_{k,\epsilon,\delta}(\delta)$. Note also that $u_{k,\epsilon}=h_{k,\epsilon,\delta}$ on $\partial M$.
	Therefore, together with \eqref{key5-comparision},  Lemma \ref{lemma1-comparision}
	yields 	in $\Omega_{\delta}$ that 
	\begin{equation}
		\begin{aligned}
			u_{k,\epsilon}\geq h_{k,\epsilon,\delta}=\log\frac{k}{k\mathrm{\sigma}+\delta}+\frac{1}{2}\log\frac{(1-\epsilon)^2 }{2\tilde{C}_{\sup_{\partial M} {\psi}+\epsilon}}+\frac{1}{\mathrm{\sigma}+\delta}-\frac{1}{\delta}.  \nonumber
		\end{aligned}
	\end{equation}

\end{proof}


\begin{proposition}
	\label{proposition-asymptotic}

	Assume \eqref{concave}, \eqref{elliptic-weak-3},
	 \eqref{sumfi>0} and  \eqref{non-degenerate2} hold. 	Suppose  in addition that $(0,\cdots,0,1)\in\mathring{\Gamma}_{\mathcal{G}}^{f}$ and $(1,\cdots,1,-1)\in\bar{\Gamma}_{\mathcal{G}}^{f}$. 
	Then for any admissible complete metric ${g}_{\widetilde{u}_{\infty}}=e^{2\widetilde{u}_\infty}g$  satisfying \eqref{main-equ0-Schouten}, we have 
	\begin{equation}
		\label{asymptotic-rate2}
		\begin{aligned}
			\lim_{x\rightarrow\partial M} (\widetilde{u}_\infty(x)+\log \mathrm{\sigma}(x))\geq \frac{1}{2}\log\frac{1}{2 \tilde{C}_{\sup_{\partial M}\psi}}.  
		\end{aligned}
	\end{equation}
	
\end{proposition}

\begin{proof}
	 Let $u_{k,\epsilon}$ be an admissible solution to \eqref{001}.
	 Since $\sup_{\partial \Gamma_{\mathcal{G}}^f}f <\psi<\sup_\Gamma f$, by Theorem  
	 \ref{thm2-finiteBVC} we know
	 \begin{equation}
	 	\label{equ-belong-to-920}
	 	\begin{aligned} 
	 		 \lambda(-g^{-1}A_{g_{u_{k,\epsilon}}})\in \mathring{\Gamma}_{\mathcal{G}}^f \mbox{ in } \bar{M}.
	 	\end{aligned}
	 \end{equation} 
This is crucial for adapting the comparision principle (Lemma \ref{lemma1-comparision}),
	 which implies  
	 \begin{equation}
	 	\label{b4b}
	 	\begin{aligned}
	 		\widetilde{u}_\infty\geq u_{k,\epsilon} \mbox{ in } M.
	 	\end{aligned}
	 \end{equation}
 Combining with \eqref{inequa-asymptotic} from Lemma \ref{lemma2-asymptotic}, 
 	we know that in $\Omega_\delta$,
 \begin{equation}
 	\begin{aligned}
 		\widetilde{u}_\infty+\log \mathrm{\sigma}\geq \frac{1}{2}\log\frac{(1-\epsilon)^2
 		}{2\tilde{C}_{\sup_{\partial M} {\psi}+\epsilon}}
 		+\frac{1}{\mathrm{\sigma}+\delta}-\frac{1}{\delta}.  \nonumber
 	\end{aligned}
 \end{equation}
 Thus we get \eqref{asymptotic-rate2}. 
\end{proof}

\begin{proposition}
	\label{unique-prop}
	Suppose	\eqref{concave}, \eqref{elliptic-weak-3}, \eqref{sumfi>0} and \eqref{non-degenerate2} hold.
	Assume in addition that $(0,\cdots,0,1)\in\mathring{\Gamma}_{\mathcal{G}}^{f}$ and $(1,\cdots,1,-1)\in\bar{\Gamma}_{\mathcal{G}}^{f}$. 
	Let  $u^{(k)}$ be admissible solution to
	\begin{equation}	\begin{aligned}		 
			f(\lambda(g_{u^{(k)}}^{-1}V[u^{(k)}]))=\psi    
			\mbox{ in }\bar{M},
			\,\,
			u^{(k)}=\log k \mbox{ on } \partial M. \nonumber
	\end{aligned}	\end{equation}
Then $u_\infty(x)=\underset{k\to+\infty}{\lim}\, u^{(k)}(x)$ exists and is smooth in  the interior $M$.
 
Moreover, for  any admissible solution $\widetilde{u}_\infty$  to 
	the Dirichlet problem for equation \eqref{mainequ-02-2-1} with infinite boundary value condition, we have
 \begin{equation}
 	\label{inequality1-intermed}
 	\begin{aligned}
 		u_\infty\leq \widetilde{u}_\infty \leq u_\infty +\frac{1}{2}\left[\log\tilde{C}_{\sup_{\partial M}\psi}-\log\tilde{C}_{\inf_{\partial M}\psi}\right] \mbox{ in } M.
 	\end{aligned}
 \end{equation}
	
\end{proposition}

\begin{proof}
 Theorem \ref{thm2-finiteBVC} 	yields $\lambda(g^{-1}V[u^{(k)}])\in\mathring{\Gamma}_{\mathcal{G}}^{f}$, which is the crucial ingredient for applying Lemma \ref{lemma1-comparision}, thereby implying
	 $\widetilde{u}_\infty\geq u^{(k)}$ and  $u^{(k+1)}\geq u^{(k)}$ $(\forall k\geq1)$. Then  $u_\infty(x)=\underset{k\to+\infty}{\lim}\, u^{(k)}(x)$ exists and  $\widetilde{u}_\infty\geq u_\infty$. By local estimates, Evans-Krylov theorem and Schauder theory, we may conclude that $u_\infty$ is smooth in the interior.
From  Theorem \ref{thm2-finiteBVC} or Proposition \ref{pro1-unique}, we also see  
$$ \lambda(g^{-1}V[u_\infty])\in\mathring{\Gamma}_{\mathcal{G}}^{f} \textrm{ and }  \lambda(g^{-1}V[\widetilde{u}_\infty])\in\mathring{\Gamma}_{\mathcal{G}}^{f}
 \textrm{  in } M.$$
	
	Next, we prove the  inequality \eqref{inequality1-intermed}. 
Firstly, by Proposition \ref{proposition-asymptotic},
$$
			\lim_{x\rightarrow\partial M} ({u}_\infty(x)+\log \mathrm{\sigma}(x))\geq \frac{1}{2}\log\frac{1}{2 \tilde{C}_{\sup_{\partial M}\psi}}.  
$$
Combining with
	\eqref{asymptotic-rate1} from Proposition \ref{super123}, we derive $$\lim_{x\rightarrow\partial M}(\widetilde{u}_\infty(x)-u_\infty(x))\leq \frac{1}{2}\left[\log\tilde{C}_{\sup_{\partial M}\psi}-\log\tilde{C}_{\inf_{\partial M}\psi}\right].$$
	Assume 
	there is an  interior point $x_0\in M$ such that
	$$(\widetilde{u}_\infty-u_\infty)(x_0)=\sup_M(\widetilde{u}_\infty-u_\infty)>\frac{1}{2}\left[\log\tilde{C}_{\sup_{\partial M}\psi}-\log\tilde{C}_{\inf_{\partial M}\psi}\right].$$ 
	This is a contradiction, since the maximum principle and Lemma \ref{lemma3.4} together yield that  $\widetilde{u}_\infty(x_0)\leq u_\infty(x_0)$.  
\end{proof}

  \section{Equations on complete noncompact manifolds}
  \label{sec-complete-noncompact-existence}
We employ an approximate method, building on Theorem 
\ref{existence1-compact-2}.  
Let $\{M_k\}_{k=1}^{+\infty}$ be an exhaustion series of domains  
with
\begin{itemize}
	\item $M=\cup_{k=1}^{\infty} M_k, \mbox{  }  \bar{M}_k =M_k\cup\partial M_k,\mbox{  }  \bar{M}_k\subset\subset M_{k+1}$.
	
	\item  $\bar{M}_k$ is a compact $n$-manifold with smooth boundary $\partial M_k$.
\end{itemize}

According to 
Theorem 
\ref{existence1-compact-2}, for each $k\geq 1$,  there is an admissible function $u_k\in C^\infty(M_k)$  satisfying
\begin{equation}
	\label{approximate-DP2-2}
	\begin{aligned}
		f (\lambda[-g_{u_k}^{-1} A_{g_{u_k}}] ) =\psi    \mbox{ in } M_k,
	\end{aligned}
\end{equation}
\begin{equation}
	\label{approximate-DP3-2}
	\begin{aligned}
		\lim_{x\rightarrow\partial M_k} u_k(x)=+\infty.  
	\end{aligned}
\end{equation}

Let $\hat{u}$ be as in Proposition \ref{thm4-construction}.
Since  $\lambda(-{g}^{-1}A_{{g}_{\hat{u}}})\in
\mathring{\Gamma}_{\mathcal{G}}^f$, by Lemma 	 \ref{lemma1-comparision} we have
\begin{proposition}
	\label{thm-c0-lower-2}
	For any admissible solution $u_k$ to \eqref{approximate-DP2-2}-\eqref{approximate-DP3-2}, we have 
	\begin{equation}
		u_k\geq  \hat{u} \mbox{ in } M_k, \,\,\forall k\geq 1. \nonumber
	\end{equation} 
\end{proposition}

Since $\bar{\Gamma}^{\psi}\subseteq\mathring{\Gamma}_{\mathcal{G}}^f$  
we know 
$\lambda[-g_{u_k}^{-1} A_{g_{u_k}}]\in \mathring{\Gamma}_{\mathcal{G}}^f.$
By the comparison principle (Lemma 	 \ref{lemma1-comparision}), we immediately derive the following result.
\begin{proposition}
	\label{thm-c0-upper-2}
	Let $u_k$ be the admissible solution to problem \eqref{approximate-DP2-2}-\eqref{approximate-DP3-2}, then
	\begin{equation}
		\begin{aligned}
			u_{k+1}\leq  u_k \mbox{ in } M_k, \,\,\forall k\geq 1. \nonumber
		\end{aligned}
	\end{equation} 
\end{proposition}

	\noindent
	\textit{Proof of Theorem \ref{thm1-noncompact}}.
Proposition \ref{thm-c0-upper-2} states that $\{u_k\}_{k=m}^\infty$ is a decreasing sequence  of functions on $M_m$; while such a decreasing  sequence 
is uniformly bounded from below, according to Proposition \ref{thm-c0-lower-2}.
Thus we can take the limit 
\begin{equation}
	\label{uinfty-lim}
	u_\infty=\lim_{k\to+\infty} u_k.
\end{equation} 
Moreover,
$u_\infty\geq  \hat{u}$ in $M$. 
In Propositions  \ref{thm-c0-lower-2}-\ref{thm-c0-upper-2} and Subsections \ref{subsection-local-C1}-\ref{subsection-local-C2}, we establish  the local estimates up to second order derivatives  for each $u_k$.
Combining with
Evans-Krylov theorem and classical Schauder theory,  
we conclude that $u_\infty$ is 
a smooth admissible solution to the equation 	\eqref{main-equ0-Schouten}.



	Let $w_\infty $ be another admissible solution to the equation \eqref{main-equ0-Schouten}, then by the comparison principle we have 
	$u_k\geq w_\infty$ in $M_k$ for all $k\geq 1$. Thus
	$u_\infty(x)\geq  w_\infty(x),$ $\forall x\in M.$
	This means that the solution given by \eqref{uinfty-lim}  is the maximal solution.  The uniqueness of maximal solution is obvious.


	\begin{appendix}

	\section{Further discussions on open symmetric convex cones}
	\label{sec1-cones}
	
		Let $\Gamma$ be an open symmetric convex cone  $\Gamma\subsetneq\mathbb{R}^n$,  with vertex at origin and  nonempty 
	boundary 
	$\partial \Gamma\neq\emptyset$,  containing  the positive cone. Given such a cone $\Gamma$ and a non-zero constant $\varrho<n$,
	as in Corollary \ref{coro1-laplace1}
	 we may take throughoutly 
	\begin{equation}
		\label{map1}
		\begin{aligned}
			\tilde{\Gamma}  =
			\left\{(\lambda_1,\cdots,\lambda_n): 
			\lambda_i=\frac{1}{\varrho}\left(\sum_{j=1}^n \mu_j-(n-\varrho)\mu_i\right),
			\mbox{ } (\mu_1,\cdots,\mu_n)\in \Gamma 
			\right\}. \nonumber
		\end{aligned}
	\end{equation} 
	
From  Theorem \ref{thm3-completemetric}, 
it is of special interest to consider the condition
\begin{equation}
	\label{key-condition1-construct}
	(\frac{\tau-2}{\tau-1}, \cdots, \frac{\tau-2}{\tau-1}, \frac{\tau}{\tau-1}) \in\bar{\Gamma}.
\end{equation}
For simplicity,  $ (\gamma, \cdots,\gamma,\gamma+\varrho)=	(\frac{\tau-2}{\tau-1}, \cdots, \frac{\tau-2}{\tau-1}, \frac{\tau}{\tau-1})$ if  we take
\begin{equation}
	\label{beta-gamma-A-3}
	\begin{aligned}
		\varrho=\frac{n-2}{\tau-1}, \mbox{ }
		\gamma=\frac{(\tau-2)(n-2)}{2(\tau-1)}. 
	\end{aligned}
\end{equation} 
Obviously, under the assumption \eqref{tau-alpha-sharp}, the condition \eqref{key-condition1-construct} holds if	
\begin{equation}
	\label{tau-alpha-4}
	\begin{cases}
		\tau\leq  2-\frac{2}{\varrho_\Gamma} \,& \mbox{ if } \alpha=-1,\\
		\tau\geq  2  \,& \mbox{ if } \alpha=1. 
	\end{cases}
\end{equation}
\begin{remark}
When taking the cone   $\mathring{\Gamma}_{\mathcal{G}}^f$ in \eqref{key-condition1-construct},
we may apply  Theorem \ref{thm3-completemetric} 
to consult the existence.
\end{remark}
	
	
	For purpose of verifying   \eqref{key-condition1-construct} or  \eqref{tau-alpha-4}, which is imposed as a proper condition 
	in Lemma \ref{lemma1-construct1-modifiedSchouten}
	to construct admissible metrics without pseudo-admissible metric assumption, it seems necessary to estimate $\varrho_\Gamma$ and $\varrho_{\tilde{\Gamma}}$. 
	For this reason,
	together with \eqref{check-2} 
  and \eqref{key-condition1-construct},
	 the  cones with $\varrho_{\tilde{\Gamma}}\geq 2$ 
	 or with $\varrho_{{\Gamma}}\geq 2$
	 are of special interests. 
	Below we prove some related results.
	First we prove a key ingredient by projection. 
	Let $\Gamma_\infty$ 
	be 
	the projection of $\Gamma$ to the subspace of former $n-1$ subscripts,
	 that is   \begin{equation}	\label{construct1-Gamma-infty}	\begin{aligned}	\Gamma_\infty:=\{(\lambda_{1}, \cdots, \lambda_{n-1}): (\lambda_{1}, \cdots, \lambda_{n-1},\lambda_n) \in \Gamma\}.	\end{aligned}\end{equation}
	As noted by 
	\cite{CNS3}, when $\Gamma$ is of type 1,  $\Gamma_\infty\neq\mathbb{R}^{n-1}$, and $\Gamma_\infty$ is an open symmetric convex cone in $\mathbb{R}^{n-1}$.
	From the construction of $\Gamma_\infty$, we can verify that 
	\begin{lemma}
		\label{lemma1-preserving-kappa}
		If $\Gamma$ is of type 1 
		then 
  $\kappa_{\Gamma_\infty}=\kappa_\Gamma$ 
  and $ \varrho_\Gamma\leq \varrho_{\Gamma_\infty}.$  
	\end{lemma}

	
	This is a key ingredient for estimating $\varrho_\Gamma$ from above.
	Inspired by this lemma, we will construct certain cones iteratively via projection. 
	To do this, if $\kappa_\Gamma=n-2$ 
  then we obtain $\Gamma_\infty$, which is of type 2. 
	When $\kappa_\Gamma =\kappa_{\Gamma_\infty} \leq n-3$,  
	i.e., $\Gamma_\infty$ is also of type 1, 
	similarly we further construct a cone, denoted by $$\Gamma^\infty_{\mathbb{R}^{n-2}}\subset \mathbb{R}^{n-2},$$  which is the projection of $\Gamma_\infty$ to the subspace of former $n-2$ subscripts.
	Accordingly, we can construct the cones by projection as follows:
	\begin{equation}
		\label{construction-cones-infty}
		\begin{aligned}
			\Gamma^\infty_{\mathbb{R}^{n-3}}\subset \mathbb{R}^{n-3}, \cdots, \Gamma^\infty_{\mathbb{R}^{\kappa_\Gamma+1}} \subset \mathbb{R}^{\kappa_\Gamma+1}.
		\end{aligned}
	\end{equation}
	
	From this construction, we know $\Gamma^\infty_{\mathbb{R}^{n-1}} =\Gamma_\infty  
	$ and
	$\Gamma^\infty_{\mathbb{R}^{\kappa_\Gamma+1}}$
	is of type 2.
	In fact for general $\kappa_\Gamma+1\leq k\leq n-1$, one can check that 
	\begin{equation}
		\label{construction2-cones-infty}
		\begin{aligned}
			\Gamma^\infty_{\mathbb{R}^{k}}
			=\left\{(\lambda_1,\cdots,\lambda_{k})\in \mathbb{R}^{k}: (\lambda_1,\cdots,\lambda_{k},  {\overbrace{R,\cdots,R}^{(n-k)-\mathrm{entries}}}) \in\Gamma  \mbox{ for some } R>0 
			\right\}.  
		\end{aligned}       
	\end{equation}
	
	For $1\leq k\leq n$,  
 as in \cite{Harvey2012Lawson-Adv,Harvey2013Lawson-IUMJ} 
	we denote 
	\begin{equation}
		\label{def-P_k}
		\mathcal{P}_{k}	=
		\left\{(\lambda_1,\cdots,\lambda_{n}): \lambda_{i_1}+\cdots+\lambda_{i_{k}}>0, \,  \forall 1\leq i_1<\cdots<i_{k}\leq n\right\}. 
	\end{equation} 
	%
	We may rewrite Lemma  \ref{lemma1}  as follows:
	\begin{lemma}
		\label{lemma2}
		Suppose $\Gamma$ is a cone of $\kappa_{\Gamma}=k$ with $1\leq k\leq n-2$. Then 
		$\Gamma\subseteq\mathcal{P}_{k+1}.$
	\end{lemma}

As a complementary,
we pose a problem as in the following.

\begin{problem}
	\label{problem-1}
	Let $\Gamma$ be a cone of $\kappa_{\Gamma}=k$ with $1\leq k\leq n-2$. 
	For such $\Gamma$, is $\Gamma_{n-k}\subseteq \Gamma$ correct?
		
	
\end{problem}


Next  we can check that
	
	\begin{lemma}
		\label{lemma1-rigidity}
		If $\Gamma=\mathcal{P}_k$ for some $1\leq k\leq n$, then $\kappa_{\Gamma}=k-1$ and
		$\varrho_{\Gamma}=k. $
		In particular, $\varrho_\Gamma=1+\kappa_\Gamma.$
	\end{lemma}

According to Lemma \ref{lemma1-preserving-kappa} we may conclude the following proposition.
\begin{proposition}
\label{lemma1-upper-rho}
For any  $\Gamma$, we have 
 $$\varrho_\Gamma\leq  \kappa_\Gamma+1,$$  
with equality if and only if
$\Gamma 
=\mathcal{P}_{k} \mbox{ for some } 1\leq k\leq n.$

\end{proposition}


\begin{proof}
For  $\kappa_\Gamma=n-1$, the statement is obvious.
Next we assume   $\kappa_\Gamma\leq n-2$. 
 Using Lemma \ref{lemma1-preserving-kappa},
 $\varrho_\Gamma\leq \varrho_{\Gamma^\infty_{\mathbb{R}^{\kappa_\Gamma+1}}}\leq 1+\kappa_\Gamma.$
 (This can also be derived from Lemmas \ref{lemma2}, \ref{lemma1-rigidity}).
Finally we prove the rigidity. 
If $\varrho_\Gamma=\kappa_\Gamma+1$, then   $\varrho_{\Gamma^\infty_{\mathbb{R}^{\kappa_\Gamma+1}}} =\kappa_\Gamma+1$ by Lemma \ref{lemma1-preserving-kappa}. This implies  
$\Gamma^\infty_{\mathbb{R}^{\kappa_\Gamma+1}}=
 \{(\lambda_1,\cdots,\lambda_{1+\kappa_\Gamma}): \sum_{j=1}^{\kappa_\Gamma+1}\lambda_j>0  \}.$
Thus $\Gamma=\mathcal{P}_{\kappa_\Gamma+1}$.

\end{proof}

As a consequence,   
we may verify  condition \eqref{tau-alpha-4} in some case.

\begin{corollary}
\label{coro1-verify}
Given a cone $\Gamma$ with $\kappa_\Gamma\leq n-3$, we have $1+(n-2)\varrho_\Gamma^{-1}\geq 2$.
\end{corollary}

From the construction of conformal admissible metric in Subsection \ref{subsec1-construction-admissiblemetric},
 it is of  interest to compute $\varrho_{\tilde{\Gamma}}$  explicitly or give a lower bound of it. 

\begin{proposition}	\label{lemma1-lower-rho}
Given a cone $\Gamma$ and a constant $\varrho\leq \varrho_\Gamma$ with $\varrho\neq0$, we assume  $\tilde{\Gamma}$ is the corresponding cone as in \eqref{map1}. Then we have
\begin{enumerate}
\item If $\varrho<0$ then $\tilde{\Gamma}$ is of type 2 and $\varrho_{\tilde{\Gamma}}
=\varrho_\Gamma+\frac{\varrho_\Gamma(n-\varrho_\Gamma)}{\varrho_\Gamma-\varrho}$.

\item If $0<\varrho\leq\varrho_{\Gamma}$ then
\begin{itemize}
\item When $\Gamma$ is of type 1, $\varrho_{\tilde{\Gamma}}=n-\varrho.        $ In particular, $\varrho_{\tilde{\Gamma}}\geq n-\varrho_\Gamma\geq n-\kappa_\Gamma-1$. 

\item When $\Gamma$ is of type 2, $\varrho_{\tilde{\Gamma}}>n-\varrho.$ 
In particular, $\varrho_{\tilde{\Gamma}}+\varrho_\Gamma> n$.
\end{itemize} 
\end{enumerate}
\end{proposition}

\begin{proof}

Case 1: $\varrho<0$.
Let $\mu=(1,\cdots,1,1-\varrho_\Gamma)\in\partial \Gamma$. The corresponding vector
\begin{equation}
\begin{aligned}
\lambda=\frac{\varrho-\varrho_\Gamma}{\varrho}(1,\cdots,1,1-\frac{\varrho_\Gamma(n-\varrho)}{\varrho_\Gamma-\varrho})\in\partial\tilde{\Gamma}. \nonumber
\end{aligned}
\end{equation}

Case 2: $0<\varrho\leq\varrho_\Gamma$. Set $\mu=(0,\cdots,0,1)\in\bar \Gamma$.   Accordingly we have
\begin{equation}
\begin{aligned}
\lambda=\frac{1}{\varrho}(1,\cdots,1, 1-(n-\varrho))\in\overline{\tilde{\Gamma}}. \nonumber
\end{aligned}
\end{equation}

If $\Gamma$ is of type 1, then $\varrho_{\tilde{\Gamma}}=n-\varrho.$
So  $\varrho_{\tilde{\Gamma}}\geq n-\varrho_\Gamma\geq n-\kappa_\Gamma-1$ by Proposition \ref{lemma1-upper-rho}.  
In particular, when $\tilde{\Gamma}$ is of type 2 (if and only if $\varrho<\varrho_\Gamma$), we obtain  $	\varrho_{\tilde{\Gamma}}>n-\varrho. $ 

If $\Gamma$ is of type 2, then $	\varrho_{\tilde{\Gamma}}>n-\varrho. $  
 
\end{proof}

\begin{corollary}
\label{coro68}
Fix a cone $\Gamma$.
Let  $\varrho$ be as in  \eqref{beta-gamma-A-3}, and assume $\varrho\leq \varrho_\Gamma$.  
As in  \eqref{map1} we obtain $\tilde{\Gamma}$. We have $\varrho_{\tilde{\Gamma}}\geq 2$, provided either one of the following hods.
\begin{enumerate}
\item 	$(\alpha,\tau)$ obeys \eqref{tau-alpha-4}.

\item 
$\tau>1$ and 
$\kappa_\Gamma\leq n-3$.
\end{enumerate}

\end{corollary}

However, 
one  could not expect to obtain  an effective estimate for all type 2 cones as shown by the following:

\begin{corollary}
\label{corollary1-example-type2}
For  any   $1<t<n$, there is a type 2 cone $\tilde{\Gamma}$
with  $\varrho_{\tilde{\Gamma}}=t.$ 
\end{corollary}
\begin{proof}
Given $1<t<n$, there is a unique $\varrho<0$ such that $\frac{ n-\varrho}{1-\varrho}=t$. For such $\varrho$  we get a type 2 cone
$\tilde{\Gamma}=\left\{(\lambda_1,\cdots,\lambda_n)\in\mathbb{R}^n: \sum_{j=1}^n \lambda_j -\varrho\lambda_i >0, \, 1\leq i\leq n\right\}.$ 
By Proposition \ref{lemma1-lower-rho}, $\varrho_{\tilde{\Gamma}} 
=1+\frac{ n-1}{1-\varrho}=t$.

\end{proof}

\end{appendix}

\subsubsection*{Acknowledgements} 
{\small
The author  wishes to express his gratitude to Professor Yi Liu for 
answering certain questions related to the proof of Lemma \ref{lemma-diff-topologuy}.
The author also  wishes to   thank Ze Zhou for useful discussion on the homogeneity lemma. 
The author would like to thank Professor Robert
McOwen for sending his beautiful paper  \cite{McOwen1993}.
The author is  supported by  the National Natural Science Foundation of China  (Grant No. 11801587),   Guangdong Basic and Applied Basic Research Foundation (Grant No. 2023A1515012121) and 
Guangzhou Science and Technology Program  (Grant No. 202201010451).}


\bigskip


\end{document}